\RequirePackage[l2tabu, orthodox]{nag}
\documentclass[11pt]{amsart} 
\author{Andy Hammerlindl}
\title{Ergodic components of partially hyperbolic systems}

\usepackage{amsmath}
\usepackage{amssymb}
\usepackage{amsfonts}
\usepackage{amsthm}
\usepackage{amscd}
\usepackage{graphicx}
\usepackage{setspace}
\usepackage{cleveref}
\usepackage{fourier}
\usepackage[T1]{fontenc}

\makeatletter
\def\saveenum{\xdef\@savedenum{\the\c@enumi\relax}}
\def\resetenum{\global\c@enumi\@savedenum}
\makeatother

\newcommand{\graph}{\operatorname{graph}}
\newcommand{\Aut}{\operatorname{Aut}}

\newcommand{\bbR}{\mathbb{R}}
\newcommand{\bbQ}{\mathbb{Q}}
\newcommand{\bbZ}{\mathbb{Z}}
\newcommand{\bbN}{\mathbb{N}}
\newcommand{\bbT}{\mathbb{T}}

\newcommand{\bbS}{\mathbb{S}^1}
\newcommand{\orbi}{\bbS/\bbZ_2}

    \newcommand{\tSig}{\tilde{\Sig}}
    \newcommand{\bbZd}{\mathbb{Z}^d}
    \newcommand{\bbRd}{\mathbb{R}^d}
    
    \newcommand{\bbQd}{\mathbb{Q}^d}

    \newcommand{\kbbZ}{\tfrac{1}{k} \bbZ}

    \newcommand{\tilh}{\tilde{h}}

    \newcommand{\bt}{\beta}

\newcommand{\subof}{\subset}
\newcommand{\ti}{\times}

\newcommand{\Es}{E^s}
\newcommand{\Ec}{E^c}
\newcommand{\Eu}{E^u}

\newcommand{\Ws}{W^s}
\newcommand{\Wc}{W^c}
\newcommand{\Wu}{W^u}
\newcommand{\Wcu}{W^{cu}}
\newcommand{\Wcs}{W^{cs}}

\newcommand{\inv}{^{-1}}

\newcommand{\tX}{\tilde{X}}
\newcommand{\tY}{\tilde{Y}}
\newcommand{\tx}{\tilde{x}}
\newcommand{\ty}{\tilde{y}}
\newcommand{\tq}{\tilde{q}}
\newcommand{\tp}{\tilde{p}}
\newcommand{\tM}{\tilde{M}}
\newcommand{\hM}{\hat{M}}
\newcommand{\tS}{\tilde{S}}
\newcommand{\tJ}{\tilde{J}}
\newcommand{\tf}{\tilde{f}}
\newcommand{\tN}{\tilde{N}}

\newcommand{\ep}{\epsilon}
\newcommand{\lam}{\lambda}
\newcommand{\Lam}{\Lambda}
\newcommand{\Gam}{\Gamma}
\newcommand{\gam}{\gamma}
\newcommand{\Sig}{\Sigma}
\newcommand{\sig}{\sigma}

\newcommand{\Fix}{\operatorname{Fix}}
\newcommand{\Homeo}{\operatorname{Homeo}}
\newcommand{\HPG}{\Homeo^+(\Gam)}

\newcommand{\HPR}{\Homeo^+(\bbR)}

\newcommand{\qandq}{\quad \text{and} \quad}

\newcommand{\dist}{\operatorname{dist}}

\newcommand{\id}{\operatorname{id}}
\newcommand{\supp}{\operatorname{supp}}

\newcommand{\Aff}{\operatorname{Aff}}
\newcommand{\Trans}{\operatorname{Trans}}
\newcommand{\Isom}{\operatorname{Isom}}

\newcommand{\fAB}{f_{AB}}

\numberwithin{equation}{section}
\newtheorem{thm}[equation]{Theorem}
\newtheorem{cor}[equation]{Corollary}
\newtheorem{lemma}[equation]{Lemma}
\newtheorem{prop}[equation]{Proposition}
\newtheorem{question}[equation]{\textbf{Question}}
\newtheorem*{questionA}{\textbf{Question}}
\newtheorem*{questionB}{\textbf{Question}}
\newtheorem{conjecture}[equation]{\textbf{Conjecture}}
\newtheorem*{conjecture1}{\textbf{Conjecture 1}}
\newtheorem*{conjecture2}{\textbf{Conjecture 2}}
\newtheorem*{conjecture3}{\textbf{Conjecture 3}}
\newtheorem{assumption}[equation]{\textbf{Assumption}}

\theoremstyle{remark}
\newtheorem*{remark} {\textbf{Remark}}

\newtheorem*{notation} {\textbf{Notation}}

\providecommand{\acknowledgement}{{\noindent \textbf{Acknowledgements}}\quad}

\begin{document}

\maketitle

%

\begin{abstract}
    This paper gives a complete classification of the possible ergodic
    decompositions for certain open families of volume-preserving partially
    hyperbolic diffeomorphisms.
    These families include systems with compact center leaves and
    perturbations of Anosov flows under conditions
    on the dimensions of the invariant subbundles.
    The paper further shows that the non-open accessibility classes form a $C^1$
    lamination and gives results about the accessibility classes
    of non-volume-preserving systems.
\end{abstract}



\medskip

\textbf{Note: this document has been modified slightly from
earlier preprints. The numbering of sections was changed to match
the published version and an erratum has been added to the end.}

\section{Introduction}

Invariant measures are important objects in the study of
dynamical systems.
Often, these measures are ergodic, allowing a single orbit
to express the global behaviour of the system.
However, this is not always the case.
For instance, a Hamiltonian system always possesses a
smooth invariant measure, but a generic smooth Hamiltonian
yields level sets on which the dynamics are not ergodic \cite{mm1974}.
Any invariant measure may be expressed as a linear combination
of ergodic measures
and while such a decomposition always exists,
it is not, in general, tractable to find it.
For partially hyperbolic systems,
there is a natural candidate for the
ergodic decomposition given by the accessibility classes of the system.
This paper analyzes certain families of partially hyperbolic systems,
characterizing the possible accessibility classes and
showing that these coincide with the ergodic components of any smooth
invariant measure.

\medskip

By the classical work of Hopf,
the geodesic flow on a surface of negative curvature is ergodic
\cite{hopf1939}.
Further, by the work Anosov and Sinai, the flow is
\emph{stably ergodic} meaning that
all nearby flows are also ergodic \cite{Anosov,as1967}.
Based on these techniques, Grayson, Pugh, and Shub showed
that the time-one map
of this geodesic flow is also stably ergodic
as a diffeomorphism \cite{grayson1994stably}.
To prove this, they observed two important properties.
The first property is \emph{partial hyperbolicity}.
A diffeomorphism $f$ is partially hyperbolic if 
there is an invariant splitting
of the tangent bundle of the phase space $M$ into three subbundles
\[
    TM = \Eu \oplus \Ec \oplus \Es
\]
such that vectors in the unstable bundle $\Eu$
are expanded by the derivative $Tf$,
vectors in the stable bundle $\Es$ are contracted,
and these dominate any expansion and
contraction of vectors in the center bundle $\Ec$.
(Appendix \ref{sec-define} gives a precise definition.)
The second property is \emph{accessibility}.
For a point $x \in M$, the accessibility class $AC(x)$
is the set of all points that can be reached from $x$ by a 
concatenation of paths, each tangent to either $\Es$ or $\Eu$.
A system is called accessible if its phase space consists of a single
accessibility class.
For the geodesic flow, the phase space $M$ is the unit tangent bundle of the
surface, $\Ec$ is the direction of the flow, and $\Es$ and $\Eu$ are given by the
horocycles.
Grayson, Pugh, and Shub demonstrated that any diffeomorphism near the
time-one map of the flow 
is both partially hyperbolic and accessible and used this to prove
its ergodicity.
This breakthrough was followed by a number of papers demonstrating stable
ergodicity for specific cases of partially hyperbolic systems
(see the surveys \cite{RHRHU-survey,wilkinson2010conservative})
and lead Pugh and Shub to formulate the following conjecture
\cite{pugh2000stable}.

\begin{conjecture1}
    Ergodicity holds on an open and dense set of
    volume-preserving partially hyperbolic
    diffeomorphisms.
\end{conjecture1}
They further split this into two subconjectures.

\begin{conjecture2}
    Accessibility implies ergodicity.  \end{conjecture2}
\begin{conjecture3}
    Accessibility holds on an open and dense set of partially hyperbolic
    diffeomorphisms (volume-preserving or not).
\end{conjecture3}
The Pugh-Shub conjectures have been established in a number of settings.
In particular, they are true when the center bundle $\Ec$ is one-dimensional
\cite{RHRHU-accessibility}.
However, there are a number of partially hyperbolic systems which arise
naturally and which are not ergodic, leading to the following questions.

\begin{questionA}
    Is it possible to give an exact description of the set of non-ergodic\\
    partially hyperbolic diffeomorphisms?
\end{questionA}
\begin{questionB}
    For a non-ergodic partially hyperbolic diffeomorphism, do the \\ ergodic
    components coincide with the accessibility classes of the system?
\end{questionB}
This paper answers these questions in the affirmative under certain
assumptions on the system.
We first give one example as motivation before introducing more general
results.
Consider on the 3-torus
$\bbT^3 = \bbR^3 / \bbZ^3$ a diffeomorphism $f$ \linebreak{}
defined by
\[
    f(x,y,z)=(2x+y,x+y,z).
\]
The eigenvalues are $\lam < 1 < \lam \inv$ and $f$ is therefore partially 
hyperbolic.
Arguably, this is the simplest partially hyperbolic example one can
find.
It preserves Lebesgue measure but is not ergodic.
Further, there are several ways to construct nearby
diffeomorphisms which are also non-ergodic.
With a bit of thought, the following methods come to mind.

\begin{enumerate}
    \item Rotate $f$ slightly along the center direction, yielding a diffeomorphism
    \[        (x,y,z) \mapsto (2x+y,x+y,z+\theta)  \]
    for some small rational $\theta \in \bbR/\bbZ$.

    \item Compose $f$ with a map of the form $(x,y,z) \mapsto (\psi(x,y,z),z)$
    for some $\psi:\bbT^3 \to \bbT^2$.

    \item Perturb $f$ on a subset of 
    the form $\bbT^2 \times X$ where $X \subsetneq \bbS$.

    \item Conjugate $f$ with a diffeomorphism close to the identity.
      \end{enumerate}
The results of this paper imply that any non-ergodic diffeomorphism in a
neighbourhood of $f$ can be constructed by applying these four steps in this
order.

\section{Statement of Results} \label{sec-results} 


%

We again refer the reader to the appendix for a list of
definitions.

Suppose $A$ and $B$ are automorphisms of a
compact nilmanifold $N$
such that $A$ is \mbox{hyperbolic} and $A B=B A$.
Then, $A$ and $B$ define a diffeomorphism
\[
    \fAB : M_B \to M_B,\quad(v,t) \mapsto (Av, t)
\]
on the manifold
\[
        M_B = N \times \bbR / (v, t) \sim (Bv, t-1).
\]
Call $\fAB$ an \emph{AB-prototype}.

Note that every AB-prototype is an example of a volume-preserving, partially
hyperbolic, non-ergodic system.
Further, just like the linear example on $\bbT^3$ given above,
every AB-prototype may be perturbed to
produce nearby diffeomorphisms which are also non-ergodic.

To consider such perturbations,
we use the notion of leaf conjugacy as introduced in \cite{HPS}.
Two partially hyperbolic diffeomorphisms $f$ and $g$ are
\emph{leaf conjugate}
if there are invariant foliations $\Wc_f$ and $\Wc_g$ tangent to
$\Ec_f$ and $\Ec_g$ and a homeomorphism $h$ such that
for every leaf in $L$ in $\Wc_f$, $h(L)$ is a leaf of $\Wc_g$
and $h(f(L)) = g(h(L))$.

We now define a family of diffeomorphisms which will be the focus of the
paper.
A partially hyperbolic system $f:M \to M$ is an \emph{AB-system}
if it preserves an orientation of the center bundle $\Ec$
and is leaf conjugate to an AB-prototype.
%
%
%
%

In order to consider skew-products over infranilmanifolds
and systems which do not preserve an orientation of $\Ec$, we also consider the
following generalization.
A diffeomorphism $f_0$ is an \emph{infra-AB-system} if an iterate of $f_0$
lifts to an AB-system on a finite cover.
To the best of the author's knowledge,
this family of partially hyperbolic diffeomorphisms includes
every currently known example of a non-ergodic system
with one-dimen\-sional center.
Further, there are manifolds on which every
conservative partially hyperbolic diffeomorphism is an AB-system.

\begin{question}
    Suppose $f$ is a conservative, non-ergodic, partially hyperbolic
    $C^2$ diffeomorphism with one-dimensional center.
    Is $f$ necessarily an infra-AB-system?
\end{question}

Skew products with trivial bundles correspond to AB-systems where $B$ is
the identity map.  The suspensions of Anosov diffeomorphisms correspond to
the case $A=B$.  These are not the only cases, however.  For instance,
one could take hyperbolic automorphisms $A,B:\bbT^3 \to \bbT^3$ defined by the
commuting matrices

\[        \begin{pmatrix}
        3&2&1\\
        2&2&1\\
        1&1&1  \end{pmatrix}
    \qandq
        \begin{pmatrix}
        2&1&1\\
        1&2&0\\
        1&0&1  \end{pmatrix}
    .
\]

Throughout this paper, the letters
$A$ and $B$ will always refer to the maps associated to
the AB-system under study, and $N$ and $M_B$ will be the manifolds in the
definition.
In general, if $f:M \to M$ is an AB-system, $M$ need only be homeomorphic to
$M_B$, not diffeomorphic \cite{fj1978anosov,fg2012anosov}.

We show that
every conservative AB-system belongs to one of
three cases,
each with distinct dynamical and ergodic properties.

\begin{thm} \label{thm-consAB}
    Suppose $f:M \to M$ is a $C^2$ AB-system which preserves a smooth volume
    form.
    Then, one of the following occurs.
    \begin{enumerate}
        \itemsep0.5em 
        \item $f$ is accessible and stably ergodic.
        \item $\Eu$ and $\Es$ are jointly integrable and $f$ is topologically
        conjugate to
        $M_B \to M_B,$ $(v,t) \mapsto (Av, t + \theta)$
        for some\, $\theta$.
        Further, $f$ is (non-stably) ergodic if and only if\, $\theta$ defines an
        irrational rotation.
        \item There are $n \ge 1$, a $C^1$ surjection $p: M \to \bbS$, and
        a non-empty open set $U \subsetneq \bbS$ such that
        \begin{itemize}
            \item
            for every connected component $I$ of $U$,
            \,$p \inv(I)$
            is an $f^n$-invariant subset
            homeomorphic to $N \times I$ and the restriction of $f^n$ to
            this subset is accessible and ergodic, and
            \item
            for every $t \in \bbS \setminus U$,
            \,$p \inv(t)$
            is an $f^n$-invariant submanifold tangent to
            \mbox{$\Eu \oplus \Es$} and homeomorphic to $N$.
        \end{itemize}  \end{enumerate}  \end{thm}
Note that the first case can be thought of as a degenerate form of the third
case with $U = \bbS$.
Similarly, the second case with rational rotation
corresponds to $U = \varnothing$.

To give the ergodic decomposition of these systems, we decompose the
measure and show that each of the resulting measures is ergodic.
Suppose $\mu$ is a smooth measure on a manifold $M$ and $p:M \to \bbS$ is
continuous and surjective such that $p_*\mu = m$
where $m$ is Lebesgue measure on $\bbS = \bbR/\bbZ$.
The Rokhlin disintegration theorem \cite{rokhlin1962}
implies that $\mu$ can be written as
\[
    \mu = \int_{t \in \bbS} \mu_t\,dm(t)
\]
where the support of each $\mu_t$ is contained in $p \inv(t)$.
Moreover, this disintegration is essentially unique; if measures
$\{\nu_t\}_{t \in \bbS}$ give another
disintegration of $\mu$, then $\nu_t = \mu_t$ for $m$-a.e. $t \in \bbS$.
For an open interval $I \subset \bbS$
define
\[
    \mu_I := \frac{1}{m(I)} \int_I \mu_t\,dm(t).
\]
Note that $\mu_I$ is the normalized restriction
of $\mu$ to $p \inv(I)$. 
Then an open subset $U \subset \bbS$ yields a decomposition
\begin{equation} \label{decomp}
    \mu = \sum_I m(I)\,\mu_{I} + \int_{t \in \bbS \setminus U} \mu_t\,dm(t)
\end{equation}
where $\sum_I$ denotes summation over all of the connected components $I$ of
$U$.

\begin{thm} \label{thm-decomp}
    If $f:M \to M$ is a $C^2$ AB-system and $\mu$ is a smooth, invariant, non-ergodic measure
    with $\mu(M)=1$, then
    there are $n \ge 1$, a $C^1$ surjection $p: M \to \bbS$, and
    an open set $U \subsetneq \bbS$ such that $p_*\mu=m$ and
    \eqref{decomp} is the ergodic decomposition of\, $(f^n, \mu)$.
\end{thm}
If $f$ is in case (3) of \eqref{thm-consAB}, then the $n$, $p$, and $U$ can be
taken to be the same in both theorems.
If $f$ is in case (2) and non-ergodic, then $\theta$ is rational,
and the map $p$ can be defined by composing the topological conjugacy
from $M$ to $M_B$ with a projection from $M_B$ to $\bbS$.

As $f$ preserves $\mu$ and $p_*\mu = m$, it follows that
$p(f(x)) = p(x) + q$ for some rational $q \in \bbS$
and all $x$ with $p(x) \notin U$.
Because of this, one can derive
the ergodic decomposition of $(f,\mu)$ from \eqref{thm-decomp}.
Each component is either of the form
$\frac{1}{n} \sum_{j=1}^n \mu_{t+j q}$
or
$\frac{1}{n} \sum_{j=1}^n \mu_{I_{k,j}}$
where if $I_k = (a,b)$ then $I_{k,j} = (a+j q,b+j q)$.
In \eqref{thm-decomp}, the ergodic components of $(f^n, \mu)$ are mixing and,
in fact, have the Kolmogorov property \cite{BW-annals}.
The ergodic components of $f$ are mixing if and only if
\eqref{thm-decomp} holds with $n=1$.

Using the perturbation techniques of \cite{RHRHU-accessibility},
for any AB-prototype $\fAB$,
rational number $\theta = \frac{k}{n}$,
and open subset $U \subsetneq \bbS$
which satisfies $U + \theta = U$,
one can construct an example of a volume-preserving 
AB-system which satisfies \eqref{thm-decomp}
with the same $n$ and $U$.
In this sense, the classification
given by \eqref{thm-consAB} and \eqref{thm-decomp}
may be thought of as complete.
Versions of 
these theorems 
for infra-AB-systems are given in Section \ref{sec-infra}.


\medskip

Accessibility also has applications beyond the conservative setting.
For instance, Brin
showed that accessibility and a non-wandering
condition imply that the system is (topologically)
transitive \cite{brin1975trans}.
Therefore, we state a version of \eqref{thm-consAB}
which assumes only this non-wandering condition.
For a homeomorphism $f:M \to M$,
a \emph{wandering domain} is a non-empty open subset $U$
such that $U \cap f^n(U)$ is empty for all $n  \ge 1$.
Let $NW(f)$
be the non-wandering set, the set of all points
$x \in M$ which do not lie in a wandering domain.

\begin{thm} \label{thm-ABNW}
    Suppose $f:M \to M$ is an AB-system such that $NW(f)=M$.
    Then, one of the following occurs.
    \begin{enumerate}
        \item $f$ is accessible and transitive.
        \item $\Eu$ and $\Es$ are jointly integrable and $f$ is topologically
        conjugate to
        $M_B \to M_B,$ $(v,t) \mapsto (Av, t + \theta)$
        for some $\theta$.
        Further, $f$ is transitive if and only if $\theta$ defines an
        irrational rotation.
        \item There are $n \ge 1$, a continuous surjection $p: M \to \bbS$, and
        a non-empty open set $U \subsetneq \bbS$ such that
        \begin{itemize}
            \item
            for every connected component $I$ of $U$,
            $p \inv(I)$
            is an $f^n$-invariant subset
            homeomorphic to $N \times I$, and
            \item
            for every $t \in \bbS \setminus U$,
            $p \inv(t)$
            is an $f^n$-invariant submanifold tangent to
            $\Eu \oplus \Es$ and homeomorphic to $N$.
        \end{itemize}
        The restriction of $f^n$ to a subset
        $p \inv(t)$ or $p \inv(I)$ is
        transitive.
    \end{enumerate}  \end{thm}
The non-wandering assumption is used in only a few places in the proof and so
certian results may be stated without this assumption.
For a partially hyperbolic diffeomorphism with one-dimensional center,
a \emph{$us$-leaf} is a complete $C^1$ submanifold
tangent to $\Eu \oplus \Es$. 

\begin{thm} \label{thm-cmptleaf}
    Every non-accessible AB-system has a compact $us$-leaf.
\end{thm}
\begin{thm} \label{thm-ratAB}
    Suppose $f:M \to M$ is a non-accessible AB-system with at least one compact
    periodic $us$-leaf.
    Then, there are $n  \ge  1$, a continuous surjection $p:M \to \bbS$
    and an open subset $U \subset \bbS$ with the following properties.
    
    For $t \in \bbS \setminus U$, $p \inv(t)$ is an $f^n$-invariant compact $us$-leaf.
    Moreover, every $f$-periodic compact $us$-leaf is of this form.

    For every connected component $I$ of $U$, $p \inv(I)$
    is $f^n$-invariant, homeomorphic to $N \times I$ and,
    letting $g$ denote the restriction of $f^n$ to $p \inv(I)$,
    one of three cases occurs{:}
    \begin{enumerate}
        \item
        $g$ is accessible,
        \item
        there is an open set $V \subset p \inv(I)$ such that
        \[
            \overline{g(V)} \subset V, \quad
            \bigcup_{k \in \bbZ} g^k(V) = p \inv(I), \quad
            \bigcap_{k \in \bbZ} g^k(V) = \varnothing,
        \]
        and the boundary of $V$
        is a compact $us$-leaf, or
        \item
        there are no compact $us$-leaves in $p \inv(I)$,
        uncountably many non-compact $us$-leaves in $p \inv(I)$,
        and $\lam  \ne  1$ such that $g$ is semiconjugate to
        \[
            N \times \bbR \to N \times \bbR, \quad
            (v,t) \mapsto (Av, \lam t).
      \]  \end{enumerate}  \end{thm}
It is relatively easy to construct examples in the first two cases
above.
Section \ref{sec-incoherent}
gives an example of the third case.
It is based on the
discovery by
F.~Rodriguez Hertz, J.~Rodriguez Hertz, and
R.~Ures of a non-dynamically
coherent system on the 3-torus \cite{RHRHU-non-dyn}.
Theorem \eqref{thm-ratAB} corresponds to a rational rotation on an 
$f$-invariant circle.
The following two theorems correspond to irrational rotation.

\begin{thm} \label{thm-irratAB}
    Suppose $f:M \to M$ is a non-accessible AB-system with no periodic compact
    $us$-leaves.
    Then, there is a continuous surjection $p:M \to \bbS$ and a $C^1$
    diffeomorphism $r:\bbS \to \bbS$ such that
    \begin{itemize}
        \item
        $NW(f) = p \inv(NW(r))$,
        \item
        if $t \in NW(r)$ then $p \inv(t)$ is a compact $us$-leaf and
        $f(p \inv(t)) = p \inv(r(t))$, and
        \item
        if $I$ is a connected component of $\bbS \setminus NW(r)$,
        then $f(p \inv(I)) = p \inv(r(I))$.
        In particular, $p \inv(I) \subset M$ is a wandering domain.
    \end{itemize}  \end{thm}
\begin{thm} \label{thm-irratsemi}
    Suppose $f:M \to M$ is a non-accessible AB-system with no periodic compact
    $us$-leaves.
    Then, $f$ is semiconjugate to
    \[
        M_B \to M_B, \quad (v,t) \mapsto (Av, t + \theta)
    \]
    for $\theta$ defining an irrational rotation.
\end{thm}
One can construct $C^1$ examples of AB-systems satisfying the conditions of
\eqref{thm-irratAB} and with $NW(f)  \ne  M$.
For instance, if $r$ is a Denjoy diffeomorphism of the circle, simply consider
a direct product $A \times r$ where $A$ is Anosov.

The diffeomorphism $f$ in \eqref{thm-ABNW}--\eqref{thm-irratsemi} need only be
$C^1$ in general.  If $f$ is a $C^2$ diffeomorphism, then the
surjection $p:M \to \bbS$ may be
taken as $C^1$.
This is a consequence of the following regularity result,
proven in Section
\ref{sec-regularity}.

\begin{thm} \label{thm-regularity}
    For a non-accessible partially hyperbolic $C^2$ diffeomorphism with one-dimensional
    center, the $us$-leaves form a $C^1$ lamination.
\end{thm}
The existence of a $C^0$ lamination was shown in \cite{RHRHU-accessibility}.

%

\medskip

The next sections
discuss how this work relates to other results in partially
hyperbolic theory, first for three-dimensional systems in Section
\ref{sec-dimthree} and for higher dimensions in Section \ref{sec-higher}.
Section \ref{sec-outline} gives an outline of the proof and of the
organization of the rest of the paper.
The appendix gives precise definitions for many of the terms
used in these next few sections.

\section{Dimension three} \label{sec-dimthree} 

The study of partially hyperbolic systems has had its greatest success in
dimension three, where $\dim \Eu = \dim \Ec = \dim \Es = 1$.
Still, in this simplest of cases, a number of important questions remain open.
%
%
Hertz, Hertz, and Ures posed the following conjecture specifically
regarding ergodicity.

\begin{conjecture} \label{conj-torus}
    If a conservative partially hyperbolic diffeomorphism in dimension three
    is not ergodic, then there is a periodic 2-torus tangent to $\Eu \oplus \Es$.
\end{conjecture}
They also showed that the existence of such a torus would have strong
dynamical consequences.
We state this theorem as follows.

\begin{thm}
    [\cite{RHRHU-tori}]
    If a partially hyperbolic diffeomorphism on a three dimensional
    manifold $M$
    has a periodic 2-torus tangent to $\Eu \oplus \Es$,
    then $M$ has solvable fundamental group.
\end{thm}
In fact, the theorem may be stated in a much stronger form.
See \cite{RHRHU-tori} for details.

Work on classifying partially hyperbolic systems
has seen some success
in recent years, at least for 3-manifolds with ``small'' fundamental group.
This was made possible by the breakthrough results of Brin, Burago, and Ivanov
to rule out partially hyperbolic diffeomorphisms on the 3-sphere
and prove dynamical coherence
on the 3-torus \cite{BI,BBI2}.
Building on this work, the author and R.~Potrie gave a classification
up to leaf conjugacy of all partially hyperbolic systems on 3-manifolds with
solvable fundamental group.
Using the terminology of the current paper,
the conservative version of
this classification can be stated as follows.

\begin{thm}
    [\cite{HP2}] \label{thm-classthree}
    A conservative partially hyperbolic diffeomorphism on a 3-mani\-fold with
    solvable fundamental group is (up to finite iterates and finite covers)
    either
    \begin{enumerate}
        \item[(a)]
        an AB-system,

        \item[(b)]
        a skew-product with a non-trivial fiber bundle, or

        \item[(c)]
        a system leaf conjugate to an Anosov diffeomorphism.
    \end{enumerate}  \end{thm}
Further, the ergodic properties of each of these three cases have been
examined in detail.
Case (a) is the subject of the current paper.
Case (b) was studied in \cite{RHRHU-nil}, where it was first shown that
there are manifolds on which all partially hyperbolic systems
are accessible and ergodic.
Case (c) was studied in \cite{HamUres}, which showed that if such a system
is not ergodic then it is topologically conjugate to an Anosov
diffeomorphism (not just leaf
conjugate).
It is an open question if such a non-ergodic system can occur.
All of these results can be synthesized into the following statement,
similar in form to \eqref{thm-consAB}.

\begin{thm} \label{thm-consthree}
    Suppose $M$ is a 3-manifold with solvable fundamental group and
    $f:M \to M$ is a $C^2$ conservative partially hyperbolic system.
    Then, (up to finite iterates and finite covers) one of the following occurs.
    \begin{enumerate}
        \item $f$ is accessible and stably ergodic.
        \item $\Eu$ and $\Es$ are jointly integrable and $f$ is topologically
        conjugate either to a linear hyperbolic automorphism of\, $\bbT^3$ or to
        \[
            M_B \to M_B, \, (v,t) \mapsto (Av, t + \theta)
              \]
        where $A,B:\bbT^2 \to \bbT^2$ define an AB-prototype and $\theta \in \bbS$.
        \item There are $n \ge 1$, a $C^1$ surjection $p: M \to \bbS$, and
        a non-empty open set $U \subsetneq \bbS$ such that
        \begin{itemize}
            \item
            for every connected component $I$ of $U$,
            $p \inv(I)$
            is an $f^n$-invariant subset
            homeomorphic to $\bbT^2 \times I$ and the restriction of $f^n$ to
            this subset is accessible and ergodic,
            \item
            for every $t \in \bbS \setminus U$, 
            \ $p \inv(t)$
            is an $f^n$-invariant 2-torus tangent to
            $\Eu \oplus \Es$.
        \end{itemize}  \end{enumerate}  \end{thm}
If \eqref{conj-torus} is true, then this theorem encapsulates every
possible ergodic decomposition for a 3-dimensional partially hyperbolic
system.

\begin{question}
    Is the condition ``with solvable fundamental group''
    necessary in \eqref{thm-consthree}?
\end{question}
\section{Higher dimensions} \label{sec-higher} 

We next consider the case of skew products in higher dimension.
%
%
%
%
%
%
%
In related work, K.~Burns and A.~Wilkinson
studied
stable ergodicity of rotation extensions and of more general group extensions
over Anosov diffeomorphisms \cite{BW-skew},
and
M.~Field, I.~Melbourne, V.~Ni{\c{t}}ic{\u{a}}, and A.~T\"or\"ok
have analyzed group extensions over 
Axiom A systems,
proving results on transitivity, ergodicity, and rates of
mixing \cite{fmt2005,fmt2007,mnt2012}.

In this paper, we use the following definition
taken from \cite{gogolev2011compact}.
Let $\pi:M \to X$ define a fiber bundle on a compact manifold $M$
over a topological manifold $X$.
If a partially hyperbolic diffeomorphism $f:M \to M$ is such that the
center direction $\Ec_f$ is tangent to the fibers of the bundle and
there is a homeomorphism
$A:X \to X$ satisfying $\pi f = A \pi$,
then $f$ is a \emph{partially hyperbolic skew product}.
We call $A$ the \emph{base map} of the skew product.
While $f$ must be $C^1$, $\pi$ in general will only be continuous.

This definition has the benefit that it is open: any $C^1$-small
perturbation of a partially hyperbolic skew product is again
a partially hyperbolic skew product.
This can be proven using the results in \cite{HPS}
and the fact that the base map is
expansive.
The base map also has the property that it is topologically Anosov
\cite{aoki-hiraide-book}.  As with smooth Anosov systems, it is an open
question if all topologically Anosov systems are algebraic in nature.

\begin{question} \label{question-base}
    If $A$ is a base map of a partially hyperbolic skew product,
    then is $A$ topologically conjugate to a hyperbolic infranilmanifold
    automorphism?
\end{question}
We now consider the case where $\dim \Ec = 1$ in order to relate skew products
to the AB-systems studied in this paper.
The following is easily proved.

\begin{prop}
    Suppose $f$ is a partially hyperbolic skew product where the base map
    is a hyperbolic nilmanifold automorphism and $\Ec$ is
    one-dimensional and has an orientation preserved by $f$.
    Then, $f$ is an AB-system if and only if the fiber bundle defining the
    skew product is trivial.
\end{prop}
If we are interested in the ergodic properties of the system, we can further
relate accessibility to triviality of the fiber bundle.

\begin{thm} \label{thm-trivial}
    Suppose $f$ is a partially hyperbolic skew product where the base map
    is a hyperbolic nilmanifold automorphism and $\Ec$ is
    one-dimensional and orientable.
    If $f$ is not accessible, then the fiber bundle defining the
    skew product is trivial.
\end{thm}
\begin{cor}
    Suppose $f$ is a conservative $C^2$
    partially hyperbolic skew product where the base map
    is a hyperbolic nilmanifold automorphism and $\Ec$ is
    one-dimensional and has an orientation preserved by $f$.
    Then, $f$ satisfies one of the three cases of \eqref{thm-consAB}
    and if $f$ is not ergodic,
    its ergodic decomposition is given
    by \eqref{thm-decomp}.
\end{cor}
Theorem \eqref{thm-trivial} is proved in Section \ref{sec-products}.
A similar statement, \eqref{prop-infraskew},
still holds when ``nilmanifold'' is replaced by
``infranilmanifold'' and the condition on orientability is dropped.

Every partially hyperbolic skew product has compact center leaves and an open
question, attributed in \cite{RHRHU-survey} to C.~C.~Pugh,
asks if some form of
converse statement holds.

\begin{question} \label{question-cc}
    Is every partially hyperbolic diffeomorphism with compact center
    leaves finitely covered by a partially hyperbolic skew
    product?
\end{question}
This question was studied independently
by D.~Bohnet, P.~Carrasco, and A.~Go\-golev
who gave positive answers under certain assumptions
\cite{bohnet2011thesis, bohnet2013codimension, carrasco2010thesis, gogolev2011compact}.
In relation to the systems studied in the current paper,
the following results are relevant.

\begin{thm}
    [\cite{gogolev2011compact}]
    If $f$ is a partially hyperbolic diffeomorphism
    with compact center leaves, and
    $\dim \Ec = 1$, $\dim \Eu  \le  2$, and $\dim \Es  \le  2$,
    then $f$ is finitely covered by a skew product.
\end{thm}
\begin{cor}
    Suppose $f:M \to M$ is a partially hyperbolic diffeomorphism
    with compact center leaves,
    $\dim \Ec = 1$, and $\dim M = 4$.
    If $f$ is not accessible, then $f$ is an infra-AB-system.
\end{cor}
A compact foliation is \emph{uniformly compact}
if there is a uniform bound on the volume of the leaves.

\begin{thm}
    [\cite{bohnet2013codimension}]
    If $f$ is a partially hyperbolic diffeomorphism with
    uniformly compact center leaves
    and $\dim \Eu = 1$, then
    $f$ is finitely covered by a partially hyperbolic skew product
    where the base map is a hyperbolic toral automorphism.
\end{thm}
\begin{cor}
    Suppose $f$ is a partially hyperbolic diffeomorphism
    with uniformly compact center leaves and
    $\dim \Eu = \dim \Ec = 1$.
    If $f$ is not accessible, then $f$ is an infra-AB-system.
      \end{cor}
In the conservative setting, we may then invoke
the results of the current paper to describe the ergodic properties of
these systems.

\begin{question} \label{question-combo}
    If $f$ is a non-accessible partially hyperbolic diffeomorphism
    with compact one-dimensional center leaves, then is $f$ an
    infra-AB-system?
\end{question}
Positive answers to both \eqref{question-base} and \eqref{question-cc}
would give a positive answer to \eqref{question-combo}.

\medskip

In his study of hyperbolic flows, Anosov established a dichotomy, now known as
the ``Anosov alternative'' which states that every
transitive Anosov flow is either
topologically mixing or the suspension of an Anosov diffeomorphism with
constant roof function \cite{Anosov,fmt2007}.
Ergodic variants of the Anosov alternative have also been studied
and the following holds.

\begin{thm}
    [\cite{plante1972,bpw}]
    For an Anosov flow $\phi_t:M \to M$,
    the following are equivalent:
    \begin{itemize}
        \item the time-one map $\phi_1$ is not accessible,

        \item the strong stable and unstable foliations are jointly integrable,
    \end{itemize}
    and both imply
    the flow is topologically conjugate to
    the suspension of an Anosov diffeomorphism.
\end{thm}
\begin{cor}
    Suppose every Anosov diffeomorphism is topologically conjugate to an
    infranilmanifold automorphism.
    Then, every non-accessible time-one map of an Anosov flow is an
    infra-AB-system.
\end{cor}
Thus, if the conjecture about Anosov diffeomorphisms is true,
then the results given in Section \ref{sec-infra}
will classify the ergodic properties of diffeomorphisms
which are perturbations of time-one maps of Anosov flows.
This conjecture is true when the Anosov diffeomorphism has a one
dimensional stable or unstable bundle \cite{newhouse1970codimension}.

\begin{cor}
    Suppose $f$ is the time-one map of an Anosov flow
    with $\dim \Eu_f = 1$.
    If $f$ is not accessible, then it is an AB-system.
\end{cor}

\section{Outline} \label{sec-outline} 

Most of the remaining sections focus
on proving the results listed in Section
\ref{sec-results} and we present here an outline of the main ideas.

A partially hyperbolic system
has \emph{global product structure}
if it is dynamically coherent and,
after lifting the foliations to the
universal cover $\tM$,
the following hold for all $x,y \in \tilde M${:}
\begin{enumerate}
    \item $\Wu(x)$ and $\Wcs(y)$ intersect exactly once,

    \item $\Ws(x)$ and $\Wcu(y)$ intersect exactly once,

    \item if $x \in \Wcs(y)$, then $\Wc(x)$ and $\Ws(y)$ intersect exactly
    once, and

    \item if $x \in \Wcu(y)$, then $\Wc(x)$ and $\Wu(y)$ intersect exactly
    once.
\end{enumerate}
\begin{thm} \label{thm-ABGPS}
    Every AB-system has global product structure.
\end{thm}
This proof of this theorem is left to Section \ref{sec-openness}.
That section also proves the following.

\begin{thm} \label{thm-openAB}
    AB-systems form a $C^1$-open subset of the space of diffeomorphisms.
\end{thm}
Now assume $f$ is a non-accessible AB-system.
There is a lamination
consisting of $us$-leaves \cite{RHRHU-accessibility},
and this lamination lifts to the universal cover.
Global product structure implies that for a center leaf $L$ on the
cover, every leaf of the lifted $us$-lamination intersects $L$ exactly once.
Each deck transformation maps the lamination to itself and
this leads to an action of the fundamental group on a closed subset
of $L$ as depicted in Figure \ref{fig-drawing}.

In Section \ref{sec-actions}, we consider an order-preserving
action of a nilpotent group $G$ on a closed subset $\Gam \subset \bbR$.
We also assume there is $f$ acting on $\Gam$ such that $f G f \inv = G$.
Then, $f$ and $G$ generate a solvable group.
Solvable groups acting on the line were studied by Plante
\cite{planteSolv}.
By adapting his results, we prove \eqref{master-lemma} which
(omitting some details for now) states that either $\Fix(G)$ is non-empty
or, up to a common semiconjugacy from $\Gam$ to $\bbR$, each $g \in G$
gives a translation $x \mapsto x + \tau(g)$ and
$f$ gives a scaling $x \mapsto \lam x$.

Instead of applying this result immediately to AB-systems,
Section \ref{sec-AIsys} introduces the notion of an ``AI-system''
which can be thought of as the lift of an AB-system to a covering space
homeomorphic to $N \times \bbR$ where, as always, $N$ is a nilmanifold.
Using \eqref{master-lemma}, Section \ref{sec-AIsys}
gives a classification result,
\eqref{thm-mainAI},
for the accessibility classes of AI-systems.
Section \ref{sec-ABsys}
applies the results for AI-systems to give results about 
AB-systems and gives a proof of \eqref{thm-cmptleaf}.
The higher dimensional dynamics of the AB-system depend on the one-dimensional
dynamics on an invariant circle.
Sections \ref{sec-rat} and \ref{sec-irrat}
consider the cases of rational and irrational rotation respectively
and prove theorems \eqref{thm-ratAB}--\eqref{thm-irratsemi}.

Section \ref{sec-proving} gives the proofs of 
\eqref{thm-consAB}, \eqref{thm-decomp}, and \eqref{thm-ABNW}
based on the other results.
In order to establish the ergodic decomposition, the lamination of $us$-leaves
must be $C^1$.
By \eqref{thm-regularity}, this holds if the diffeomorphism is $C^2$.
The proof requires a highly technical
application of Whitney's extension theorem and is given in
Section \ref{sec-regularity}.
The specific version of this regularity result for AB-systems can be stated as
follows.

\begin{prop} \label{prop-pproj}
    Let $f:M \to M$ be a $C^2$ AB-system.
    Then, there is a $C^1$ surjection $p:M \to \bbS$
    and $U \subset \bbS$
    such that the compact $us$-leaves of $f$ are exactly the sets
    $p \inv(t)$ for
    $t \in \bbS \setminus U$.
    
    If $S$ is a center leaf which intersects each compact $us$-leaf
    exactly once, then $p$ may be defined so that its restriction to $S$ is a
    $C^1$-diffeomorphism.

    If $\mu$ is a probability measure given by a $C^1$ volume form on $M$,
    then $p$ may be chosen so that $p_* \mu$ is Lebesgue measure on
    $\bbS = \bbR / \bbZ$.
\end{prop}
Section \ref{sec-products} proves \eqref{thm-trivial} concerning the triviality
of non-accessible skew products.
Infra-AB-systems are treated in Section \ref{sec-infra}.
\begin{figure}[t]
{\centering
\includegraphics{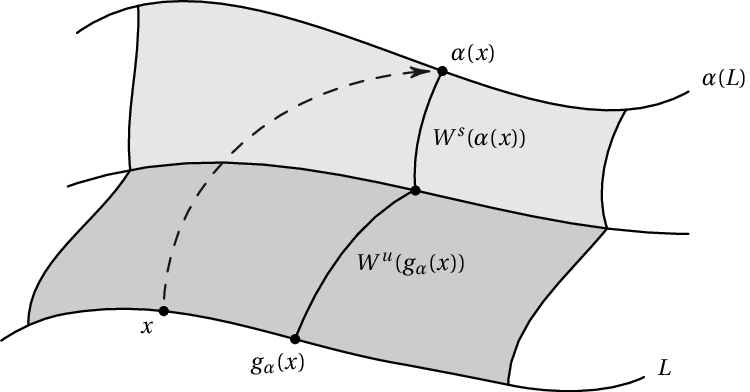}
}
\caption{After lifting to the universal cover, an AB-system has a
center leaf $L$ invariant under the lifted dynamics $f$.
Each deck transformation $\alpha$ then defines a function
$g_\alpha:L \to L$ where $g_\alpha(x)$ is the unique point for which
$\Ws(\alpha(x))$ intersects $\Wu(g_\alpha(x))$.
These functions together with $f$
define
a solvable action on a closed subset of $L$
and this action is semiconjugate to an affine action on $\bbR$.
}
\label{fig-drawing}
\end{figure}

\section{Actions on subsets of the line} \label{sec-actions} 

\begin{notation}
    To avoid excessive parentheses, if 
    $f$ and $g$ are composable functions,
    we simply write $f g$ for the composition.
    In this section, $\mu$ is a measure on the real line and
    $\mu[x,y)$ denotes the measure of the half-open interval $[x,y)$.
\end{notation}
Let $\HPR$ denote the group of orientation-preserving homeo\-morphisms of the
line.
If $\Gam$ is a non-empty closed subset of $\bbR$, let $\HPG$ denote
the group of all homeomorphisms of $\Gam$ which are restrictions of elements of
$\HPR$.
That is, $g$ is in $\HPG$ if it is a homeomorphism of $\Gam$ and
$g(x)<g(y)$ for $x<y$.

We now adapt results of Plante to this setting.

\pagebreak 

\begin{prop} \label{prop-mu}
    Suppose $\Gam$ is a non-empty closed subset of $\bbR$ and
    $G$ is a subgroup of $\HPG$ with non-exponential growth.
    Then, there is a measure $\mu$ on\, $\bbR$ such that
    \begin{itemize}
        \item
        $\supp \mu \subset \Gam$,
        \item
        $\mu(X) = \mu(g(X))$ for all $g \in G$ and Borel sets $X \subset \bbR$,
        and
        \item
        if $X \subset \bbR$ is compact, then $\mu(X) < \infty$.
    \end{itemize}  \end{prop}
\begin{proof}
    In the case $\Gam=\bbR$, this is a restatement of (1.3) in
    \cite{planteSolv}.
    One can check that the techniques in \cite{planteSolv}
    and \cite{plante1975foliations} extend immediately to the case
    $\Gam  \ne  \bbR$.
\end{proof}
\begin{prop} \label{prop-tau}
    Let $\Gam$, $G$, and $\mu$ be as in \eqref{prop-mu}
    and suppose $\Fix(G)$ is empty.
    Then there is a non-zero homomorphism $\tau:G \to \bbR$
    such that for all $x \in \bbR$
    \[
        \tau(g) = \left\{
        \begin{array}{lr}
        \mu [x,g(x)) & \text{if $x < g(x)$}, \\
        0 & \text{if $x = g(x)$}, \\
        -\mu [g(x),x) & \text{if $g(x) < x$}.
        \end{array}
        \right.
    \]  \end{prop}
\begin{proof}
    Choose any $x \in \bbR$ and define $\tau$ as above.
    One can then show that $\tau$ is a non-zero homomorphism and independent of
    the choice of $x$.
    See (5.3) of \cite{plante1975foliations} for details.
          \end{proof}
\begin{prop} \label{prop-lamF}
    Let $\Gam$, $G$, $\mu$, $\tau$ be as in \eqref{prop-tau}
    and suppose $f \in \HPR$ is such that $F:G \to G$ defined by
    $F(g)(x) = f g f \inv(x)$
    is a group automorphism.
    Then, there is $\lam > 0$ such that $\tau(F(g)) = \lam \tau(g)$
    for all $g \in G$.

    Moreover, if $\lam  \ne  1$, then $f_* \mu = \lam \mu$
    and any homeomorphism of\, $\bbR$ which commutes with $f$ has a fixed point.
\end{prop}
\begin{proof}
    The first half of the statement follows as an adaptation of
    \S 4 of \cite{planteSolv}.
    Further, if $\lam  \ne  1$, then $f_* \mu = \lam \mu$ by
    (4.2) of \cite{planteSolv}.
    To prove the final claim, we first show that if $\lam  \ne  1$ then
    $f$ has a fixed point.
    Consider $x \in \Gam$.
    As $\Fix(G)$ is empty by assumption, there is $g \in G$ such that
    $x < g(x)$.
    Then,
    \[
        \mu [x, +\infty)  \ge  \mu [x, g^k(x)) = k\,\tau(g)
    \]
    for all $k  \ge  1$.
    This shows that $\mu [x, +\infty) = \infty$ for any $x \in \bbR$.

    Assume, without loss of generality,
    that $\lam < 1$ and $x < f(x)$ for some $x \in \bbR$.
    Then,
    \[
        \mu [x, \sup_{k \ge 0} f^k(x) )
        = \sum_{k=0}^{\infty} \lam^k \mu [x, f(x))
        < \infty
    \]
    and therefore,
    $x_0 := \sup_{k \ge 0} f^k(x) < \infty$
    is a fixed point for $f$.
    If $h \in \HPR$ commutes with $f$ then for all $k \in \bbZ$
    \[
        \mu [x_0, h^k(x_0) )
        = \mu [f(x_0), fh^k(x_0) )
        = \lam\,\mu [x_0, h^k(x_0) )
    \]
    which is possible only if $\mu [x_0, h^k(x_0) ) = 0$.
    Then $\mu [x_0, \sup_{k \in \bbZ} h^k(x_0) ) = 0$
    and so $\sup h^k(x_0) < \infty$ is a fixed point for $h$.
\end{proof}
We now consider the case where $G$ is a fundamental group of a nilmanifold.

\begin{prop} \label{prop-nofixcoset}
    Let $G$ be a torsion-free, finitely-generated, nilpotent group
    and suppose
    $\phi \in \Aut(G)$ is such that $\phi(g)  \ne  g$ for all non-trivial
    $g \in G$.
    If $H$ is a $\phi$-invariant subgroup, then $\phi(gH)  \ne  gH$ 
    for all non-trivial cosets $gH  \ne  H$.
\end{prop}
\begin{proof}
    First, we show that the function $\psi:G \to G$
    defined by $\psi(g)= g \inv \phi(g)$
    is a bijection.
    If $G$ is abelian, then
    $G$ is isomorphic to $\bbZ^d$ for some $d$
    and $\psi$ is an invertible linear map, and hence bijective.
    Suppose now that $G$ is non-abelian and let $Z$ be its group-theoretic
    center.
    Pick some element
    $g_0 \in G$.
    As $G/Z$ is of smaller nilpotency class, by induction
    there is $g \in G$ such that
    $\psi(g Z) = g_0 Z$ or equivalently
    $\psi(g) z_0 = g_0$
    for some $z_0 \in Z$.
    As $\psi|_Z$ is an automorphism of $Z$, there is $z \in Z$ such that
    \begin{math}
        \psi(g z) = \psi(g) \psi(z) = \psi(g) z_0 = g_0.
      \end{math}
    As $h$ was arbitrary, this shows
    $\psi$ is onto.

    To prove injectivity,
    suppose $\psi(g)=\psi(g')$.
    By induction, $g' = g z$ for some $z \in Z$.
    Then,
    \[
        \psi(g)=\psi(g')=\psi(g)\psi(z)  \quad \Rightarrow \quad  \psi(z) = 1  \quad \Rightarrow \quad  z = 1  \quad \Rightarrow \quad  g'=g.
    \]
    If $H$ is a $\phi$-invariant subgroup, then
    $\psi(H)=H$ and the bijectivity
    of $\psi$ implies that $\psi(gH)  \ne  H$ for any non-trivial coset.
\end{proof}
The results of J.~Franks and A.~Manning
\cite{Franks1,Franks2,Manning}
show that for any Anosov diffeomorphism
on a nilmanifold, the resulting automorphism on the fundamental group
satisfies the hypotheses of \eqref{prop-nofixcoset}.

\medskip

\begin{lemma} \label{master-lemma}
    Suppose $\Gam \subset \bbR$, $G < \HPG$, and $f \in \HPR$ are such that
    \begin{itemize}
        \item
        $\Gam$ is closed and non-empty,
        \item
        $G$ is finitely generated and nilpotent,
        \item
        $F:G \to G$ defined by $F(g)(x) = f g f \inv(x)$ \\
        is a group automorphism
        with no non-trivial fixed points, and
        \item
        $\Fix(G)$ is empty.  \end{itemize}
    Then, there are
    \begin{itemize}
        \item
        a closed non-empty subset $\Gam_0 \subset \Gam$,
        \item
        a continuous surjection $P: \bbR \to \bbR$,
        \item
        a non-zero homeomorphism $\tau:G \to \bbR$, and
        \item
        $0 < \lam  \ne  1$  \end{itemize}
    such that
    for $x,y \in \bbR$ and $g \in G$
    \begin{itemize}
        \item
        $x \le y$ implies $P(x)  \le  P(y)$,
        \item
        $P g(x) = P(x) + \tau(g)$,
        \item
        $P f(x) = \lam P(x)$,
        \item
        $\Gam_0 = \{ x \in \Gam : g(x) = x \text{ for all } g \in \ker \tau\}$,
        and
        \item
        for each $t \in \bbR$,
        $P \inv(t)$ is either a point $z \in \Gam_0$
        or an interval $[a,b]$ with $a,b \in \Gam_0$.
    \end{itemize}
    Moreover, any homeomorphism which commutes with $f$ has a fixed point
    in $P \inv(0)$.
\end{lemma}


\begin{proof}
    The conditions on $G$ imply that it has non-exponential growth
    \cite{gromov1981groups}.
    Therefore, we are in the setting of the previous propositions.
    In particular, there are $\mu$, $\tau$, and $\lam$ as above.

    First, suppose that the image $\tau(G)$ is a cyclic subgroup
    of $\bbR$
    in order to derive a contradiction.
    In this case, the condition $\tau F = \lam \tau$ in
    \eqref{prop-lamF} implies that
    $\lam \tau(G) = \tau(G)$
    and therefore $\lam = 1$.
    Then, $F$ maps a coset of $\ker \tau$ to itself.
    As $\HPG$ is torsion free, so is $G$, and by \eqref{prop-nofixcoset},
    $F$ has a non-trivial fixed point, in contradiction to the hypotheses of
    the lemma being proved.
    Therefore, $\tau(G)$ is non-cyclic.

    Consequently, $\tau(G)$ is a dense subgroup of $\bbR$.
    Further $\lam  \ne  1$, as otherwise, one could derive a
    contradiction exactly as above.
    By \eqref{prop-lamF},
    $f$ has at least one fixed point, say $x_0 \in \bbR$.
    Define a function $P:\bbR \to \bbR$ by
    \[
        P(x) = \left\{
        \begin{array}{lr}
        \mu [x_0,x) & \text{if $x > x_0$}, \\
        0 & \text{if $x = x_0$}, \\
        -\mu [x,x_0) & \text{if $x < x_0$}.
        \end{array}
        \right.
    \]
    By definition, $P$ is (non-strictly) increasing.
    The density of $\tau(G)$ implies that $P(\bbR)$ is dense.
    Then, as a monotonic function without jumps, $P$ is continuous
    and therefore surjective.
    For each $t \in \bbR$,
    the pre-image $P \inv(t)$ is either a point or a closed interval,
    In either case, one can verify that $g(P \inv(t)) = P \inv(t)$
    for all $g \in \ker \tau$
    and therefore the boundary of $P \inv(t)$ is in $\Gam_0$.
    The other properties of $P$ listed in the lemma are easily verified.

    The statement for homeomorphisms commuting with $f$ follows
    by adapting the proof of \eqref{prop-lamF}.
\end{proof}
\section{AI-systems} \label{sec-AIsys} 

We now consider partially hyperbolic systems on non-compact manifolds.
Suppose $M$ is compact and $f:M \to M$ is partially hyperbolic.
Then, any lift of $f$ to a covering space of $M$ is also
considered to be partially
hyperbolic.
Also, any restriction of a partially hyperbolic diffeomorphism to an open
invariant subset is still considered to be partially hyperbolic.

Let $A$ be a hyperbolic automorphism of the 
compact nilmanifold $N$
and $I \subset \bbR$ an open interval.
The \emph{AI-prototype} is defined as
\[
    f_{AI} : N \times I \to N \times I,\quad (v,t) \to (Av, t).
\]
A partially hyperbolic diffeomorphism $f$ on a (non-compact)
manifold $\hM$
is an \emph{AI-system}
if it
has global product structure,
preserves the orientation of its center direction, and
is leaf conjugate to an AI-prototype. 


\begin{thm} \label{thm-mainAI}
    Suppose $f:\hM \to \hM$ is an AI-system
    with no invariant compact $us$-leaves.
    Then, either
    \begin{enumerate}
        \item
        $f$ is accessible,
        \item
        there is an open set $V \subset \hM$ such that
        \[
            \overline{f(V)} \subset V, \quad
            \bigcup_{k \in \bbZ} f^k(V) = \hM, \quad
            \bigcap_{k \in \bbZ} f^k(V) = \varnothing,
        \]
        and the boundary of $V$ is a compact $us$-leaf, or
        \item
        there are no compact $us$-leaves in $\hM$,
        uncountably many non-compact $us$-leaves in $\hM$
        and there is $\lam  \ne  1$ such that $f$ is semiconjugate to
        \[
            N \times \bbR \to N \times \bbR, \quad
            (v,t) \mapsto (Av, \lam t).
        \]  \end{enumerate}  \end{thm}

\begin{notation}
    For a point $x$ on a manifold supporting a partially hyperbolic system,
    let $\Ws(x)$ be the stable manifold through $x$,
    and $\Wu(x)$ the unstable manifold.
    Then $AC(x)$, the accessibility class of $x$,
    is the smallest set containing $x$ which satisfies
    \[    
        \Ws(y) \cup \Wu(y) \subset AC(x)
    \]
    for all $y \in AC(x)$.
    For an arbitrary subset $X$ of the manifold, define
    \[
        \Ws(X) = \bigcup_{x \in X} \Ws(x), \quad
        \Wu(X) = \bigcup_{x \in X} \Wu(x), \qandq
        AC(X) = \bigcup_{x \in X} AC(x).
    \]
    Note that $AC(X)$ may or may not be a single accessibility class.
\end{notation}

\begin{prop}
    [\cite{RHRHU-accessibility}] \label{prop-nonacc}
    Suppose $f$ is a partially hyperbolic system
    with one-dimen\-sional center
    on a (not necessarily compact) manifold $M$.
    For $x \in M$, the following are equivalent{:}
    \begin{itemize}
        \item
        $AC(x)$ is not open.
        \item
        $AC(x)$ has empty interior.
        \item
        $AC(x)$ is a complete $C^1$ codimension one submanifold.  \end{itemize}
    If $L$ is a curve through $x$ tangent to the center direction, then the
    following are also equivalent to the above:
    \begin{itemize}
        \item
        $AC(x) \cap L$ is not open in $L$.
        \item
        $AC(x) \cap L$ has empty interior in $L$.  \end{itemize}
    If $f$ is non-accessible, the set of non-open accessibility classes form a
    lamination.
\end{prop}

\begin{assumption}
    For the remainder of the section, assume $f:\hM \to \hM$ is a non-accessible
    AI-system.
\end{assumption}
All of the analysis of this section will be on the universal cover.
Let $\tM$ and $\tN$ be the universal covers of $M$ and $N$.
Then, $f$ and the leaf conjugacy $h$ lift to functions $f:\tM \to \tM$, and $h:\tM
\to \tN \times I$ still denoted by the same letters.
Every lifted center leaf of the lifted $f$ is of the form $h \inv(v \times I)$
for some $v \in \tN$.
In general, the choice of the lifts of $f$ and $h$ are not unique.
They may be chosen, however, so that $h f h \inv(v \times I) = Av \times I$
where
$A:\tN \to \tN$ is a hyperbolic Lie group automorphism.
As $A$ fixes the identity element of the Lie group,
there is a center leaf mapped to itself by $f$.  Let $L$ denote this leaf.
As $L$ is homeomorphic to $\bbR$, assume there is an ordering
on the points of $L$
and define
open intervals $(a,b) \subset L$ for $a,b \in L$ and suprema $\sup X$ for
subsets $X \subset L$ exactly as for $\bbR$.

Define a closed subset
\[
    \Lam = \{ t \in L : AC(t) \text{ is not open} \}.
\]
\begin{lemma}
    $\Lam$ is non-empty.
\end{lemma}
\begin{proof}
    As $\tM$ is connected, if all accessibility classes were open,
    $f$ would be accessible (both on $\tM$ and $\hM$).
    Therefore, there is at least one non-open accessibility class.
    By global product structure, this class intersects $L$.
\end{proof}
\begin{lemma} \label{lemma-ussu}
    If $t \in \Lam$, then
    \begin{math}
        AC(t) = \Ws \Wu(t) = \Wu \Ws(t).
    \end{math}  \end{lemma}
This is an adaptation to the case of global product structure of local
arguments used in the proof of \eqref{prop-nonacc}.

\begin{proof}
    Each center leaf in $\tM$ is of the form $h \inv(v \times I)$ for some
    $v \in \tN$.
    By global product structure, for each $v \in \tN$, there exist
    unique points $x_v,y_v,z_v,t_v \in \tM$
    \begin{figure}[t]
    {\centering
    \includegraphics{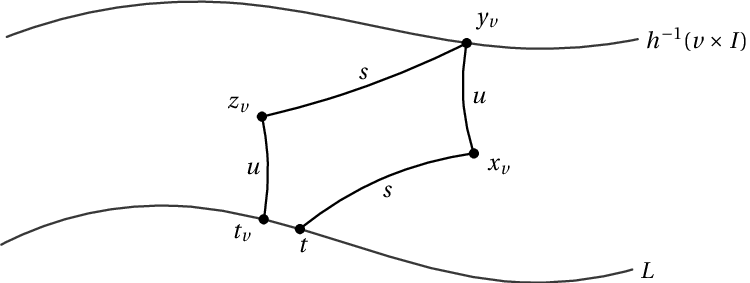}
    }
    \caption{A ``bracket'' of points defined by global product structure.
    The proof of \eqref{lemma-ussu} shows that if $t \in \Lam$, then $t_v = t$.
    }
    \label{fig-bracket}
    \end{figure}
    such that
    \[    
        x_v \in \Ws(t), \quad
        y_v \in \Wu(x_v) \cap h \inv(v \times I), \quad
        z_v \in \Ws(y_v), \quad
        t_v \in \Wu(z_v) \cap L.
    \]
    See Figure \ref{fig-bracket}.
    These points depend continuously on $v$.
    As $\tN$ is connected, the set 
    \[
        \{t_v : v \in \tN\} \subset L \cap AC(t)
    \]
    is connected and, by \eqref{prop-nonacc}, has empty interior as a subset
    of $L$.
    Therefore, it consists of the single point $t$.
    This shows that both
    $W^s W^u(t)$ and $W^u W^s(t)$
    intersect each center leaf $h \inv(v \times I)$ in the same unique
    point $y_v$
    and so the two sets are identical.
    This set is both $s$-saturated and $u$-saturated and so contains
    $AC(t)$.
\end{proof}
By global product structure,
for any $x \in \tM$, there is a unique point $R(x) \in L$
such that $\Wu(x)$ intersects $\Ws(R(x))$.
This defines a retraction, $R: \tM \to L$.
By the previous lemma,
if $t \in \Lam$,
then $R \inv(t) = AC(t)$.

Let $\alpha:\tM \to \tM$ be a deck transformation of the covering $\tM \to \hM$.
Then,
as depicted in Figure \ref{fig-drawing},
$\alpha$ defines a map $g_\alpha \in \Homeo^+(\Lam)$ given by the restriction
of $R \circ \alpha$ to $\Lam$.
Define
\[    G = \{ g_\alpha: \alpha \in \pi_1(\tM) \}.
\]
\begin{lemma} \label{lemma-Ggood}
    $G$ is a finitely generated, nilpotent subgroup of\, $\Homeo^+(\Lam)$.
\end{lemma}
\begin{proof}
    For $\alpha \in \pi_1(\tM)$ and $t \in \Lam$,
    $g_\alpha(t)$ is given by the
    unique intersection of $\alpha(AC(t))$ and $L$.
    Then,
    \[    
        AC(g_\alpha(g_\beta(t))) =
        \alpha(AC(g_\beta(t))) = 
        \alpha \beta(AC(t)) =
        AC(g_{\alpha \beta}(t))
    \]
    shows that $\pi_1(\hM) \to \Homeo^+(\Lam),\ \alpha \mapsto g_\alpha$
    is a group homomorphism.
    As $\hM$ is homotopy equivalent to the nilmanifold $N$,
    its fundamental group is finitely generated and nilpotent.
\end{proof}
It is necessary to define $G$ with elements in $\Homeo^+(\Lam)$
as, in general, the same
construction on $L$ will define a subset of $\Homeo^+(L)$
but not a subgroup.

\begin{lemma} \label{lemma-KfixAI}
    For a point $t \in \Lam$,
    $AC(t) \subset \tM$
    projects to a compact $us$-leaf in $\hM$
    if and only if
    $t \in \Fix(G)$.
\end{lemma}
\begin{proof}
    Consider $t \in \Lam$ and let $\hat X \subset \hM$ be the image of $AC(t)$
    by the covering $\tM \to \hM$.
    First, suppose $t \in \Fix(G)$.
    By global product structure, there is a unique map $\sigma:\tN \to AC(t)$
    such that $h \sigma(v) \in v \times I$ for every $v \in \tN$.
    For any deck transformation $\alpha \in \pi_1(\hM)$,
    \[
        \alpha(AC(t))=AC(g_\alpha(t))=AC(t)
    \]
    which implies that $\alpha \sigma = \sigma \alpha_N$
    where $\alpha_N$ is the corresponding deck transformation for the
    covering $\tN \to N$.
    It follows that
    $\sigma$ quotients to a homeomorphism from the compact nilmanifold
    $N$ to $\hat X$
    and therefore $\hat X$ is compact.

    To prove the converse, suppose $\hat X$ is compact.
    From the definition of an AI-system, one can see that
    every center leaf on $\hM$ is properly embedded.
    Therefore, $\hat X$ intersects each center leaf in a compact set.
    If $\tilde X$ is the pre-image of $\hat X$ by
    covering $\tM \to \hM$,
    then $\tilde X$ intersects each center leaf on $\tM$ in a compact set.
    In particular, $\tilde X \cap L$ is compact.
    Note that $\tilde X \cap L$ is exactly equal to the orbit
    $G t = \{ g(t) : g \in G\}$.
    Define $s = \sup G t$.
    Then, $s \in G t$ by compactness and
    $g(G t) = G t$ implies $g(s)=s$ for each
    $g \in G$.  This shows that $\{s\} = G s = G t$ and therefore
    $t = s \in \Fix(G)$.
\end{proof}
\begin{lemma} \label{lemma-AIinAI}
    Suppose $J \subset L$ is an open interval such
    that $\partial J \subset \Fix(f) \cap \Fix(G)$.
    Let $X$ be the image of $AC(J)$ by the covering $\tM \to \hM$.
    Then,
    $f|_X$ is an AI-system.
\end{lemma}
This lemma is the justification for assuming there are no invariant, compact
leaves in \eqref{thm-mainAI}.  If such leaves exist, the AI-system can be
decomposed into smaller systems.

\begin{proof}
    Assume the subinterval $J$ in the hypothesis is of the form $J=(a,b)$
    with $a,b \in L$.
    Unbounded subintervals of the form $(a,+\infty)$ and $(-\infty,b)$ are
    handled similarly.

    For every center leaf $h \inv(v \times I)$,
    let $a_v, b_v \in I$ be such that
    $v \times a_v \in h(AC(a))$
    and $v \times b_v \in h(AC(b))$.
    The set $\tX = \bigcup_{v \in \tN} h \inv(v \times (a_v, b_v) )$
    is $s$-saturated, $u$-saturated, and contains $J$.
    Therefore, $AC(J) \subset \tX$.
    By global product structure, one can show that $\tX \subset AC(J)$, so the
    two sets are equal.
    By its construction $\tX$ is simply connected, and invariant under deck
    transformations.  Therefore, it is the universal cover for $X$.
    Global product structure is inherited from $\tM$.
    For instance, for $x,y \in AC(J)$, there is a unique point $z \in \tM$
    such that $z \in \Ws(x) \cap \Wcu(y)$.
    Since, $\Ws(x) \subset \tX$, $z$ is in $\tX$.

    Compose $h$ with a homeomorphism which maps each $v \times (a_v,b_v)$
    to $v \times (0,1)$ by rescaling the second coordinate.
    This results in a leaf conjugacy between $f$ on $\tX$ and $A \times \id$
    on $\tN \times (0,1)$ which quotients down to a leaf conjugacy from $X$
    to $N \times (0,1)$.
\end{proof}
We now show that if the AI-system has no fixed compact $us$-leaves,
then 
it satisfies either case (2) or case (3) of \eqref{thm-mainAI}
depending on whether
it has any (non-fixed) compact $us$-leaves.

\begin{lemma}
    If\, $\Fix(G)$ is non-empty and $\Fix(f) \cap \Fix(G)$ is empty,
    then $f$ satisfies case~(2) of \eqref{thm-mainAI}.
\end{lemma}
\begin{proof}
    We first show that $f$ restricted to $L$ is fixed-point free.
    Suppose, instead, that $f(t)=t \in L$.
    By assumption $t \notin \Fix(G)$, so let
    $J$ be the connected component of $L \setminus \Fix(G)$
    containing $t$.
    As $\Fix(G)$ is $f$-invariant, $f(J)=J$ and each $s \in \partial J$
    is then an element of $\Fix(f) \cap \Fix(G)$, a contradiction.

    Without loss of generality, assume $t < f(t)$ for all $t \in L$.
    Choose some $t_0 \in \Fix(G)$ and define $L^+ = \{ t \in L : t > t_0 \}$.
    Then,
    \[
        \overline{f(L^+)} \subset L^+, \quad
        \bigcup_{k \in \bbZ} f^k(L^+) = L, \qandq
        \bigcap_{k \in \bbZ} f^k(L^+) = \varnothing.
   \]
    One can then show that the covering $\tM \to \hM$
    takes $AC(L^+)$ to an open set
    $V \subset \hM$ 
    which satisfies
    the second case of \eqref{thm-mainAI}.
\end{proof}
\begin{lemma} \label{lemma-mainAIthree}
    If\, $\Fix(G)$ is empty,
    then $f$ satisfies case (3) of \eqref{thm-mainAI}.
\end{lemma}
\begin{proof}
    In this case,
    the hypotheses of \eqref{master-lemma} hold with $\Gam = \Lam$.
    Let $P: L \to \bbR$ and $\tau:G \to \bbR$
    be as in \eqref{master-lemma}.

    If $\alpha \in \pi_1(\hM)$ is a deck transformation $\tM \to \tM$, then
    $h \alpha h \inv$ is equal to $\alpha_N \times \id$ on $\tN \times I$
    for some deck transformation $\alpha_N \in \pi_1(N)$.
    As $N$ is a nilmanifold,
    any homomorphism from $\pi_1(N)$ to $\bbR$ defines a unique homomorphism
    from the nilpotent Lie group $\tN$ to $\bbR$ \cite{malcev}.
    This implies that there is a unique Lie group homomorphism $T:\tN \to \bbR$
    such that
    \begin{math}
        T \alpha_N(v) = T(v) + \tau(g_\alpha)
    \end{math}
    for all $v \in \tN$ and $\alpha \in \pi_1(\hM)$.

    Let $R:\tM \to L$ be the retraction defined earlier in this section
    and
    let $H:\tM \to \tN$ be the composition of
    the leaf conjugacy $h:\tM \to \tN \times I$
    with projection onto the first coordinate.
    Define
    \[   
        Q : \tM \to \bbR,\quad
        x \mapsto P R(x) - T H(x).
    \]
    We will show that $Q$ quotients to a function $\hM \to \bbR$
    and use this to construct
    the semiconjugacy in the last case of \eqref{thm-mainAI}.

    First, consider a point $x \in \tM$
    which has a non-open accessibility class.
    Then, $R(x) \in \Lam$ and,
    for $\alpha \in \pi_1(\hM)$,
    \[
        P R(\alpha(x)) = P g_\alpha R(x) = P R(x) + \tau(g_\alpha)
    \]
    and
    \[
        T H \alpha(x) = T \alpha_N H(x) = T H(x) + \tau(g_\alpha)
    \]
    which together show $Q \alpha(x) = Q(x)$.

    Now, consider a point $x \in \tM$ which has an open accessibility class,
    and let $J \subset \tM$ be the connected component of
    $\Wc(x) \cap AC(x)$ which contains $x$.
    The set $\Gam_0$ from \eqref{master-lemma} is a subset of $\Gam=\Lam$
    and therefore $P$ is constant on $L \setminus \Lam$.
    Then, $P R$ is constant on $J$ and, by continuity,
    constant on the closure of $J$ as well.
    As 
    $H$ is constant on center leaves,
    $Q = P R - T H$ is also constant on the closure of $J$.
    Let $y$ be a point on the boundary of $J$.  Then, as $AC(y)$ is non-open,
    $Q(x) = Q(y) = Q \alpha(y)=Q \alpha(x)$.
    This shows that $Q$ quotients down to a function
    $\hat Q:\hM \to \bbR$.
    A much simpler argument shows that $H:\tM \to \tN$ quotients down to a
    function $\hat H:\hM \to N$.

    The properties of $F$ and $P$ in \eqref{master-lemma} imply that
    $T A = \lam T$
    and therefore
    $T H f = T A H = \lam T H$.
    As
    $P R f = P f R = \lam P R$,
    this shows that $Q f = \lam Q$.
    Then,
    $\hat H \times \hat Q$ is the desired semiconjugacy in \eqref{thm-mainAI}.
    By \eqref{master-lemma}, $P(\Lam)=\bbR$ and so $\Lam$ is uncountable.
    Each $G$-orbit of $\Lam$ corresponds to a distinct $us$-leaf,
    and so there are uncountably many.
\end{proof}
This concludes the proof of \eqref{thm-mainAI}.
We note one additional fact which will be used in the next section.

\begin{cor} \label{cor-AIcomfix}
    If\, $\Fix(G)$ is empty,
    any homeomorphism of $L$ which commutes with $f$ has a fixed point.
\end{cor}
\begin{proof}
    This follows from the use of \eqref{master-lemma} in the previous proof.
\end{proof}

\section{AB-systems} \label{sec-ABsys} 

\begin{assumption}
    In this section, assume $f:M \to M$ is a non-accessible AB-system.
\end{assumption}
The AB-prototype $f_{AB}$ has an invariant center leaf which is a circle.
By the leaf conjugacy,
$f$ also has an invariant center leaf.
Call this leaf $S$.
Note that $f$ lifts to an AI-system.
This is because the AB-prototype $f_{AB}$ lifts to the AI-prototype
$A \times \id$
on $N \times \bbR$.
If $h:M \to M_B$ is the leaf conjugacy,
then $h f h \inv$ is homotopic to $f_{AB}$
and therefore also lifts to $N \times \bbR$.

Let $\pi:\tM \to M$ be the universal covering, and choose a lift $\tf:\tM \to \tM$
and $\tS$ a connected component of $\pi \inv (S)$ such that $\tf(\tS)=\tS$.
The universal cover $\tN \times \bbR$ of the manifold $M_B$ has a deck
transformation of the form $(v,t) \mapsto (B v, t-1)$.
Conjugating this by the lifted leaf conjugacy gives a deck
transformation $\beta:\tM \to \tM$ and one can assume that $\beta(\tS) = \tS$.
Then, $\tS$ plays the role of $L$ in the previous section.
Define
$\Lam = \{ t \in \tS : AC(t) \text{ is not open} \}$
and
$G$ as a subgroup of $\Homeo^+(\Lam)$
as in the previous section.

\begin{lemma}
    $\Fix(G)$ is non-empty.
\end{lemma}
\begin{proof}
    This follows from \eqref{cor-AIcomfix}
    since $\beta$ and $\tf$ are commuting
    diffeomorphisms when restricted to $\tS$
    and $\beta$ is fixed-point free.
\end{proof}
\begin{lemma} \label{lemma-compactFix}
    For $t \in \Lam$, $AC(\pi(t)) \subset M$ is compact if and only if
    $t \in \Fix(G)$.
\end{lemma}
\begin{proof}
    If $t \in \Fix(G)$, then, by \eqref{lemma-KfixAI},
    $AC(\pi(t))$ is covered by a compact $us$-leaf
    of the AI-system and is therefore compact
    itself.

    Conversely, suppose $t \in \Lam$ is such that $AC(\pi(t)) \subset M$ 
    is a compact $us$-leaf.
    Note that as $\beta(\Fix(G)) = \Fix(G)$ there are
    $a,b \in \Fix(G)$ such that $a < t < b$ in the ordering on $\tS$.
    Then, $Gt$ is contained in $(a,b)$, a bounded subset of $\tS$.
    Considering the supremum as in \eqref{lemma-KfixAI},
    one shows that $s := \sup G t$ is in $\Fix(G)$.
    Consequently,
    $AC(\pi(t))$ accumulates on $\pi(s)$
    which, as $AC(\pi(t))$ is compact,
    implies $\pi(s) \in AC(\pi(t))$
    and so there is a deck transformation $\alpha:\tM \to \tM$
    such that $\alpha(s) \in AC(t)$.
    This implies there is $k \in \bbZ$ and $g \in G$ such that
    $t = \beta^k g(s) = \beta^k(s) \in \Fix(G)$.
\end{proof}
In this, and the next two sections, define
\[
    K = \{ x \in S: AC(x) \subset M \text{ is compact} \}.
\]
The last lemma shows that 
$K = \pi(\Fix(G))$.

\begin{cor} \label{cor-Kclosed}
    $K$ is closed and non-empty.\qed
\end{cor}
This also completes the proof of \eqref{thm-cmptleaf}.

\begin{cor} \label{cor-KNW}
    $K \cap NW(f|_S)$ is non-empty.
\end{cor}
\begin{proof}
    $K$ is non-empty, $f$-invariant, and closed.
\end{proof}
\begin{cor} \label{cor-perper}
    $f$ has a compact periodic $us$-leaf if and only if
    $f|_S$ has rational rotation number.
\end{cor}
\begin{proof}
    As a consequence of \eqref{lemma-compactFix},
    any compact $us$-leaf $X$ in $M$ intersects $S$ in a unique point
    $t$.
    If $f^n(X)=X$ then $f^n(t)=t$ and $f|_S$ has rational rotation
    number.
    If, conversely, $f|_S$ has rational rotation number,
    its non-wandering set
    consists of periodic points, and a compact periodic leaf exists
    by \eqref{cor-KNW}.
\end{proof}
The following is also from the last proof.

\begin{cor} \label{cor-pern}
    All compact periodic $us$-leaves have the same period.\qed
\end{cor}
\begin{lemma} \label{lemma-KSconj}
    If $K = S$, then $f$ on $M$ is topologically conjugate
    to a function
    \begin{math}
        (v,x) \mapsto (Av, \tilde r(x))
    \end{math}
    defined on the manifold
    \[
        M_B = N \times \bbR / (Bv, t) \sim (v, t+1)
    \]
    where $\tilde r:\bbR \to \bbR$
    is a lift of a homeomorphism $r : \bbR/\bbZ \to \bbR/\bbZ$
    topologically conjugate to $f|_S$.
\end{lemma}
\begin{proof}
    Let $\phi:\tS \to \bbR$ be
    any homeomorphism such that $\phi \beta(t) = \phi(t)+1$ for all $t$.
    Define $\tilde r$ as $\phi \tf \phi \inv$.
    Extend $\phi$ to all of $\tM$ by making it constant on accessibility
    classes.
    As in the proof of \eqref{lemma-mainAIthree},
    let $H:\tM \to \tN$ be the first coordinate
    of the lifted leaf conjugacy $h:\tM \to \tN \times \bbR$.
    Then, the function $H \times \phi:\tM \to \tN \times \bbR$
    gives a topological conjugacy between
    $\tf$ on $\tM$ and $A \times \tilde r$.

    The fundamental group of $M_B$ is generated by deck transformations of the
    form $(v,t) \mapsto (\alpha_N(v),t)$ or $(v,t) \mapsto (B v, t-1)$.
    Using the fact that $\Fix(G) = \tS$ and the definition of $\tilde r$,
    one can then show that $H \times \phi$
    quotients down to a topological conjugacy defined from $M$ to $M_B$.
      \end{proof}
\begin{lemma} \label{lemma-AIinAB}
    Suppose $J \subset S$ is an open interval such
    that $\partial J \subset \Fix(f) \cap K$.
    Then, $f|_{AC(J)}$ is an AI-system.
\end{lemma}
\begin{proof}
    Let $\tJ$ be a lift of $J$ to $\tS$.
    Then, as $f(J)=J$, $\tf(\tJ) = \beta^k(\tJ)$
    for some $k \in \bbZ$.
    By replacing the lift $\tf$ by $\tf \beta^k$,
    assume, without loss of generality that $\tf(\tJ)=\tJ$.
    As $K=\pi(\Fix(G))$, $\partial \tJ \subset \Fix(\tf) \cap \Fix(G)$,
    and so by \eqref{lemma-AIinAI},
    $AC(\tJ)$ projects to $X$ on $\hM$
    such that the dynamics on $X$ is an AI-system.
    As $\tJ$ is contained in a fundamental domain of the covering $\tS \to S$,
    one can show that $X$ is contained in a fundamental domain of the covering
    $\hM \to M$.
    Therefore, the dynamics on $\pi(AC(\tJ)) = AC(J)$ is an AI-system.
      \end{proof}
We now give a $C^0$ version of \eqref{prop-pproj}.

\begin{lemma} \label{lemma-ctsproj}
    There is a continuous surjection $p:M \to \bbS$
    such that $p|_S$ is a homeomorphism,
    $p|_{\Wc(x)}$ is a covering for any center leaf $\Wc(x)$ ($x \in M$)
    and $p$ is constant on each compact accessibility class.
\end{lemma}
\begin{proof}
    Define $p$ on $S$ so that $p|_S$ maps $S$ to $\bbS$ with constant speed
    along $S$.
    Extend $p$ to $AC(K) \cup S$ by making $p$ constant on accessibility
    classes.
    Then, for any center leaf $\Wc(x)$, let $J$ be a connected component
    of $\Wc(x) \setminus AC(K)$
    and define $p$ on $J$ so that $J$ is mapped at constant speed to $\bbS$
    and extends continuously to the boundary $\partial J \subset AC(K)$.
    Transversality of the center foliation and $us$-lamination
    implies that $p$ is continuous.
    The other properties are easily
    verified.
\end{proof}
Compare this short $C^0$ proof
to the $C^1$ proof in Section \ref{sec-regularity}.

We now consider the cases of rational and irrational rotation
of $f|_S$ separately in the next two sections.

\section{Rational rotation} \label{sec-rat} 

This section proves \eqref{thm-ratAB}.

\begin{assumption}
    Assume $f$ is a non-accessible AB-system with at least one periodic
    compact $us$-leaf.
\end{assumption}
Let $S$, $K$, and other objects be defined as in Section \ref{sec-ABsys}.
By \eqref{cor-pern}, all compact periodic leaves have the same period.
Call this period $n$.
Define $K_n = K \cap \Fix(f^n) \subset S$.
By \eqref{cor-Kclosed}, $K_n$ is closed.
Let $p:M \to \bbS$ be the projection given by \eqref{lemma-ctsproj} and
define $U \subset \bbS$ as $U = \bbS \setminus p(K_n)$.

Note that if $t \notin U$, then $p \inv(t)$ is an $f^n$-invariant compact
$us$-leaf.  Moreover, every such leaf is of this form.
This proves the first part of \eqref{thm-ratAB}.

To prove the rest of the theorem, replace $f$ by its iterate $f^n$
and assume $n = 1$.
The new $f$ is still an AB-system, albeit with a different ``\!$A$'' than before.
Now $K_n = \Fix(f) \cap K \subset S$.
If $I$ is a connected component of $U \subset \bbS$, then
$p \inv(I) \cap S$ is a connected component of $S \setminus K_1$
and \eqref{lemma-AIinAB} implies that $f$ restricted to
$p \inv(I) = AC(\pi(J))$ is an AI-system.
Since $J \cap K_n$ is empty, $AC(J)$ contains no invariant compact $us$-leaves.
Therefore, the AI-system falls into one of the cases given in
\eqref{thm-mainAI}.
As these cases correspond exactly to those given in \eqref{thm-ratAB},
this concludes the proof.

\section{Irrational rotation} \label{sec-irrat} 

This section proves \eqref{thm-irratAB} and \eqref{thm-irratsemi}.

\begin{assumption}
    Assume $f$ is a non-accessible AB-system with no periodic
    compact $us$-leaves.
\end{assumption}
Let $S$, $K$ and other objects be defined as in Section \ref{sec-ABsys}.
By \eqref{cor-perper}, $f|_S$ has irrational rotation number.

\begin{lemma}
    $NW(f|_S) \subset K$.
\end{lemma}
\begin{proof}
    For any $C^1$ circle diffeomorphism with irrational rotation,
    the non-wander\-ing set is minimal.
    The result then follows from \eqref{cor-KNW}.
\end{proof}
\begin{lemma}
    If $I$ is a connected component of $S \setminus NW(f|_S)$,
    then $AC(I)$ is a wandering domain.
    That is, the sets $f^k(AC(I)) = AC(f^k(I))$
    are pairwise disjoint for all $k \in \bbZ$.
\end{lemma}
\begin{proof}
    Let $J$ be the closure of $I$.
    Note that any compact leaf in $AC(J)$ must be of the form $AC(t)$ for some
    $t \in J$.
    By the properties of circle diffeomorphisms, the sets
    $f^k(J)$ are pairwise disjoint.
    By the last lemma, $\partial J \subset K$.
    If $AC(J)$ intersects $AC(f^k(J))$, then
    this intersection has a boundary consisting of compact $us$-leaves.
    Such a compact leaf would intersect $S$
    in a point $t \in J \cap f^k(J)$,
    a contradiction.
\end{proof}
\begin{lemma}
    $NW(f) = AC(NW(f|_S)).$
\end{lemma}
\begin{proof}
    The last lemma shows $NW(f) \subset AC(NW(f|_S))$.

    To prove the other inclusion,
    suppose $t \in NW(f|_S)$, $x \in AC(t)$ and $V \subset M$
    is a neighbourhood of $x$.
    There is a sequence $\{n_k\}$ such that $f^{n_k}(t)$ converges to $t$.
    By taking a further subsequence, assume $f^{n_k}(x)$ converges to
    some point $y \in AC(t)$.
    Let $D \subset V$ be a small unstable plaque containing $x$.
    Then $f^{n_k}(D)$ is a sequence of ever larger unstable plaques, and
    \[
        W^u(y) \subset \overline{\bigcup_k f^{n_k}(D)}.
    \]
    Unstable leaves of the Anosov diffeomorphism $A$ are dense in $N$
    \cite{Franks1}.
    Therefore, by the leaf conjugacy, $W^u(y)$ is dense in $AC(t)$.
    This shows that some iterate $f^{n_k}(V)$ intersects $V$.
\end{proof}
Now, let $p:M \to \bbS$ be as in \eqref{lemma-ctsproj}.
We may assume $p|_S$ is a $C^1$-diffeomorphism.
Define $r:\bbS \to \bbS$ by $r p(t)=p f(t)$ for all $t \in S$.
Then, \eqref{thm-irratAB} can be proved from the above lemmas.
As $r$ has irrational rotation number, it is semiconjugate to a rigid
rotation $t \mapsto t + \theta$.  Using this and the leaf conjugacy,
one can prove \eqref{thm-irratsemi} using an argument similar to the proof of
\eqref{lemma-KSconj}.

\section{Proving theorems 
\eqref{thm-consAB}, \eqref{thm-decomp}, and \eqref{thm-ABNW}} \label{sec-proving}

This section gives the proofs of several of the theorems stated in Section
\ref{sec-results}
based on results proved in other sections.

The proof of \eqref{thm-ABNW} makes use of a result of Brin regarding
transitivity \cite{brin1975trans}.
The following is an extension of this result to the non-compact case,
though the proof is in essence the same.

\begin{prop}
    [Brin] \label{prop-brintrans}
    Suppose $f$ is a partially hyperbolic diffeomorphism
    of a (not necessarily compact)
    manifold $M$.
    If $V$ is open and $f(V)=V \subset NW(f)$,
    then $\overline V = AC(\overline V)$.

    In particular, if $f$ is accessible and $NW(f)=M$, then
    $f$ is transitive.
\end{prop}
\begin{proof}
    For $\ep > 0$ and $y \in M$, let $\Wu_\ep(y)$ be the set of all points
    reachable from $y$ by a path tangent to $\Eu$ of length less than $\ep$. 

    If $x \in V$, then $x \in NW(f)$ implies there are sequences
    $\{x_k\}$ and $\{y_k\}$ both converging to $x$
    and such that $y_k = f^{j_k}(x_k)$ for some non-zero $j_k \in \bbZ$.
    By swapping $x_k$ with $y_k$ if necessary, assume every $j_k$
    is positive.
    If $j_k$ is bounded, then $x$ is periodic, so we may freely
    assume that $j_k \to +\infty$.
    As $V$ is open, there is $\ep>0$ such that $\Wu_\ep(x_k) \subset V$
    for all large $k$.  The uniform expansion of $\Eu$ implies there is
    $r_k \to \infty$ such that
    \begin{math}
        \Wu_{r_k}(y_k) \subset
        f^{j_k}(\Wu_\ep(x_k)) \subset
        f^{j_k}(V) = V
    \end{math}
    and therefore the entire unstable manifold $\Wu(x)$ lies in the closure
    of $V$.
    This proves $\Wu(\overline V)=\overline V$.
    Similarly, $\Ws(\overline V)=\overline V$ and so
    $AC(\overline V) = \overline V$.
\end{proof}
\begin{proof}
    [Proof of \eqref{thm-ABNW}.]
    By \eqref{prop-brintrans},
    any accessible $f$ satisfies case (1) of \eqref{thm-ABNW}.
    Therefore, assume that $f$ is non-accessible.

    For now, assume $f$ has no periodic compact $us$-leaves,
    so that \eqref{thm-irratAB} holds.
    That theorem, with the assumption $NW(f)=M$,
    implies that $NW(r)=\bbS$ and that every point in $M$
    lies in a compact $us$-leaf.
    This shows that \eqref{lemma-KSconj} holds and
    the $r$ in that lemma can be taken as the same $r$ in \eqref{thm-irratAB}.
    As $NW(r)=\bbS$, $r$ is
    topologically conjugate to a rigid rotation $t \mapsto t + \theta$
    and therefore $f$ satisfies case (2) of \eqref{thm-ABNW}.

    For the remainder of the proof,
    assume $f$ has a periodic compact $us$-leaf,
    so that \eqref{thm-ratAB} holds.
    Let $I$ be a connected component of $U$ and $g: p \inv(I) \to p \inv(I)$
    be as in \eqref{thm-ratAB}.
    The condition $NW(f)=M$ implies $NW(g)=p \inv(I)$.
    This is only possible in the first of the three cases in \eqref{thm-ratAB},
    where $g$ is accessible.
    Then,
    $g$ is transitive by \eqref{prop-brintrans}.

    If $t \in \bbS \setminus U$, then $f^n$ restricted to $p \inv(I)$ is
    topologically conjugate to a hyperbolic nilmanifold automorphism
    and is therefore transitive \cite{Franks1}.
    Hence, if $U$ is non-empty, the third case of \eqref{thm-ABNW} is
    satisfied.

    If $U$ is empty, then every $p \inv(t)$ is an $f^n$-invariant compact
    $us$-leaf and
    \eqref{lemma-KSconj} holds with $r:\bbS \to \bbS$
    topologically conjugate to a rigid rational rotation $t \mapsto t + \theta$.
    This shows that $f$ is in case (2) of \eqref{thm-ABNW}.
      \end{proof}
To prove ergodicity of the components of the decomposition given in
\eqref{thm-decomp},
we use results given in \cite{BW-annals},
\cite{RHRHU-accessibility}, and in the classical work of Birkhoff and Hopf.
These results were formulated for systems on compact manifolds,
but the proofs are local in nature, involving short holonomies along stable
and unstable manifolds.
The results, therefore, generalize to the non-compact
case so long as the measure is still finite.

\begin{prop} \label{prop-ergprops}
    Let $f$ be a homeomorphism of a (not necessarily compact)
    manifold $M$ and
    let $C_0(M)$ be the space of continuous functions $M \to \bbR$
    with compact support.
    Suppose $\mu$ is an invariant measure with $\mu(M)=1$ and
    there is an invariant closed submanifold $S$
    such that $\mu$ is equivalent to Lebesgue measure on $S$.
    \begin{enumerate}
        \item
        For $\phi \in C_0(M)$
        the limits
        \[        
            \phi^s(x) = \lim_{n \to \infty} \frac{1}{n} \sum_{k=1}^n \phi f^k(x) \qandq
            \phi^u(x) = \lim_{n \to \infty} \frac{1}{n} \sum_{k=1}^n \phi f^{-k}(x)
        \]
        exist and are equal $\mu$-almost everywhere.

        \item
        There is a countable set $\{\phi_j\}_{j=1}^\infty \subset C_0(M)$
        (depending only on $M$)
        such that $(f,\mu)$
        is ergodic if and only if
        $\phi_j^s$ and $\phi_j^u$ are constant
        $\mu$-almost everywhere
        for every $j$.
    \end{enumerate}
    Further, suppose $f$ is a $C^2$ partially hyperbolic diffeomorphism
    with one dimensional center.
    \saveenum
    \begin{enumerate}
        \resetenum
        \item
        If $\phi \in C_0(M)$,
        then $\phi^s$ is constant on stable leaves and
        $\phi^u$ is constant on unstable leaves.

        \item
        If $S = M$, $X^s, X^u \subset M$ are measurable,
        and
        \[        
            W^s(X^s)=X^s, \quad
            W^u(X^u)=X^u \qandq
            \mu(X^s \triangle X^u) = 0,
              \]
        then there is $X \subset M$ measurable
        such that
        \[        
            AC(X)=X \qandq \mu(X^s \triangle X) = 0 = \mu(X^u \triangle X).
        \]
        \item
        If $S = M$ and $f$ is accessible, then $(f,\mu)$ is ergodic.
    \end{enumerate}  \end{prop}
\begin{proof}
    Item (1) is a re-statement of the classic Birkhoff Ergodic Theorem.

    To prove (2), let $\{\phi_j\}$ be a countable set whose
    linear span is dense in $C_0(M)$ with respect to the supremum norm.
    As any function in $C_0(S)$ may be extended to a function in $C_0(M)$,
    the linear span of $\{\phi_j\}$ is dense in $L^1(\mu)$.
    Suppose
    the bounded linear operator $\phi \mapsto \phi^s$ on $L^1(\mu)$
    takes every element of $\{\phi_j\}$ to the
    subspace of constant functions.
    By density, every $\phi \in L^1(\mu)$ is mapped to the same subspace.
    Therefore $(f,\mu)$ is ergodic.
    The converse statement in (2) follows directly from the properties of
    ergodicity.

    Proofs of (3)--(5) can be found in both \cite{BW-annals} and
    \cite{RHRHU-accessibility}.
\end{proof}
\begin{proof}
    [Proof of \eqref{thm-decomp}]
    As $\mu$ is a finite, $f$-invariant measure
    which is equivalent to Lebesgue, $NW(f) = M$
    by Poincar\'e recurrence.
    Let $p$, $n$, and $U$ then be given as in \eqref{thm-ABNW}.
    By \eqref{prop-pproj}, assume
    $p_* \mu = m$
    where $m$ is Lebesgue measure on $\bbS$.
    Without loss of generality, assume $n = 1$.

    For each connected component $I$ of $U$, the set
    $p \inv (I)$ is an accessibility class
    and therefore $(f, \mu_I)$ is ergodic by \eqref{prop-ergprops}
    where $\mu_I$ is as in \eqref{decomp}.

    Let $\{\phi_j\}_{j=1}^{\infty}$ be as in \eqref{prop-ergprops} and for
    $j \in \bbN$ and $q \in \bbQ$
    define
    \begin{math}
        X^s_{j,q} = \{x \in M: \phi_j^s(x) < q \}.
    \end{math}
    Define $X^u_{j,q}$ similarly.
    By items (3) and (4) of \eqref{prop-ergprops},
    there is $X_{j,q} = AC(X_{j,q})$
    equal mod zero to both $X^s_{j,q}$ and
    $X^u_{j,q}$.
    Define a ``bad'' set $Y$
    by
    \[    
        Y = \bigcup_{j,q} \bigl(
        X^s_{j,q} \triangle X_{j,q}
        \cup
        X^u_{j,q} \triangle X_{j,q}
        \bigr)
    \]
    and note that $\mu(Y) = 0$.
    Equation \eqref{decomp} implies that
    there is a ``good'' set $Z \subset \bbS \setminus U$
    such that $U \cup Z$ has full measure in $\bbS$ and
    $\mu_t(Y \cap p \inv(t)) = 0$ for all $t \in Z$
    where $\mu_t$ is given by the decomposition in \eqref{decomp}.
    By \eqref{prop-pproj}, we may further assume that $\mu_t$ is equivalent to
    Lebesgue measure on $p \inv(t)$ for all $t \in Z$.

    As $p \inv(t)$ is an accessibility class,
    every $X_{j,q} \cap p \inv(t)$ is either empty or all of $p \inv(t)$.
    Therefore for $t \in Z$,
    every $X^s_{j,q}$ and $X^u_{j,q}$
    either has $\mu_t$-measure equal to zero or one,
    and item (2) of \eqref{prop-ergprops} implies that $(f, \mu_t)$ is ergodic.
    Thus, modulo a set of measures whose combined support
    has $\mu$-measure zero, every measure in \eqref{decomp} is ergodic.
    This shows that \eqref{decomp} is the ergodic decomposition of $\mu$.
\end{proof}
One might be tempted to prove \eqref{thm-decomp} by arguing that for $t \notin
U$, $f$ restricted to $p \inv(t)$ is an Anosov diffeomorphism and therefore
the invariant measure $\mu_t$ is ergodic.
The problem is that we have only shown that $p \inv(t)$ is a $C^1$ submanifold
of $M$,
which is not enough regularity to conclude ergodicity for an Anosov system.
Hence, the above proof.

\begin{proof}
    [Proof of \eqref{thm-consAB}]
    If $f$ is in case (1) or (3) of \eqref{thm-ABNW},
    it is fairly easy to show that $f$ is also in the corresponding case of
    \eqref{thm-consAB}.
    Therefore, assume $f$ is in case (2) of \eqref{thm-ABNW}.

    If $\theta$ is rational, then $(v,t) \mapsto (Av, t + \theta)$
    is non-transitive and therefore $f$ is not ergodic.

    Suppose $\theta$ is irrational and $f$ is not ergodic.
    Then there are $j \in \bbN$ and $q \in \bbQ$
    such that the sets
    $X^s_{j,q},$
    $X^u_{j,q},$ and
    $X_{j,q}$,
    defined as in the last proof,
    have neither zero measure nor full
    measure with respect to the $f$-invariant measure $\mu$.
    Write $X = X_{j,q}$.
    As $X = AC(X)$, there is $Y \subset \bbS$ such that $X = p \inv(Y)$
    and $p_* \mu = m$ implies that $m(Y)$ is neither zero nor one.
    The condition $p_* \mu = m$ further implies that $p$ gives a semiconjugacy
    from $f$ to a rigid irrational rotation $R_\theta(x) = x + \theta$ on $\bbS$.
    Then, $f(X)=X$ implies $R_\theta(Y)=Y$ which contradicts the ergodicity
    of $(R_\theta, m)$.
      \end{proof}

\section{Regularity} \label{sec-regularity} 

This section proves \eqref{thm-regularity},
showing that the $us$-lamination of a partially hyperbolic
diffeomorphism is $C^1$ if the center is one-dimensional
and the diffeomorphism is $C^2$.


\begin{prop} \label{prop-niceholo}
    Suppose $f:M \to M$ is a $C^2$
    dynamically coherent partially hyperbolic diffeomorphism
    with one-dimensional center.
    Then any unstable holonomy $h$ inside a $cu$-leaf is $C^1$.
    Moreover, the derivative of $h$ tends uniformly to one as the unstable
    distance between the point $x$ and its image $h(x)$ tends to zero.
\end{prop}
\begin{proof}
    That such a holonomy is $C^1$ is proved in an erratum \cite{PSWc} to the
    paper \cite{PSW}.
    If $y \in \Wu(x)$ and $h$ is the holonomy taking $x$ to $y$,
    then adapting the argument in
    \S3 of \cite{pughshub1972}
    one can show that the norm of the derivative of $h$ at $x$
    is given by
    \[
        J_{x y} =
        \prod_{n=0}^\infty \frac{\|T^c_{f^{-n}(y)}f\|}{\|T^c_{f^{-n}(x)}f\|}
    \]
    where $T^c_z f : \Ec_z \to \Ec_{f(z)}$ is the restriction of the derivative
    $T_z f: T_z M \to T_{f(z)} M$. 
    As $f$ is $C^2$, the derivative $T_z f$ is Lipschitz in $z$ and
    the center bundle $\Ec$ is H\"older by \cite{HPS}.
    Therefore,
    \[
        \log J_{x y}  \le  \sum_{n=0}^\infty L \bigl[\dist(f^{-n}(x), f^{-n}(y))\bigr]^\theta
                     \le  \sum_{n=0}^\infty L \bigl[C \mu^{-n}\bigr]^\theta \bigl[\dist(x,y)\bigr]^\theta
    \]
    for appropriate constants $L, C, \mu > 1$ and $0 < \theta < 1$.
    This shows that $J_{x y}$ tends uniformly to one
    as $\dist(x,y)$ tends to zero.
\end{proof}
\begin{prop} \label{prop-extendg}
    Suppose $f:M \to M$ is a $C^2$
    dynamically coherent partially hyperbolic diffeomorphism
    with one-dimensional center.
    Suppose $L_0 \subset M$ is a compact interval inside a center leaf and
    $g:L_0 \to \bbR$ is $C^1$.
    Then $g$ extends to a $C^1$ function defined on a neighbourhood of\, $L_0$
    which is constant on $us$-leaves.
\end{prop}
\begin{proof}
    Without loss of generality, assume $g$ is defined so that $|g(x)-g(y)|$ is
    the arc length of the center segment between $x$ and $y$.
    Any other $C^1$ function on $L_0$ can be constructed by composition
    with this specific $g$.

    By local product structure and the compactness of $L_0$,
    one may construct a compact set $C \subset M$ containing $L_0$
    with the following properties{:}
    \begin{itemize}
        \item
        The interior of $C$ contains the (one-dimensional) interior
        of $L_0$.
        \item
        If $\Wc(x)$ is a center leaf, then every connected component of
        $\Wc(x) \cap C$ is a compact interval, called a ``center segment.''
        \item
        If $AC(y)$ is a $us$-leaf, then every connected component of
        $AC(y) \cap C$ is a compact set homeomorphic to a closed ball and
        called a ``$us$-plaque.''
        \item
        Each center segment intersects each $us$-plaque in exactly one point.
        \item
        $L_0$ is a center segment.
    \end{itemize}
    By a $C^1$ change of coordinates, assume that $C \subset \bbR^d$.

    Let $\Sigma \subset C$ be the union of all $us$-plaques,
    and $\Sigma' \subset \Sigma$ the union of all $us$-plaques which
    are accumulated on by other $us$-plaques.
    If $x \in \Sigma'$,
    define
    \[
        D(x) =
        \lim_{n \to \infty} \frac{\| \sigma_n \cap L_0 - \sigma \cap L_0\|}
        {\| \sigma_n \cap L - \sigma \cap L\|}
    \]
    where $L$ is the center segment through $x$,
    $\sigma$ is the $us$-plaque through $x$,
    and $\sigma_n$ are $us$-plaques converging to $\sigma$.
    By \eqref{prop-niceholo},
    this limit exists, is independent of the sequence $\sigma_n$ tending to
    $\sigma$, and is non-zero.
    The $C^1$ regularity of the holonomies also implies that
    if $\rho_n$ is another sequence of $us$-plaques converging to $\sigma$,
    then
    \[
        D(x) =
        \lim_{n \to \infty} \frac{\| \sigma_n \cap L_0 - \rho_n \cap L_0\|}
        {\| \sigma_n \cap L - \rho_n \cap L\|}
    \]
    so long as $\sigma_n  \ne  \rho_n$ for large $n$.
    Further, by \eqref{prop-niceholo}, the ratio $D(L_1 \cap \sigma)/D(L_2 \cap \sigma)$
    tends
    uniformly to one as $\dist(L_1,L_2)$ tends to zero.
    As $D$ is continuous when restricted to each center segment and uniformly
    continuous on each $us$-plaque $\sigma$,
    it is therefore continuous on all of $\Sigma'$.
    Define $D(x)=1$ for all $x \in L_0$ and note that this agrees with the
    above definition on the intersection $\Sigma' \cap L_0$.
    Then, choose a continuous positive extension $D:\Sigma \cup L_0 \to \bbR$.

    Also extend $g:L_0 \to \bbR$ to a function
    $g:\Sigma \cup L_0 \to \bbR$ by making it constant on each $us$-plaque.
    To further extend $g$ to a $C^1$ function on all of $C$,
    we will define for each point $x \in \Sigma \cup L_0$ a candidate derivative
    $dg_x:\bbR^d \to \bbR$
    and show that Whitney's extension theorem applies.
    Choose an orientation for $\Ec$ and
    for each $x \in \Sigma \cup L_0$,
    let $v^c_x$ be the unique oriented unit vector in $\Ec_x$.
    Define $dg_x$ as the unique linear map such that
    $dg_x(v^c_x) = D(x)$ and $\ker dg_x = \Eu_x \oplus \Es_x$.
    As both $D(x)$ and the splitting $\Eu_x \oplus \Ec_x \oplus \Es_x$
    are continuous in $x$, the linear map $dg_x$ is continuous in $x$.

    Define the function $R: C \times C \to \bbR$ by
    \[
        R(x_n,y_n) = \frac{1}{\|y_n-x_n\|}{\Bigl(g(y_n)-g(x_n) - dg_{x_n}(y_n-x_n)\Bigr)}.
    \]
    To apply Whitney's extension theorem, one needs to show that for any two
    sequences $\{x_n\}_{n=1}^\infty$ and $\{y_n\}_{n=1}^\infty$ with $\|x_n-y_n\|$ converging to zero,
    the sequence $R(x_n,y_n)$
    also converges to zero.
    If this does not hold, there are sequences $\{x_n\}$ and $\{y_n\}$
    so that $R(x_n,y_n)$ is bounded away from zero.
    Therefore, without loss of generality, one may replace these sequences by
    subsequences and assume $x_n$ and $y_n$ both converge
    to a point $q \in C$.
    We will also restrict to further subsequences as necessary later in the
    proof.

    We prove the convergence in progressively more general cases.

    \textbf{Case 1.\ }
    First, assume $x_n$, $y_n$, and $q$ are all on the same center segment $L  \ne  L_0$.
    Let $\sigma_n$, $\rho_n$ and $\sigma$ be such that
    \[
        \sigma_n \cap L = x_n, \quad
        \rho_n \cap L = y_n, \quad \text{and} \quad
        \sigma \cap L = q.
    \]
    If $\sigma \notin \Sigma'$, then $x_n=y_n=q$ for large $n$.
    Therefore, assume $\sigma \in \Sigma'$.
    Then,
    \[
        \lim_{n \to \infty} \frac{g(y_n)-g(x_n)}{\|x_n-y_n\|} =
        \lim_{n \to \infty} \frac{\| \sigma_n \cap L_0 - \rho_n \cap L_0\|}
        {\| \sigma_n \cap L - \rho_n \cap L\|} =
        D(q).
    \]
    As both the candidate derivative $dg_x$
    and the center direction $v^c_x$ are continuous in $x$,
    \begin{align*}
        \lim_{n \to \infty} \frac{1}{\|y_n-x_n\|}\,dg_{x_n}(y_n-x_n)
        &= \bigl(\lim_{n \to \infty} dg_{x_n}\bigr)
        \bigl(\lim_{n \to \infty} \frac{y_n-x_n}{\|y_n-x_n\|}\bigr) \\
        &= dg_{q}(v^c_q)
        = D(q).
    \end{align*}
    Therefore, $\lim_{n \to \infty} R(x_n, y_n) = D(q)-D(q)=0$.

    \textbf{Case 2.\ }
    Now, consider the case where
    $x_n$ and $y_n$ are on the same center segment $L_n$ for each $n$.
    Define $x^c_n$ to be on the same $us$-plaque as $x_n$
    and the same center segment
    as $q$.  Define $y^c_n$ similarly.
    Then,
    \[
        g(x_n) - g(y_n) = g(x^c_n) - g(y^c_n).
    \]
    By \eqref{prop-niceholo},
    \[
        \lim_{n \to \infty} \frac{\|y_n - x_n\|}{\|y^c_n - x^c_n\|} = 1.
    \]
    Thus,
    \[
        \lim_{n \to \infty} \frac{g(y_n)-g(x_n)}{\|y_n-x_n\|} =
        \lim_{n \to \infty} \frac{g(y^c_n)-g(x^c_n)}{\|y^c_n-x^c_n\|} =
        D(q)
    \]
    where the last equality is by the previous case.  As before,
    \[    
        \lim_{n \to \infty} \frac{1}{\|y_n-x_n\|}\,dg_{x_n}(y_n-x_n)
        = dg_{q}(v^c_q)
        = D(q)
    \]
    and therefore $\lim_{n \to \infty} R(x_n, y_n) = 0$.

    \textbf{Case 3.\ }
    Now consider $x_n$ and $z_n$ as general sequences in $\Sigma$ converging to
    $q$.  Define $y_n$ as the unique point lying on the same center segment as $x_n$
    and the same $us$-plaque as $z_n$.
    By taking subsequences, assume
    \[
        \lim_{n \to \infty} \frac{z_n - y_n}{\|z_n - y_n\|}
    \]
    exists.
    By continuity of the partially hyperbolic splitting, this limit is in
    $\Eu_q \oplus \Es_q$.
    Therefore,
    \[
        \lim_{n \to \infty} \frac{1}{\|z_n-y_n\|}\,dg_{x_n}(z_n-y_n)
        = \bigl(\lim_{n \to \infty} dg_{x_n}\bigr)
        \bigl(\lim_{n \to \infty} \frac{z_n-y_n}{\|z_n-y_n\|}\bigr)
        = 0
    \]
    implying, with $g(z_n)=g(y_n)$, that
    \[
        \lim_{n \to \infty}
        \frac{1}{\|z_n-y_n\|} \bigl(g(z_n)-g(y_n)-dg_{x_n}(z_n-y_n)\bigr)
        = 0.
    \]
    By transversality of the foliations, there is a constant $c_1 > 0$ such
    that
    \begin{math}
        \|z_n - x_n\|  \ge  c_1 \|z_n - y_n\|
    \end{math}
    and therefore
    \[
        \lim_{n \to \infty}
        \frac{1}{\|z_n-x_n\|} \bigl(g(z_n)-g(y_n)-dg_{x_n}(z_n-y_n)\bigr)
        = 0
    \]
    as well.
    Again by transversality, there is $c_2>0$ such that
    \begin{math}
        \|z_n - x_n\|  \ge  c_2 \|y_n - x_n\|
    \end{math}
    and therefore by the previous case
    \[
        \lim_{n \to \infty}
        \frac{1}{\|z_n-x_n\|} \bigl(g(y_n)-g(x_n)-dg_{x_n}(y_n-x_n)\bigr)
        = 0.
    \]
    Added together, these limits show that
    $\lim_{n \to \infty} R(x_n,z_n) = 0$.

    \textbf{Case 4.\ }
    Now consider the case where $x_n \in L_0$ and $z_n \in \Sigma$ for all $n$.
    Define $y_n$ from $x_n$ and $z_n$ exactly as in the last case.
    Then,
    \begin{align*}
        R(x_n,z_n) =
        &\frac{1}{\|z_n-x_n\|} \bigl(g(z_n)-g(y_n)-dg_{x_n}(z_n-y_n)\bigr)
        + \\
        &\frac{1}{\|z_n-x_n\|} \bigl(g(y_n)-g(x_n)-dg_{x_n}(y_n-x_n)\bigr)
    \end{align*}
    and, similar to the previous case, both summands can be shown to converge
    to zero.
    The case $x_n \in \Sigma$ and $z_n \in L_0$ is almost identical.

    \textbf{Case 5.\ }
    If both $\{x_n\}$ and $\{z_n\}$ are in $L_0$, then
    $\lim_{n \to \infty} R(x_n,z_n) = 0$
    simply by the fact that $g$ is $C^1$ when restricted to $L_0$.

    \textbf{The general case.\ }
    The final case to consider is where $\{x_n\}$ and $\{z_n\}$ are general
    sequences in $X = \Sigma \cup L_0$.
    By taking subsequences, one can assume each sequence lies either 
    entirely in $L_0$ or entirely in $\Sigma$ and therefore reduce to a
    previous case.
\end{proof}
We now prove the following restatement of \eqref{thm-regularity}.

\begin{cor}
    If $f:M \to M$ is a non-accessible, partially hyperbolic $C^2$
    diffeomorphism with
    one-dimensional center, the non-open accessibility classes form a $C^1$
    lamination.
    That is, around any point $x \in M$ there is a neighbourhood $V$
    and functions $g:V \to \bbR$ and $\psi:V \to \bbR^{d-1}$
    such that $g \times \psi$ is a $C^1$ embedding 
    and if $AC(y)$ is a $us$-leaf
    and $\sigma$ a connected component of $AC(y) \cap V$,
    then $\sigma = g \inv(t)$ for some $t \in \bbR$.
\end{cor}
\begin{proof}
    Define a coordinate chart $\phi \times \psi:V \to \bbR \times \bbR^{d-1}$
    such that the kernel of 
    the derivative $d \phi : T_x M \to \bbR$ at
    $x$ is equal to $\Eu_x \oplus \Es_x$.  
    By \eqref{prop-extendg}, after replacing $V$ by a subset, there is a $C^1$
    function $g:V \to \bbR$ constant on $us$-plaques and such that $g$ and $\phi$
    are equal on a center segment through $x$.
    Then, the derivative of $g \times \psi$ is invertible at $x$ and so, after
    again replacing $V$ by a subset, $g \times \psi$ is the desired $C^1$
    embedding.
\end{proof}
We now proceed to prove \eqref{prop-pproj}.
Recall the definition of an AI-system from Section \ref{sec-AIsys}.

\begin{prop} \label{prop-leafnbhd}
    Let $f:\hM \to \hM$ be a $C^2$ AI-system
    and $X \subset \hM$ a compact $us$-leaf.
    Then, there is a neighbourhood $V$ of $X$,
    an open subset $U \subset (0,1)$ and
    function $p:V \to (0,1)$ and $\psi:V \to X$
    such that $p \times \psi$ is a
    $C^1$
    diffeomorphism
    and the compact $us$-leaves in $V$ are exactly of the form
    $p \inv(t)$ for $t \notin U$.

    Moreover, $p$ restricted to each center segment
    \,$L \subset V$ is a $C^1$ diffeomorphism.
\end{prop}
In this context,
a center segment is a connected component of the intersection of $V$ with
a center leaf.

\begin{proof}
    There is a neighbourhood $V$ of $X$ such that inside $V$
    each center segment intersects each compact $us$-leaf in a unique point.
    Therefore, the proofs of the previous results of this section
    hold as before with compact
    $us$-leaves now filling the role of $us$-plaques.
    This gives the existence of $p$ and $\psi$.

    As the function $D$ is positive in the proof of
    \eqref{prop-extendg}, for $x \in X$ and unit vector $v^c \in \Ec_x$
    the derivative $d p_x$ of $p$ satisfies $d p_x(v^c)  \ne  0$.
    By continuity, this property holds for all $x$ in a neighbourhood
    of $X$
    and so, by replacing $V$ by a subset, the restriction of $p$ to any
    center segment $L$ has non-zero derivative along all of $L$.
\end{proof}
As it is a local result, \eqref{prop-leafnbhd} also holds for a compact
$us$-leaf in an AB-system instead of an AI-system.
To go from the local to the global requires a technical lemma which ``fills in
the gaps'' between compact $us$-leaves.

\begin{lemma} \label{lemma-gluing}
    Let $N$ be a $C^1$ manifold, and for $0<\ep<\tfrac{1}{2}$ define
    \[
        V_\ep =  N \times ([0,\ep) \cup (1-\ep,1]) \subset N \times [0,1].
    \]
    If there are
    $\ep>0$ and a $C^1$ function $g:V_\ep \to [0,1]$ such that
    \begin{itemize}
        \item
        $\frac {\partial g}{\partial t}\big|_{(x,t)} > 0$
        for all $(x,t) \in V_\ep$, and
        \item
        $g(x,0) = 0$ and $g(x,1) = 1$ for all $x \in N$  \end{itemize}
    then there are $\delta>0$ and a $C^1$ function
    $h:N \times [0,1] \to [0,1]$ such that
    \begin{itemize}
        \item
        $h(x,t)=g(x,t)$ for all $(x,t) \in V_\delta$,
        \item
        $(x,t) \mapsto (x,h(x,t))$ is a $C^1$ diffeomorphism of
        $N \times [0,1]$, and
        \item
        if $x \in N$ satisfies $g(x,t) = t$ for all $(x,t) \in V_\delta$,
        then $h(x,t) = t$ for all $t \in [0,1]$.
    \end{itemize}  \end{lemma}
\begin{proof}
    Pick $\delta>0$ small enough that there is a continuous function
    $h_0:N \times [0,1] \to [0,1]$
    which for each $x \in N$ satisfies the following properties{:}
    \begin{itemize}
        \item
        $t \mapsto h_0(x,t)$ is strictly increasing and linear on each of the
        intervals
        $[\delta,3 \delta]$,
        $[3 \delta, 1 - 3\delta]$, and
        $[1 - 3 \delta, 1 - \delta]$; and
        \item
        $h_0$ agrees with $g$ and
        $\frac{\partial h_0}{\partial t}$ agrees with
        $\frac{\partial g}{\partial t}$
        at the points of the form $(x,\delta)$
        and $(x,1-\delta)$.
    \end{itemize}
    Then, define $h$ by $h(x,t)=g(x,t)$ for $(x,t) \in V_\delta$,
    $h(x,t)=h_0(x,t)$ for $(x,t) \in V_{2 \delta} \setminus V_\delta$, and
    $h(x,t)=\frac{1}{2 \delta} \int_{t-\delta}^{t+\delta} h_0(x,s) ds$ otherwise.
\end{proof}
\begin{prop} \label{prop-pprojAI}
    Let $f:\hM \to \hM$ be a $C^2$ AI-system,
    and $J$ a compact interval inside a center leaf
    such that its endpoints $x_0$ and $x_1$
    lie inside compact $us$-leaves.
    Then there are $r:AC(J) \to AC(x_0)$
    and $p:AC(J) \to [0,1]$
    such that $r \times p$ is a $C^1$ diffeomorphism and
    every compact $us$-leaf in
    $AC(J)$ is of the form $p \inv(t)$ for some $t \in [0,1]$.
\end{prop}
\begin{proof}
    By approximating the center bundle $\Ec$ by a $C^1$ vector field $v$,
    one may define a $C^1$ flow taking points in $AC(x_0)$ to points in
    $AC(x_1)$.
    By rescaling $v$, assume the flow takes each point in $AC(x_0)$ to a point
    in $AC(x_1)$
    in exactly one unit of time.
    This flow then defines a $C^1$ diffeomorphism between $AC(J)$ and
    $AC(x_0) \times [0,1]$.
    Therefore, we may assume our system is defined on a space of the
    form $N \times [0,1]$ where $N$ is a manifold
    $C^1$-diffeomorphic to $AC(x_0)$
    and that $r:N \times [0,1] \to N$
    is given by projection onto the first coordinate.
    Further assume that the flow $v$ is tangent to $\Ec$ on the center leaf
    containing $J$.
    Then, when viewed as a subset of $N \times [0,1]$, $J$ is of the form
    $J = \{x_0\} \times [0,1]$.

    By adapting the arguments in the proofs of \eqref{prop-extendg}
    and \eqref{prop-leafnbhd}, there is a $C^1$ function
    $g:N \times [0,1] \to [0,1]$ which is constant on compact $us$-leaves
    and such that \mbox{$g(x_0,t) = t$} for all $t \in [0,1]$.

    Let $\Sigma \subset N \times [0,1]$
    be the union of all compact $us$-leaves.
    For a point $z \in N \times [0,1]$, let $v^c_z$ be the oriented unit
    vector in $\Ec_z$.
    Then, due to the construction of $g$ as in the proof of
    \eqref{prop-extendg},
    $dg_z(v^c_z)$ is positive for all $z \in \Sigma$.
    As $dg$ is continuous,
    there is a $C^1$ vector field $\hat v$ approximating $v^c$ such that
    $dg_z(\hat v(z))$ is positive for all $z \in \Sigma$.
    By another $C^1$ change of coordinates, assume $v$ is equal to $\hat v$
    and therefore
    $\frac{\partial g}{\partial t}\big|_{(x,t)} = dg_{(x,t)}(v(x,t))$
    for all $(x,t) \in N \times [0,1]$.
    By uniform continuity, there is $\ep > 0$ such that
    $dg_z(v(z)) > 0$ for all $z$ at distance at most $\ep$
    from $\Sigma$.
    Hence, there are at most a finite number of regions $X_i \subset N \times
    [0,1]$ such that
    \begin{itemize}
        \item
        the boundary of $X_i$ is given by two compact $us$-leaves,
        \item
        there are no compact leaves in the interior of $X_i$, and
        \item
        $\frac{\partial g}{\partial t}\big|_{(x,t)}  \le  0$
        for some $(x,t) \in X_i$.
    \end{itemize}
    By \eqref{lemma-gluing},
    define a $C^1$ function $p:N \times [0,1] \to [0,1]$ which is equal to $g$
    everywhere outside of $\cup_i X_i$ and such that
    $\frac{\partial p}{\partial t}\big|_{(x,t)} > 0$
    for all $(x,t) \in N \times [0,1]$.

    Since both $r$ and $p$ are submersions, $r \times p$ has an invertible
    derivative at every point and is therefore a $C^1$ diffeomorphism.
\end{proof}
\begin{cor}
    In the setting of \eqref{prop-pprojAI},
    if $L \subset \hM$ is a center leaf,
    then $p$ and $r$ may be chosen so that $p$ restricted to $L \cap AC(J)$ is
    a $C^1$ diffeomorphism onto $[0,1]$.
\end{cor}
\begin{proof}
    Take $J \subset L$ in the previous proof.
\end{proof}
\begin{cor}
    In the setting of \eqref{prop-pprojAI},
    if $\mu$ is a probability measure
    given by a continuous volume form on $AC(J)$,
    then $p$ may be chosen so that $p_* \mu$ is Lebesgue measure on $[0,1]$.
      \end{cor}
\begin{proof}
    Assume $\rho:N \times [0,1] \to \bbR$ is a positive density function
    such that
    \[    
        \mu(X) = \int_X \rho\,dm_N \times dm
    \]
    where $m_N \times m$ is the product of the Lebesgue measures on $N$ and
    $[0,1]$.

    If $h:[0,1] \to [0,1]$ is defined by
    $h(t) = \mu \bigl( p \inv([0,t]) \bigr)$, then
    \[
        \frac{d h}{d t} =
        \int_{N \times \{t\}} \rho\, dm_N  \]
    is continuous and positive, showing that $h$ is a $C^1$ diffeomorphism.
    Replacing $p$ with the composition $h p$, the result is proved.
\end{proof}
\begin{proof}
    [Proof of \eqref{prop-pproj}]
    As noted in Section \ref{sec-ABsys}, every AB-system $f:M \to M$ lifts to
    an AI-system $\hat f:\hM \to \hM$.
    Moreover, if the AB-system has a compact $us$-leaf,
    the covering $\hM \to M$ has a fundamental domain which is bounded between
    two compact leaves $AC(x)$ and $\beta(AC(x))$ where $\beta$ is the deck
    transformation defined in Section \ref{sec-ABsys}.
    Then,
    \eqref{prop-pprojAI} applies where the region $AC(J)$ is exactly this
    fundamental domain and therefore, there is a $C^1$ surjection
    $p:AC(J) \to [0,1]$.
    Moreover, the candidate derivative in the application of Whitney's
    extension theorem may be chosen so that it agrees on $AC(x)$
    and $\beta(AC(x))$.  Then, $p$ quotients down to a $C^1$ function
    $M \to \bbS$ as desired.

    The other statements in \eqref{prop-pproj} follow from the above two
    corollaries.
\end{proof}

\section{Skew products} \label{sec-products} 

This sections proves \eqref{thm-trivial}
showing that non-accessible skew products have trivial fiber bundles.

\begin{proof}
    [Proof of \eqref{thm-trivial}]
    As the base map $A$ has a fixed point,
    there is a fiber $S$ such that $f(S)=S$.
    By replacing $f$ by $f^2$ if necessary, assume $f$
    preserves the orientation of $S$.
    As $\pi_2(N)$ is trivial (see, for instance, \cite{Franks1}),
    the long exact sequence of fiber bundles gives a short exact sequence
    $0 \rightarrow Z \rightarrow G \rightarrow H \rightarrow 0$
    where $Z = \pi_1(S)$, $G = \pi_1(M)$, and $H = \pi_1(N)$.
    By naturality,
    $f$ induces the commutative diagram
    \[
        \begin{CD}
        0  @>>>  Z  @>>>  G  @>>>  H  @>>>  0\\
        @.  @VV{\id}V  @VV{f_*}V  @VV{A_*}V  @.\\
        0  @>>>  Z  @>>>  G  @>>>  H  @>>>  0.
        \end{CD}
    \]
    As can be shown for any circle bundle with oriented fibers,
    the subgroup $Z$ is contained in the center of
    $G$.
    In this case, as $H=G/Z$ is nilpotent, $G$ is then also nilpotent.

    Skew products have global product structure.
    The proof is similar to that given for AB-systems in Section
    \ref{sec-openness}
    and we leave the details to
    the reader.
    Similar to the case for AB-systems,
    we may then consider the universal cover $\tM$ of $M$,
    a topological line $\tS \subset \tM$ which covers $S$,
    and a lift $\tf:\tM \to \tM$
    such that $\tf(\tS)=\tS$.
    Let $\Lam \subset \tS$ be the set of all points $t \in \tM$
    such that $AC(t)$ is not open.
    Then $G$ induces an action on $\Lam$.
    
    Let $z$ be a non-trivial element of $Z$.
    Then $z$ may be regarded as a fixed-point free homeomorphism of
    $\tS$.
    By \eqref{prop-mu} and \eqref{prop-tau},
    there is a homomorphism $\tau:G \to \bbR$ such that
    $\tau(z)$ is non-zero.
    By \eqref{prop-lamF},
    there is $\lam > 0$ such that $\tau f_*(g) = \lam \, \tau(g)$
    for all $g \in G$.
    Since, $f_*(z) = z$, this implies that $\lam$ equals one.
    By rescaling $\tau$, assume $\tau(Z) = \bbZ$.
    Then, $\tau:G \to \bbR$ quotients to a homomorphism
    $\hat \tau:H \to \bbR/\bbZ$
    and $\hat \tau A_* = \hat \tau$.

    As $A$ is hyperbolic, $A_*$ has no non-trivial fixed points
    and, by \eqref{prop-nofixcoset}, no non-trivial fixed cosets.
    As all of the cosets of $\ker \hat \tau$ are fixed by $A_*$,
    it follows that $\hat \tau = 0$.
    That is, $\tau(G) = \bbZ$.
    One can then define a map which takes each
    $g \in G$ to the unique $z \in Z$ such that
    $\tau(g) = \tau(z)$.
    This shows that the exact sequence
    $0 \rightarrow Z \rightarrow G \rightarrow H \rightarrow 0$ splits.
    Then, $G$ is isomorphic to $H \times Z$ and the bundle is trivial.

    In fact, one can find a compact $us$-leaf directly.
    Viewing $H$ now as a subgroup of $G$
    equal to the kernel of $\tau$,
    choose a point $x \in \tS$ and define $y = \sup_{g \in H} g(x)$.
    Then, with $\mu$ as in \eqref{prop-mu},
    $\mu [x,y) = 0$ which implies $y < +\infty$.
    In other words, $y$ is a well-defined point in $\tS$.
    Since $y$ is in $\Fix(H)$ it projects to a point in $M$ contained in a
    compact $us$-leaf. 
\end{proof}
\section{Infra-AB-systems} \label{sec-infra} 

We now consider infra-AB-systems as defined in Section \ref{sec-results}.

First, recall the definition of an infranilmanifold.
Let $\tN$ be a simply connected nilpotent Lie group.
A diffeomorphism $\phi:\tN \to \tN$ is a \emph{(right) translation}
if there is $v \in \tN$ such that $\phi(u) = u \cdot v$ for
all $u \in \tN$.
Let $\Trans(\tN)$ be the group of all translations (which is canonically
isomorphic to $\tN$ itself).
Let $\Aut(\tN)$ be the group of all automorphisms of $\tN$.
Then
the group of \emph{affine} diffeomorphisms, $\Aff(\tN)$,
is the smallest group containing both $\Trans(\tN)$ and $\Aut(\tN)$.
Equivalently,
$\psi \in \Aff(\tN)$ if and only if
there is $\phi \in \Aut(\tN)$ and $v \in \tN$ such that
$\psi(u) = \phi(u) \cdot v$ for all $u \in \tN$.

If a subgroup
$\Gamma < \Aff(\tN)$ is such that
$\Gamma \cap \Trans(\tN)$ has finite index in $\Gamma$ and
$N_0 := \tN / \Gamma$ is a compact manifold,
then $N_0$ is a (compact) \emph{infranilmanifold}.
If $A \in \Aff(\tN)$ quotients to a function $A_0:N_0 \to N_0$
then $A_0$ is also called \emph{affine}.

\begin{thm} \label{thm-infraAB}
    Suppose $f_0$ is a conservative $C^2$ infra-AB-system.
    Then, either
    \begin{enumerate}
        \item $f_0$ is accessible and stably ergodic,
        \item
        $\Eu$ and $\Es$ are jointly integrable and $f_0$ is topologically
        conjugate to an algebraic map, or
        \item
        there are $n  \ge  1$, a $C^1$ surjection $p_0$ from $M_0$ to
        either $\bbS$ or $\orbi$,
        and a non-empty open subset $U \subsetneq p_0(M_0)$ with the following
        properties.
        
        If $t \notin U$ then $p_0 \inv(t)$ is
        an $f_0^n$-invariant compact $us$-leaf
        homeomorphic to an infranilmanifold.
        Moreover, every $f_0$-periodic compact $us$-leaf is of this form.

        If $I$ is a connected component of $U$, then $p_0 \inv(I)$
        is $f_0^n$-invariant and homeomorphic to a (possibly twisted)
        I-bundle over an infranilmanifold.
    \end{enumerate}  \end{thm}
This theorem is proved at the end of the section and
the exact nature of the ``algebraic map'' in case (2) is given in the proof.
Also, as will be evident from the proof, if $\Ec$ is orientable then
$p_0(M_0) = \bbS$.  Otherwise, $p_0(M_0)=\orbi$ which is the 
1-dimensional orbifold constructed by quotienting $\bbR$ by both $\bbZ$ and the
involution $t \mapsto -t$.
This orbifold is homeomorphic to a compact interval.
A set $p \inv(I)$ will be twisted (as an I-bundle) if and only if
$I$ is homeomorphic to a half-open interval.

The ergodic decomposition given in \eqref{thm-decomp} also generalizes.

\begin{thm}
    Let $f_0:M_0 \to M_0$ be a $C^2$ infra-AB-system
    and suppose there is a smooth, $f_0$-invariant, non-ergodic measure
    $\zeta$ supported on $M_0$.
    Then,
    there are $n  \ge  1$, a $C^1$ surjection $p_0$ from $M_0$ to
    either\, $\bbS$ or\, $\orbi$,
    and an open subset $U \subsetneq p_0(M_0)$ such that
    \begin{equation} \label{infradecomp}
        \zeta = \sum_I m(I)\,\zeta_{I} + \int_{t \notin U} \zeta_t\,dm(t)  \end{equation}
    is the ergodic decomposition for $(f_0^n, \zeta)$.
\end{thm}
Here, the components $\zeta_I$ and $\zeta_t$
of the decomposition are defined analogously to
\eqref{decomp}.

\begin{proof}
    Let $\pi:M_0 \to M$ be the finite covering and
    $f$ an AB-system such that $\pi f = f_0^m \pi$ for some $m  \ge  1$.
    Then, $\zeta$ lifts to a measure $\mu$ on $M$ which (up to rescaling the
    measure so that $\mu(M)=1$)
    satisfies the hypotheses of \eqref{thm-decomp}.
    If $\zeta_t$ is a component of the decomposition \eqref{infradecomp}, then
    its support is a single accessibility class $X_0$.
    If $X$ is a connected component of $\pi \inv(X_0) \subset M$, then there is
    an ergodic component $(f^n,\mu_t)$ of $(f^n,\mu)$
    where $\mu_t$ is supported on $X$ and
    such that $\pi^* \mu_t$ (up to
    rescaling) is equal to $\zeta_t$.
    Ergodicity of $(f_0^{m n}, \zeta_t)$ then follows from the ergodicity of
    $(f^n, \mu_t)$.
    Ergodicity of components of the form $\zeta_I$ can be proven similarly.
\end{proof}
The theorems in Section \ref{sec-results} concerning non-conservative
AB-systems may also be generalized using
techniques similar to those in the proof of \eqref{thm-infraAB} below.
In the interests of brevity, we leave the statements and proofs of
these other results to the reader.
The following two known results about functions on infranilmanifolds
will be useful.

\begin{lemma} \label{lemma-liftN}
    If $N_0$ is an infranilmanifold, there is a nilmanifold $N$
    finitely covering $N_0$
    such that every homeomorphism of $N_0$ lifts to $N$.
\end{lemma}
\begin{proof}
    This follows from the fact that $\Gamma \cap \Trans(\tN)$
    is the unique maximal normal nilpotent subgroup of $\pi_1(M)$.
    A proof of this is given in \cite{Auslander}, a paper which also contains
    an infamously incorrect result about maps between infranilmanifolds.
    (See the discussion in \cite{lee-raymond}.)
    However, the proof of the above fact about $\Gamma \cap \Trans(\tN)$
    is widely held to be correct.
\end{proof}
%
%
%

\begin{lemma} \label{lemma-hypcom}
    If a homeomorphism $B$ on a
    compact infranilmanifold $N_0$ commutes with a
    hyperbolic affine diffeomorphism $A$, then $B$ itself is affine.
\end{lemma}
\begin{proof}
    This follows by a combination of the results of Mal'cev and Franks.
    First, consider the case where $N=N_0$ is a nilmanifold.
    Let $x$ be a fixed point of $A$. Then $y := B(x)$ is also a fixed point of
    $A$.  Using the standard definition of the fundamental group for based
    spaces, the diagram
    \[
        \begin{CD}
        \pi_1(N,x) @>B_*>> \pi_1(N,y) \\
        @VV{A_*}V @VV{A_*}V \\
        \pi_1(N,x) @>B_*>> \pi_1(N,y) \\
        \end{CD}
    \]
    commutes. By \cite{malcev}, there is a unique affine map
    $\phi:(N,x) \to (N,y)$ such that $\phi_*=B_*$.  (If $x  \ne  y$ one shows this
    by considering
    two distinct lattices of the form
    $\tx \Gamma \tx \inv$ and $\ty \Gamma \ty \inv$
    on the Lie group $\tN$ in order to construct a Lie group homomorphism
    which quotients down to $\phi$.)

    As $\phi_*A_* = A_*\phi_*$, the uniqueness given
    in \cite{malcev} entails that $\phi A = A \phi$ as functions on $N$.
    As $N$ is aspherical, $\phi$ is homotopic to $B$.
    Then, using that $A$
    is a $\pi_1$-diffeomorphism as defined in \cite{Franks1}, it follows
    that $\phi$ and $B$ are equal.

    Now suppose $N_0$ is an infranilmanifold.  
    By \eqref{lemma-liftN}, there is a nilmanifold $N$ and
    a normal finite covering $N \to N_0$ such that both $A$ and $B$
    lift to functions $N \to N$.
    By abuse of notation, we still call these functions $A$ and $B$.
    As the covering is finite, there is $j  \ge  1$ such that
    $A^j \gam = \gam A^j$
    for every deck transformation $\gamma$.
    In particular, there is a deck transformation
    $\gam:N \to N$ such that $A^j B = B A^j \gam$.
    Then,
    $A^{j k} B = B (A^j \gam)^k = B A^{j k} \gam^k$ for all $k \in \bbZ$,
    and,
    taking $k  \ge  1$ such that $\gam^k$ is the identity,
    $A^{j k}$ commutes with $B$ and the problem reduces to the previous case.
\end{proof}
\begin{prop} \label{prop-infraskew}
    Suppose $f_0$ is a partially hyperbolic skew product where the base map
    is a hyperbolic infranilmanifold automorphism and $\Ec$ is
    one-dimensional.
    If $f_0$ is not accessible, it is an infra-AB-system.
\end{prop}
\begin{proof}
    Lift
    the fiber bundle projection $\pi:M_0 \to N_0$
    to $\tilde \pi:\tM \to \tN$
    where $\tM$ and $\tN$ are the universal covers.
    Let $G$ consist of those deck transformations
    $\alpha \in \pi_1(M_0)$
    which preserve the orientation of the lifted center bundle and for
    which $\tilde \pi \alpha = \gam \tilde \pi$
    for some $\gam \in \Trans(\tN)$.
    Then, $G$ is a finite index subgroup of $\pi_1(M_0)$
    defining a finite cover $M = \tM / G$
    and
    one can show that $f_0:M_0 \to M_0$ lifts to $f:M \to M$
    where the base map $A_0:N_0 \to N_0$ lifts to the nilmanifold
    $\tN / \tilde \pi(G)$.
    If $f_0$ is not accessible, then $f$ is not accessible.
    The fiber bundle on $M$ is then trivial by \eqref{thm-trivial},
    implying that $f^2$, which preserves the orientation of $\Ec$,
    is an AB-system.
\end{proof}
We now prove \eqref{thm-infraAB}.

\begin{assumption}
    For the remainder of the section,
    assume $f:M \to M$ is a non-accessible conservative $C^2$ AB-system,
    $\pi:M \to M_0$ is a (not-necessarily normal) finite covering map and
    that $f_0:M_0 \to M_0$ and $m  \ge  1$ are such that
    $\pi f = f_0^m \pi$.
\end{assumption}
Note this implies that $f_0$ is partially hyperbolic
and the splitting on the tangent
bundle $T M_0$ lifts to
the splitting for $f$ on $TM$.

For now, make the following additional assumptions,
which will be removed later.
\begin{assumption} \label{assume-infra}
    Assume until the end of the proof of \eqref{prop-inflastassume} that
    \begin{itemize}
        \item
        $\Ec$ on $M_0$ is orientable;
        \item
        $f_0$ preserves the orientation of $\Ec$; and
        \item
        $m=1$, that is, $\pi f = f_0 \pi$.
    \end{itemize}  \end{assumption}
By the assumption $m=1$, both $f_0$ and $f$ can be lifted to the same map
$\tf$ on the universal cover $\tM$.

As $f$ is an AB-system defined by nilmanifold automorphisms $A,B:N \to N$,
there is a map $H : \tM \to \tN$ whose fibers are the center leaves of $f$
and where $\tN$ is the universal cover of $N$ and therefore a nilpotent Lie
group.
Further, $A$ lifts to a hyperbolic automorphism of $\tN$, which we also denote
by $A$, and the leaf conjugacy implies that $H \tf = A H$.


Define $\tS = H \inv(\{0\})$ where $0$ is the identity element of the Lie group.
Then $\tS$ is an $\tf$-invariant center leaf which covers a circle
$S \subset M$ and $S$ further covers a circle $S_0 \subset M_0$.
By \eqref{thm-consAB},
there is a $C^1$ surjection $p:M \to \bbS$ and a constant $\theta \in \bbS$ 
such that if
$x \in M$ has non-open accessibility class $AC(x)$
then 
$p$ is constant on $AC(x)$ and $p f(x)=p(x) + \theta$.
By \eqref{prop-pproj},
assume $p$ restricted to $S$ is a $C^1$ diffeomorphism.
Using $p$ and the covering $\pi:M \to M_0$, define a map
\[
    q : M_0 \to \bbS, \quad x_0 \mapsto \sum_{y \in \pi \inv(x_0)} p(y).
\]
It follows that if
$x_0 \in M_0$ has non-open accessibility class $AC(x_0)$
then 
$q$ is constant on $AC(x_0)$ and $q f_0(x_0)=q(x_0) + \theta d$
where $d$ is degree of the covering.
Further, $q$ restricted to $S_0$ is a $C^1$ covering from $S_0$ to $\bbS$
(though not necessarily of degree $d$).
After lifting $q$ to a map $\tq:\tM \to \bbR$,
there is a homomorphism $q_*:\pi_1(M_0) \to \bbZ$ such that
$\tq \gam(\tx) = \tq(\tx) + q_*(\gam)$ for every $\tx \in \tM$
and deck transformation $\gam \in \pi_1(M_0)$.

As the deck transformations preserve the lifted center foliation,
for each $\gam \in \pi_1(M_0)$, there is a unique homeomorphism
$B_\gam:\tN \to \tN$ such that
$H \gam = B_\gam H$.

\begin{lemma} \label{lemma-allaff}
    $B_\gam \in \Aff(\tN)$ for all $\gam \in \pi_1(M_0)$. 
\end{lemma}
\begin{proof}
    We may view $\pi_1(M)$ as a finite index subgroup of $\pi_1(M_0)$.
    The definition of an AB-system implies that
    $B_\gam \in \Aff(\tN)$ for all $\gam \in \pi_1(M)$.

    Now consider the subgroups $K_3 < K_2 < K_1 < \pi_1(M_0)$
    defined as follows{:}
    \begin{align*}
        &K_1 \text{ is the kernel of $q_*$}, \\
        &K_2 = K_1 \cap \pi_1(M),\ \ \text{and}\\
        &K_3 = \{\alpha \in K_2 : \alpha \beta K_2 = \beta K_2
        \text{ for all } \beta \in K_1 \}.
    \end{align*}
    By its definition, $K_3$ is a normal finite index subgroup of
    $K_1$.
    The lift $\tf$ of $f_0$ induces a homomorphism
    $f_*: \pi_1(M_0) \to \pi_1(M_0)$
    given by $f_*(\gam) = \tf \gam \tf \inv$.
    There is a constant $c \in \bbR$ such that
    \[
        \tq \tf(\tx) = \tq(\tx) + c
    \]
    for all $\tx \in \tM$ with non-open accessibility class.
    This implies that $f_*(K_1) = K_1$.
    From this, one can show that $f_*(K_2) = K_2$
    and therefore $f_*(K_3) = K_3$.

    Note that $N_3 := \tN / \{B_\gam : \gam \in K_3\}$ is a nilmanifold
    (which finitely covers the original nilmanifold $N$),
    and the hyperbolic Lie group automorphism $A:\tN \to \tN$ descends to an
    Anosov diffeomorphism on $N_3$.

    Suppose $\gam \in K_1$.
    As $f_*$ permutes the cosets of $K_3$,
    there is $j \ge 1$ such that $f_*^j(\gam) K_3 = \gam K_3$.
    This implies that $A^j$ and $B_\gam$ descend to commuting diffeomorphisms
    on $N_3$.
    Then, by \eqref{lemma-hypcom}, $B_\gam$ is affine.
    Thus, we have established the desired result for all $\gam \in K_1$,
    and further shown that
    $N_1 := \tN / \{ B_\gam : \gam \in K_1 \}$
    is an infranilmanifold
    (finitely covered by the original nilmanifold $N$).

    Now suppose $\gam \in \pi_1(M_0)$ is an arbitrary deck transformation.
    Then
    \[
        \tq \tf \gam \tf \inv \gam \inv (\tx) = \tq(\tx)
    \]
    for all $\tx \in \tM$ with non-open accessibility class.
    This implies that $f_*(\gam) K_1 = \gam K_1$.
    and so $A$ and $B_\gam$ descend to commuting diffeomorphisms
    on $N_1$.
    As $A$ is hyperbolic, $B_\gam \in \Aff(\tN)$ by \eqref{lemma-hypcom}.
\end{proof}
If $f$ is accessible, then clearly $f_0$ is accessible.
Therefore to prove
\eqref{thm-infraAB}, it is enough to consider $f$ in cases (2) and (3)
of \eqref{thm-consAB}.

\begin{prop} \label{prop-inflastassume}
    If $f$ is in case (3) of \eqref{thm-consAB}
    and $f_0$ satisfies assumption \eqref{assume-infra},
    then $f_0$ is in case (3) of \eqref{thm-infraAB}.
\end{prop}
\begin{proof}
    By replacing $f_0$, $f$, and $\tf$ by iterates, assume $n = 1$ in
    \eqref{thm-consAB} and that the lift $\tf$ was chosen so that $\tf(\tX)=\tX$
    for every accessibility class $\tX \subset \tM$.

    The image of $q_*$ is equal to $\ell \bbZ$ for some $\ell  \ge  1$.
    Then $\tp_0 := \tfrac{1}{\ell} \tq$ quotients to a function $p_0:M_0 \to \bbS$.
    As the original $p:M \to \bbS$ was $C^1$, the functions
    $q$, $\tq$, $\tp_0$, and $p$ are also $C^1$.
    Also, $p_0$ is constant on compact $us$-leaves and
    its restriction to $S_0$ is a $C^1$ covering.
    If, for some $t \in \bbS$,  $X_0$ and $Y_0$ are compact $us$-leaves in the
    pre-image $p_0 \inv(t)$, then they lift to closed $us$-leaves
    $\tX, \tY \subset \tM$ such that $\tp_0(\tX) - \tp_0(\tY)$ is an integer.
    By the definition of $\tp_0$, there is then a deck transformation taking
    $\tX$ to $\tY$ and so $X_0 = Y_0$.
    This shows that every compact $us$-leaf in $M_0$ is of the form
    $p_0 \inv(t)$ for some $t$.

    If $X_0$ is instead an open accessibility class, then its boundary
    consists of two compact $us$-leaves and from this one can show that
    $p_0 \inv(p_0(X_0)) = X_0$.

    Note that every accessibility class $X_0$ on $M_0$ is the projection
    of an accessibility class $\tX$ on $\tM$.
    As $\tf$ fixes accessibility classes, so does $f_0$.
    Further, using $K_1$ and $N_1$ as in the proof of the lemma above,
    $X_0$ is homeomorphic to $\tX / K_1$.
    If $\tX$ is non-open, then $\tX / K_1$ is homeomorphic to the infranilmanifold
    $N_1$. If $\tX$ is open, then $\tX / K_1$ is an I-bundle over $N_1$
    where the fibers of the I-bundle are segments of center leaves.

    This shows that $f_0$ satisfies case (3) of \eqref{thm-infraAB}.
\end{proof}
\medskip

We now remove the additional assumptions above and show that this result still
holds.

\begin{prop}
    If $f$ is in case (3) of \eqref{thm-consAB}
    and $f_0$ does not satisfy assumption \eqref{assume-infra},
    then $f_0$ is in case (3) of \eqref{thm-infraAB}.
\end{prop}
\begin{proof}
    In case (3) of \eqref{thm-infraAB},
    we are free to replace $f_0$ by an iterate.
    By replacing $f_0$ by $f_0^m$, one can assume $m=1$.
    That is, $\pi f = f_0 \pi$.
    By replacing $f_0$ by $f_0^2$, one can assume $f_0$ preserves the
    orientation of any orientable bundle.
    Thus, the only condition to test is when $\Ec$ is non-orientable.

    Any non-orientable bundle on a manifold lifts to an orientable bundle on a
    double cover and any
    bundle-preserving diffeomorphism lifts as well.
    Therefore, we are free to consider the following situation.
    As before, $\Ec$ is orientable and $f_0$ preserves the orientation,
    but now there is an involution $\tau:M_0 \to M_0$,
    such that $\tau$ reverses the orientation of $\Ec$ and $\tau$ commutes with
    $f_0$.
    As a consequence of this commutativity, $\tau$ preserves the partially
    hyperbolic splitting of $f_0$.
    Choose a continuous function $p_1:M_0 \to \bbS$
    which satisfies $2 p_1(x) = p_0(x) - p_0 \tau(x)$.
    As $\tau^2$ is the identity,
    $p_1 \tau(x) = -p_1(x)$ and so $p_1$ descends to a function
    $p_2:M_0/\tau \to \orbi$.

    Since $\bbS \to \bbS,\ x \mapsto -x$ has two fixed points,
    one can show that $\tau$ fixes exactly two accessibility classes on $M_0$.
    Let $X_0$ be one of these two classes,
    and lift $\tau$ and $X_0$ to the universal cover to get
    $\tX$ and $\tilde \tau$ such that $\tilde \tau(\tX) = \tX$.
    As $f$ and $\tau$ commute, it follows from
    an adaptation of \eqref{lemma-allaff}
    that $B_{\tilde \tau} \in \Aff(\tN)$.
    If $X_0$ is compact, then $X_0 / \tau$ is homeomorphic to an
    infranilmanifold.
    If instead $X_0$ is open, then $X_0$ is an I-bundle over $N_0$ where the
    fibers are center segments, and $\tau$ reverses the orientation of these
    fibers.
    Therefore, $X_0 / \tau$ is a twisted I-bundle over an infranilmanifold.

    This shows that case (3) holds for the quotient of $f_0$ to
    $M_0/\tau$ where $p_0$ and $U \subset \bbS$
    are replaced by $p_2$ and $U/\bbZ_2 \subset \orbi$.
\end{proof}

\medskip

Now consider the situation where $f$ is in case (2) of \eqref{thm-consAB}.
The following proposition shows that
$f_0$ is ``algebraic'' as stated in case (2)
and concludes the proof of \eqref{thm-infraAB}.

\begin{prop} \label{prop-infraalg}
    Suppose $f_0$ is an infra-AB-system and
    $\Eu \oplus \Es$ is integrable.  Then there is a lift $\tf_0$
    of $f_0$ to the universal cover $\tM$ and a homeomorphism
    $h : \tM \to \tN \times \bbR$ such that
    \[
        h \tf_0 h \inv \in \Aff(\tN) \times \Isom(\bbR)
    \]
    and
    \[
        h \gam h \inv \in \Aff(\tN) \times \Isom(\bbR)
    \]
    for every deck transformation $\gam \in \pi_1(M_0)$.
\end{prop}
Here, $\Isom(\bbR)$ is the group of
functions of the form
$t \mapsto t+c$ or $t \mapsto -t+c$.

\begin{proof}
    First consider the case where $f_0$ satisfies assumption \eqref{assume-infra}
    and
    recall the functions $H: \tM \to \tN$ and $\tq:\tM \to \bbR$ defined earlier in
    this section.
    By global product structure and the integrability of $\Eu \oplus \Es$,
    $H \times \tq$ is a homeomorphism.
    The results already given in this section then show that
    $h = H \times \tq$ satisfies
    the conclusions of the lemma.

    If $f_0$ does not satisfy \eqref{assume-infra} and $\Ec$ is orientable on
    $M_0$, then there is $m>1$ such that $f_0^m$ satisfies \eqref{assume-infra}.
    Let $H$ and $\tq$ be given for $f_0^m$.
    Define $a = +1$ if $\tf_0$ preserves the orientation of $\Ec$
    and $a = -1$ if $\tf_0$ reverses the orientation.
    Define $r:\tM \to \bbR$
    by
    \begin{math}
        r(x) = \sum_{k = 0}^{m-1}a^k\ \tq \tf_0^k(x)
    \end{math}
    and take
    $h = H \times r$.

    If $\Ec$ is non-orientable on $M_0$, then $f_0$ lifts to a double cover
    on which $\Ec$ is orientable.
    Then, let $H$ and $r$ be defined as in the previous case.
    Choose a deck transformation $\tilde \tau:\tM \to \tM$
    which reverses the orientation of $\Ec$ on $\tM$
    and
    define a function $s:\tM \to \bbR$ by
    $s(x) = r(x)-r \tilde \tau(x)$
    and take
    $h = H \times s$.
\end{proof}

\section{Openness} \label{sec-openness} 

This section establishes that AB-systems have global product structure
and form an open subset of the space of $C^1$ diffeomorphisms.

\begin{lemma} \label{lemma-onesub}
    Suppose $G$ is a simply connected nilpotent Lie group.
    For any distinct $u,v \in G$, there is a unique one-dimensional Lie
    subgroup $G_{u,v}$ such that $v \inv u \in G_{u,v}$.
    (That is, $u$ lies in the coset $v G_{u,v}$.)
\end{lemma}
\begin{proof}
    This follows from the fact that for such groups, the exponential map from
    the Lie algebra to the Lie group is surjective \cite{malcev}.
\end{proof}
A right-invariant metric on such a group $G$ is a metric
$d:G \times G \to [0,\infty)$ such that
$d(u,v)=d(u \cdot w, v \cdot w)$ for all $u,v,w \in G$.
For such a metric, we define a function $d_1:G \times G \to [0,\infty)$
where $d_1(u,v)$ is the length of the path from $u$ to $v$
which lies in the coset
$v G_{u,v}$ given by \eqref{lemma-onesub}.
Clearly, $d(u,v)  \le  d_1(u,v)$ for all $u,v \in G$.
Further, $d_1$ is continuous and the ratio
$d_1(u,v)/d(u,v)$ tends uniformly to one as $d(u,v)$ tends to zero.
Note that $d_1$ is not a metric on $G$ in general.
(If $G = \bbR^d$ is abelian, however, the coset $u G_1$ is simply the line
through $u$ and $v$ and $d=d_1$.)

If $\phi:G \to G$ is an automorphism and $G_1$ is a one dimensional subgroup,
then there is $\lam$ such that $d_1(\phi(u),\phi(v)) = \lam d_1(u,v)$
for all $u,v \in G$ with $u \in v G_1$.
This follows because both $G_1$ and $\phi(G_1)$ are Lie groups isomorphic to
$\bbR$ and $d_1$ restricts to a right-invariant metric on either of $G_1$
or $\phi(G_1)$.

\begin{lemma} \label{lemma-autseq}
    Suppose $G$ is a simply connected nilpotent Lie group,
    $d$ is a right-invariant metric,
    $\{\phi_k\}$ is a sequence of Lie group automorphisms of $G$,
    $G_1 \subset G$ is a one-dimensional Lie subgroup,
    $u_0 \in G$, and $v_0 \in u_0 G_1$ with $u_0  \ne  v_0$.
    \begin{enumerate}
        \item If
        \,$\lim_{k \to \infty} d(\phi_k(u_0),\phi_k(v_0)) = 0$, then \\
        $\lim_{k \to \infty} d(\phi_k(u),\phi_k(v)) = 0$
        for all $u \in G$ and $v \in u G_1$.

        \item If $a  \ge  1$ and
        \,$\lim_{k \to \infty} a^k d(\phi_k(u_0),\phi_k(v_0)) = 0$,
        then \\
        \,$\lim_{k \to \infty} a^k d(\phi_k(u),\phi_k(v)) = 0$
        for all $u \in G$ and $v \in u G_1$.

        \item If
        \,$\sup_k d(\phi_k(u_0),\phi_k(v_0)) < \infty$,
        then \\
        \,$\sup_k d_1(\phi_k(u_0),\phi_k(\hat v)) = 1$
        for some $\hat v \in u_0 G_1$.
    \end{enumerate}  \end{lemma}
\begin{proof}
    Let $\lam_k$ be such that $d_1(\phi^k(u),\phi^k(v)) = \lam_k d_1(u,v)$
    when $u \in v G_1$.
    Then in the first item, the two limits hold if and only if
    $\lam_k \rightarrow 0$ and so one implies the other.
    For the second item, consider $a^k \lam_k$.
    For the final item, if the first supremum is finite, then
    $\Lam := \sup_k \lam_k < \infty$ and one can
    take $\hat v \in v_0 G_1$ such that $d_1(\hat v, v_0) = 1/\Lam$.
\end{proof}
We now show that every AB-system has global product structure.

\begin{proof}
    [Proof of \eqref{thm-ABGPS}]
    Let $f:\tM \to \tM$ be the lift of the AB-system to the universal cover
    and $h:\tM \to \tM_B$ the lifted leaf conjugacy to the AB-prototype.
    The functions $f$ and $h$ are written without tildes as all the analysis
    will be on the universal covers.

    Measuring distances on the manifold $\tM_B$ requires care.
    The metric $d_{\tM_B}$ on $\tM_B$ is defined by lifting a metric from $M_B$.
    If $p_k = (u_k, s_k)$, and $q_k = (v_k, t_k)$ are sequences in
    $\tM_B = \tN \times \bbR$,
    then $d_{\tM_B}(p_k,q_k)$ may not converge to zero, even if both
    $d_{\tN}(u_k,v_k) \rightarrow 0$ on $\tN$ and $|s_k-t_k| \rightarrow 0$ on
    $\bbR$.
    The convergence depends on the exact nature of the automorphism $B$.
    If $s_k$ and $t_k$ are \emph{bounded}
    sequences in $\bbR$, however, then one can show in this special case
    that $d_{\tM_B}(p_k,q_k) \rightarrow 0$ if and only if both
    $d_{\tN}(u_k,v_k) \rightarrow 0$ on $\tN$ and $|s_k-t_k| \rightarrow 0$ on
    $\bbR$.

    There is a deck transformation $\beta:\tM_B \to \tM_B$
    defined by
    $\beta(v,t) = (Bv, t-1)$ which is an isometry with respect to $d_{\tM_B}$.
    For general $\{p_k\}$ and $\{q_k\}$, let $\{n_k\}$
    be the unique sequence of integers
    such that $0  \le  |s_k - n_k| < 1$
    for all $k$.
    Then,
    $\beta^{n_k}(p_k) \in \tN \times [0,1)$ for all $k$ and
    \[    
        d_{\tM_B}(p_k,q_k) =
        d_{\tM_B}(\beta^{n_k}(p_k), \beta^{n_k}(q_k)) \rightarrow 0
    \]
    if and only if both
    \[    
        d_{\tN}(B^{n_k}(u_k),B^{n_k}(v_k)) \rightarrow 0
        \qandq |s_k-t_k| \rightarrow 0.
    \]

    In what follows, we write $d$ without a subscript for the metrics on $\tM$,
    $\tM_B$, and $\tN$.
    There is no ambiguity as they are all treated as distinct manifolds.
    If $Y$ is a subset of one of these three manifolds, then
    \[
        \dist(x,Y) := \inf_{y \in Y} d(x,y).
    \]
    Also let $d_s(x,y)$ denote distance measured along
    the corresponding stable foliation: $W^s_f$ if $x,y \in \tM$, $W^s_{A}$ if
    $x,y \in \tN$, and $W^s_{A \times \id}$ if $x,y \in \tN \times \bbR$.
    Similarly for $d_u$ and $d_c$.

    The leaf conjugacy implies that every 
    $cs$-leaf of $f$ intersects a $cu$-leaf
    in a unique center leaf.
    Therefore, establishing global product structure
    reduces to showing existence and uniqueness
    of intersections inside the $cs$ and $cu$ leaves.

    \emph{Uniqueness.}\quad
    Suppose $x \in \tM$ 
    and $x  \ne  y \in \Wc_f(x) \cap \Ws_f(x)$.
    Then as $k \to \infty$,
    \[
        d_s(f^k(x), f^k(y)) \rightarrow 0 \qandq
        d_c(f^k(x), f^k(y)) \nrightarrow 0
      \]
    since if both sequences tended to zero,
    local product structure would imply that $x$ and $y$ were equal.
    Define $p_k = h f^k(x)$ and $q_k = h f^k(y)$.
    As the leaf conjugacy is uniformly continuous,
    $d(p_k, q_k) \rightarrow 0$
    and $d_c(p_k, q_k) \nrightarrow 0$.
    If $p_k=(u_k,s_k)$ and $q_k=(v_k,t_k)$,
    then, as noted above,
    \[
        d(p_k,q_k) \rightarrow 0  \quad \Rightarrow \quad 
        |s_k-t_k| \rightarrow 0  \quad \Rightarrow \quad 
        d_c(p_k,q_k) \rightarrow 0,
    \]
    a contradiction.

    \emph{Existence.}\quad
    Suppose $x \in \tM$ lies on a center leaf $L_0$ and
    $L_1 \subset \Wcs_f(x)$ is a distinct center leaf.
    Then $h(L_0) = \{v_0\} \times \bbR$ and
    $h(L_1) = \{v_1\} \times \bbR$ for distinct points
    $v_0,v_1 \in \tN$.
    As $L_0$ and $L_1$ are subsets of the same $cs$-leaf of $f$,
    $v_0$ and $v_1$ lie on the same stable leaf of $A$.
    By \eqref{lemma-onesub},
    there is a one-dimensional subgroup $\tN_1 \subset \tN$
    such that $v_0 \inv \cdot v_1 \in \tN_1$.
    By item (2) of \eqref{lemma-autseq}, the coset $v_0 \tN_1$ is a subset of
    $\Ws_A(v_0)$.

    If $U^s_f$ is a small neighbourhood of $x$ in $W^s_f(x)$,
    then $h(U^s_f) \subset W^s_A(v_0) \times \bbR$ and
    the set $h(W^c_f(U^s_f)) = W^c_{A \times \id}(h(U^s_f))$
    is a neighbourhood of $h(x)$ in $W^s_A(v_0) \times \bbR$.
    Therefore, if $v \in \Ws_A(v_0)$ is sufficiently close to $v_0$,
    then
    there is $y \in \Ws_f(x)$ such that $h(y) \in \{v\} \times \bbR$.
    \begin{figure}[t]
    {\centering
    \includegraphics{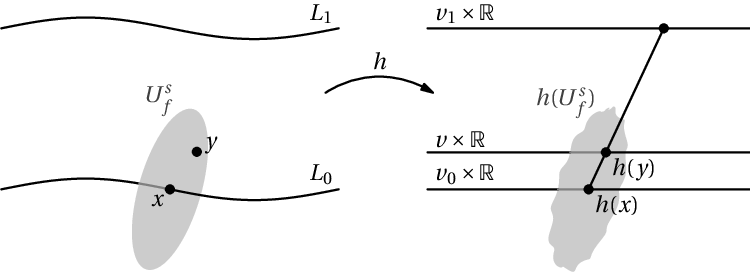}
    }
    \caption{A depiction of points and leaves occuring in the proof of
    global product structure.
    In this figure, the stable direction $\Es_f$ is shown as if it were
    two-dimensional
    and $U^s_f$ is drawn as a small plaque tangent
    to $\Es_f$.
    The entire left side of the figure lies inside a three-dimensional
    $cs$-leaf of $f$ and the right side lies inside a $cs$-leaf
    of $A \times \id$.
    }
    \label{fig-gpsproof}
    \end{figure}
    In particular, let $v$ be such that $v \in v_0 \tN_1$
    and fix such a point $y$.
    See Figure \ref{fig-gpsproof}.
    Let $\{n_k\}$ be such that $\beta^{n_k} h f^k(x) \in \tN \times [0,1)$
    for all $k$.
    Then,
    \begin{align*}
        d(f^k(x),f^k(y)) &\rightarrow 0  \quad \Rightarrow \\  
        d(\beta^{n_k} h f^k(x),\beta^{n_k} h f^k(y)) &\rightarrow 0  \quad \Rightarrow \\  
        d(B^{n_k}A^k(v_0), B^{n_k}A^k(v)) &\rightarrow 0
    \end{align*}
    which by \eqref{lemma-autseq} implies
    \begin{math}
        d(B^{n_k}A^k(v_0), B^{n_k}A^k(v_1)) \rightarrow 0.
    \end{math}

    Then, as $h f^k(L_1) = \{A^k(v_1)\} \times \bbR$,
    \begin{align*}
        \dist(\beta^{n_k}hf^k(x), \beta^{n_k}hf^k(L_1)) &\rightarrow 0  \quad \Rightarrow \\  
        \dist(hf^k(x), hf^k(L_1)) &\rightarrow 0  \quad \Rightarrow \\  
        \dist(f^k(x), f^k(L_1)) &\rightarrow 0.
    \end{align*}
    Thus, for sufficiently large $k$, $\Ws_f(f^k(x))$ intersects
    $f^k(L_1)$ showing that $\Ws_f(x)$ intersects $L_1$.
\end{proof}
A sequence $\{x_k\}$ is an $\ep$-$c$-pseudoorbit
if for each $k \in \bbZ$ the points $f(x_k)$ and $x_{k+1}$ lie $\ep$-close
on the same center leaf.
A partially hyperbolic system is \emph{plaque expansive} if there is $\ep>0$
such that if $\{x_k\}$ and $\{y_k\}$ are $\ep$-$c$-pseudoorbits
and $d(x_k,y_k)<\ep$ for all $k \in \bbZ$, then $x_0$ and $y_0$ are on the same
local center leaf.

\begin{thm} \label{thm-ABPE}
    Every AB-system is plaque expansive.
\end{thm}
Since plaque expansive systems are open in the $C^1$ topology
\cite{HPS},
this also proves \eqref{thm-openAB}.

\begin{proof}
    Let $f:M \to M$ be an AB-system.
    Let $C>1$ be a constant to be defined shortly.
    Since $f$ expands in the unstable direction,
    there is $\ep_0 > 0$ such that if points $x,y,x',y' \in M$
    satisfy
    \[    
        \frac{1}{C}  \le  d_u(x,y)  \le  C, \quad
        d_c(f(x),x') < \ep_0, \quad \text{and} \quad
        y' \in \Wc(f(y)) \cap \Wu(x')
    \]
    then $d_u(x,y)<(1-\ep_0)d_u(x',y')$.
    This result then also holds for points on the universal cover $\tM$ where
    $f$ for the remainder of the proof denotes the lift $f:\tM \to \tM$.

    Let $h:\tM \to \tN \times \bbR$ be the lifted leaf conjugacy.
    Define sets
    \begin{align*}
        X =
        \{(v,w) \in \tN \times \tN :
        \ \ v \in \Wu_A(w),
        \ \ d(v,w) \le 1\}  \end{align*}
    and
    \[        X_1 =
        \{(v,w) \in \tN \times \tN :
        \ \ v \in \Wu_A(w),
        \ \ \tfrac{1}{2}  \le  d_1(v,w)  \le  1\}.
    \]
    and a function
    \[    
        D: X \times [-1,1] \to \bbR,\quad
        (v,w,t) \mapsto d_u(h \inv(v \times \bbR),\ h \inv(w \times t)).
    \]
    That is, $D(v,w,t)$ is the distance, measured along an unstable leaf of
    $f$, between the center leaf $h \inv(v \times \bbR)$
    and the point $h \inv(w \times t)$.
    Such a function is well-defined and continuous by global product
    structure.

    If $\alpha:\tN \to \tN$ is a deck transformation for the covering $\tN \to N$,
    then $\alpha \times \id$ is a deck transformation for the covering
    $\tM_B \to M_B$ and
    one can verify that $D(\alpha(v), \alpha(w), t) = D(v,w,t)$.
    Using the compactness of $N$ and $[-1,1]$, there is $C>1$
    such that
    \[
        D(X \times [-1,1]) \subset [0, C]
        \quad \text{and}\quad
        D(X_1 \times [-1,1]) \subset [\frac{1}{C}, C].
    \]
    This defines the constant $C$ used above.

    For some $\ep>0$ let $\{x_k\}$ and $\{z_k\}$ be $\ep$-$c$-pseudoorbits
    such that $d(x_k,z_k) < \ep$.
    By increasing $\ep$ and by sliding the points $z_k$ along center leaves,
    assume, without loss of generality, that there is a point $y_k$ for
    each $k$ such that $x_k$ and $y_k$ are connected by a short unstable segment
    and $y_k$ and $z_k$ are connected by a short stable segment.
    By again increasing $\ep$, one can show that
    $\{y_k\}$ is a $\ep$-$c$-pseudoorbit.
    We may freely assume that the original $\ep$ was chosen small enough that
    $d_c(f(x_k), x_{k+1}) < \ep_0$ for all $k$.
    We will show that $x_0$ and $y_0$ lie on the same center leaf.
    An analogous argument holds for $y_0$ and $z_0$ which will complete the
    proof.

    Suppose $x_0$ and $y_0$ lie on distinct center leaves.
    Then, using $\beta$ as in the previous proof,
    there are $v_x  \ne  v_y \in \tN$ and $\{n_k\}$ such that
    $\beta^{n_k}h(x_k) \in \{B^{n_k}A^k v_x\} \times (-1,1)$
    and
    $\beta^{n_k}h(y_k) \in \{B^{n_k}A^k v_y\} \times (-1,1)$
    for all $k \in \bbZ$.
    This implies that
    \[    
        \sup_k d(B^{n_k}A^k v_x, B^{n_k}A^k v_y) < \infty.
    \]
    Let $\tN_1 \subset \tN$ be a one-dimensional subgroup such that
    $v_y \in v_x \tN_1$.
    By \eqref{lemma-autseq}, there is
    $\hat v \in v_x \tN_1$
    such that
    \[    
        \sup_{k \in \bbZ}
        d_1(B^{n_k}A^k v_x, \, B^{n_k}A^k \hat v) = 1.
    \]
    By the global product structure of $f$,
    there is a unique sequence
    $\{\hat y_k\}$ in $\tM$ such that
    $h(\hat y_k) \in \{A^k \hat v\} \times \bbR$ and
    $\hat y_k \in \Wu_f(x_k)$.
    Then,
    \begin{math}
        S = \sup_{k \in \bbZ} d_u(x_k, \hat y_k)
    \end{math}
    satisfies $\frac{1}{C}  \le  S  \le  C$.
    Let $k \in \bbZ$ be such that
    \begin{math}
        d_u(x_k, \hat y_k) > (1-\ep_0) S.
    \end{math}
    The definition of $\ep_0$ implies that
    \begin{math}
        d_u(x_{k+1}, \hat y_{k+1}) > S,
    \end{math}
    a contradiction.
\end{proof}
\section{The dynamically-incoherent example} \label{sec-incoherent} 

This section gives a construction of the example due to Hertz, Hertz, and Ures
of a partially hyperbolic system on the 3-torus having an invariant $cs$-torus
\cite{RHRHU-non-dyn}.
For this specific construction,
$\Eu$ and $\Es$ are jointly integrable and
the tangent foliation has exactly one
compact leaf.
The system therefore gives an example of case (3) of \eqref{thm-ratAB}.
We use the following to prove the example is partially hyperbolic.

\begin{prop} \label{prop-phNW}
    Suppose $f$ is a diffeomorphism of a compact manifold $M$,
    $TM = \Es \oplus \Ec \oplus \Eu$ is an invariant splitting, and
    there is $k>1$ such that
    \[
        \|Tf^k v^s_x\| < \|Tf^k v^c_x\| < \|Tf^k v^u_x\|
        \qandq
        \|Tf^k v^s_x\| < 1 < \|Tf^k v^u_x\|
    \]
    for all $x \in NW(f)$ and unit vectors $v_x^* \in E^*_x$
    ($*=s,c,u$).
    Then, $f$ is partially hyperbolic.
\end{prop}
To prove this, note
that if the above inequalities hold on $NW(f)$,
they also hold on a neighbourhood $U$ of $NW(f)$ and any orbit of $f$ has
a uniformly bounded number of points which lie outside of $U$.
The details are left to the reader.

\medskip

Now, we return to constructing the example on $\bbT^3$.
The example has a linear stable bundle, so we first consider dynamics in
dimension two.
Define $\lam = \frac{1}{2}(1 + \sqrt{5})$
and functions
\[
    \psi:\bbR \to \bbR, \, x \mapsto x + \tfrac{2}{3}\sin x \qandq
    g:\bbR^2 \to \bbR^2, \, (x,y) \mapsto (\psi(x), \lam y + \cos x).
\]
The derivative of $g$ is
\[
    Dg =
        \begin{pmatrix}
        \psi'(x) & 0 \\ - \sin x & \lam  \end{pmatrix}
    .
\]
On the vertical line $x=0$, there is an expanding fixed point for $g$.
Through this point is an invariant
one-dimensional unstable manifold associated to the larger eigenvalue
of $Dg$.
One can show that this unstable manifold may be expressed as the graph of a
function $u:(-\pi,\pi) \to \bbR$.
For now,
only consider $u$ on $[0,\pi)$.
By an invariant cone argument, one can show that $u'(x) < 0$
for all $x \in (0,\pi)$.
Using that $\psi'(x)<\lam$ when $x$ is close to $\pi$
and that
\[
    \frac{|\lam t - \sin x|}{|\psi'(x)|} >
    \frac{\lam}{|\psi'(x)|} |t| > |t|,
\]
for $t < 0$,
one can show that $\lim_{x \nearrow \pi} u'(x) = -\infty$.

Define a foliation $\Wu$ on $[0,\pi) \times \bbR$
by all graphs of functions of the form
$x \mapsto u(x) + b$
for $b \in \bbR$.
This foliation is $g$-invariant.
Reflecting about the $y$-axis,
extend this to a foliation on $(-\pi,\pi) \times \bbR$.
By including the vertical lines on the boundary, extend this foliation to
$[-\pi,\pi] \times \bbR$ and then, by $2 \pi$-periodicity in $x$, to all 
of $\bbR^2$.
Call this foliation $\Wu$ and let $\Eu$ be the $C^0$ line field tangent to
it.

Now consider the hyperbolic fixed point of $g$ on the line $x=\pi$.
Part of the stable manifold of this point is given by the graph
of a function
$c:(0,\pi] \to \bbR$.
One can show that $c'(x) > 0$ for all $x \in (0,\pi)$
and, since $\psi'(0) > \lam$, that $\lim_{x \searrow 0} c'(x) = +\infty$.
From the definition of $g$,
there is a constant $C>1$ such that $g \inv$ maps the region
$[-C,C] \times [0,\pi]$ into itself.
The stable manifold given by $\graph(c)$ must therefore
be contained in this region,
showing that $c$ is a bounded function
and can be continuously extended to all
of $[0,\pi]$.
By reflection and periodicity,
further extend $c$ to a continuous function $\bbR \to \bbR$
which is differentiable except at $2 \pi \bbZ$
and such that $g(\graph(c))=\graph(c)$.
By considering translates, $x \mapsto c(x)+b$,
define a foliation $\Wc$ on $\bbR^2$
and let $\Ec$ be the unique continuous line field on $\bbR^2$
which is tangent to $\Wc$
on $(\bbR \setminus 2 \pi \bbZ) \times \bbR$.
As $u' < 0 < c'$ on $(0,\pi)$,
$\Eu$ and $\Ec$ are transverse.

The matrix
\[
    \left(
    \begin{array}{cc}
        {1} & {1} \\
        {1} & {0}
    \end{array}
    \right)
\]
has eigenvalues $\lam = \frac{1}{2}(1 + \sqrt{5})$ and $-\lam \inv$.
Therefore, there is a lattice $\Lam \subset \bbZ^2$
such that $(y,z) \mapsto (\lam y,\,-\lam \inv z)$
quotients to an Anosov diffeomorphism on the 2-torus $\bbR^2/\Lam$.
Define $f:\bbR^3 \to \bbR^3$
by
\[
    f(x,y,z) =
    (x + \tfrac{2}{3}\sin x,\,
    \lam y + \cos x,\,
    -\lam \inv z)
\]
and a splitting $\Ec \oplus \Eu \oplus \Es$
by $\Es = \frac{\partial}{\partial z}$
and where $\Ec \oplus \Eu$ on each $x y$-plane is given by the earlier
splitting constructed for $g$.
This splitting is $f$-invariant
and there is a foliation tangent to $\Eu \oplus \Es$.
Define $M = (\bbR \times \bbR^2) / (2 \pi \bbZ \times \Lam)$.
Both $f$ and the splitting descend to $M$.
Here, $NW(f) \subset M$ consists of two tori,
one tangent to $\Ec \oplus \Es$ and
the other tangent to $\Eu \oplus \Es$.
Using \eqref{prop-phNW}, one can verify that $f$ is partially hyperbolic.
It has a foliation tangent to $\Eu \oplus \Es$ with one compact leaf and all
other leaves are planes.

This is not an example of an AB-system as there is no invariant foliation
tangent to $\Ec$.
In the above analysis, the crucial properties needed for the 
term $\cos x$ in the formula $\lam y + \cos x$ for the second coordinate
of $g$ were that $\cos' < 0$ on $(0,\pi)$ and
$\cos'(\pi)  \le  0 = \cos'(0)$.
Therefore, replace $\lam y + \cos x$
by $\lam y + \sin x - x$ in all of the above analysis.
As $\sin x - x$ is an odd function, the resulting function
$c:\bbR \to \bbR$ is odd and its graph is a $C^1$ submanifold in $\bbR^2$.
Defining
$f:\bbR^3 \to \bbR^3$
now by
\[
    f(x,y,z) =
    (x + \tfrac{2}{3}\sin x,\,
    \lam y + \sin x - x,\,
    -\lam \inv z)
\]
and quotienting by the lattice in $\bbR^3$ generated by
$\{0\} \times \Lam$ and $(2 \pi, \tfrac{2 \pi}{\lam-1}, 0)$
one constructs a skew product on $\bbT^3$
having a foliation tangent to $\Eu \oplus \Es$
with exactly one compact leaf.

\appendix \section{Definitions} \label{sec-define} 

This appendix defines a number of notions in smooth dynamical theory.

All manifolds considered in this paper are Riemannian manifolds without
boundary.
Suppose $f$ is a $C^1$ diffeomorphism on a compact manifold
and there is
a $Tf$-invariant splitting $TM = \Eu \oplus \Ec \oplus \Es$ of the tangent
bundle and $k \ge 1$ such that
\begin{math}
        \|Tf^k v^s\| < 1 < \|Tf^k v^u\|
\end{math}
for all unit vectors $v^s \in \Es$ and $v^u \in \Eu$.
If $\Ec$ is the zero bundle, then $f$ is an \emph{Anosov} diffeomorphism.
If $\Eu$, $\Ec$, and $\Es$ are all non-zero
and
\begin{math}
        \|Tf^k v^s\| < \|Tf^k v^c\| < \|Tf^k v^u\|
\end{math}
for all $p \in M$ and unit vectors
$v^s \in \Es_p$, $v^c \in \Ec_p$, and $v^u \in \Eu_p$
then $f$ is a
\emph{partially hyperbolic} diffeomorphism.
The notion of partially hyperbolicity is also extended 
to certain non-compact manifolds in Section \ref{sec-AIsys}.

A $C^1$ flow is an \emph{Anosov} flow if its time-one map is a partially
hyperbolic diffeomorphism with a center bundle given by the direction of the
flow.

A partially hyperbolic diffeomorphism $f$ is \emph{dynamically coherent}
if there are invariant foliations $\Wcu$ and $\Wcs$ tangent to
$\Ec \oplus \Eu$ and $\Ec \oplus \Es$.  As a consequence,
there is also an
invariant center foliation $\Wc$ tangent to $\Ec$.
Global product structure is defined in Section \ref{sec-outline}.

For homeomorphisms $f:X \to X$ and $g:Y \to Y$,
a \emph{topological semiconjugacy}
is a continuous surjection $h:X \to Y$
such that $hf = gh$.
If $h$ is a homeomorphism,
it is a \emph{topological conjugacy}.

Partially hyperbolic diffeomorphisms $f$ and $g$ are \emph{leaf conjugate}
if they are dynamically coherent and there is a homeomorphism $h$
such that for every center leaf $L$ of $f$, $h(L)$ is a center leaf of $g$
and $hf(L) = gh(L)$.

A homeomorphism $f:M \to M$ is \emph{(topologically) transitive}
if every non-empty open $f$-invariant subset of $M$ is dense in $M$.

For a homeomorphism $f:M \to M$, a Borel measure $\mu$ is \emph{invariant}
if $\mu(X) = \mu(f(X))$ for every measurable set $X \subset M$.
The pair $(f,\mu)$ is \emph{ergodic} if $\mu$ is $f$-invariant
and either $\mu(X)=0$ or $\mu(X) = 1$
for every $f$-invariant measurable $X \subset M$.
We often write that $f$ is ergodic or $\mu$ is ergodic if the context is clear.
For brevity, we sometimes say that a system $f$ with a finite non-probability 
measure $\mu$ is ergodic when,
to be precise, we should actually say that the pair $(f, \frac{1}{\mu(M)} \mu)$
is ergodic.
A homeomorphism $f$ is \emph{conservative} if it has an invariant measure given
by a smooth volume form on $M$.
A conservative $C^2$ diffeomorphism is \emph{stably ergodic} if it has
a neighbourhood $\mathcal{U}$ in the $C^1$ topology of $C^1$ diffeomorphisms
such that
every conservative $C^2$ diffeomorphism in $\mathcal{U}$ is also ergodic.
For a discussion of why the quirky combination of $C^1$ and $C^2$ regularity is
necessary, see \cite{wilkinson2010conservative}.

If $\tN$ is a simply connected nilpotent Lie group and $\Gamma$ is a discrete
subgroup such that $N := \tN / \Gamma$ is a compact manifold,
then $N$ is called
a (compact) \emph{nilmanifold} \cite{malcev}.
If $\tilde A:\tN \to \tN$ is a Lie group automorphism which descends to
$A:N \to N$, then $A$ is a \emph{nilmanifold automorphism}
(also called a toral automorphism when $N = \bbT^d$).
If $A$ is Anosov, it is called \emph{hyperbolic}.
Infranilmanifolds and their automorphisms are defined in
Section \ref{sec-infra}.

If $f:M \to N$ is a continuous function
and $\pi_M:\hat M \to M$ and $\pi_N:\hat N \to N$ are covering maps,
then a \emph{lift} of $f$ is a function $\hat f:\hat M \to \hat N$ such that
$\pi_N \hat f = f \pi_M$.
Note that if $\pi_M$ and $\pi_N$ are universal covering maps,
then at least one such lift exists,
but is not unique in general.



\section{Erratum} \label{sec:erratum}

This erratum addresses two issues
with the proofs in the paper. 
The first issue is that proposition (6.4) as stated
is not correct.\footnote{
Note that the numbering of sections in some preprint versions
may differ from the published version.}
For instance, the automorphism
\begin{math}
    \bbZ^2 \to \bbZ^2, \ (x,y) \mapsto (5x + 2y, 2x + y)
\end{math}
gives a counterexample
as it fixes a coset of $\bbZ \ti 2 \bbZ.$
%
The flaw in the proof is that it confuses invertibility
in $GL(n,\bbZ)$ with invertibility in $GL(n, \bbR)$
and the notions are not equivalent.
In fact, the proposition holds in the following revised version.

\begin{prop} \label{prop:indexcoset}
    Let $G$ be a torsion-free, finitely-generated, nilpotent group
    and suppose
    $\phi \in \Aut(G)$ is such that $\phi(g) \ne g$ for all non-trivial
    $g \in G$.
    If $H$ is a normal, $\phi$-invariant subgroup,
    then $\phi$ fixes at most \textbf{finitely many} cosets of $H$.
\end{prop}
We prove this revised version below.
The original proposition (6.4)
is used in only two places
in the proofs of theorem (4.3) and lemma (6.5) and we show
below how to use the revised version of the proposition
to recover the proofs of those two results.

\medskip{}

The other issue to address in the original paper
comes at the start of section 8 which deals with AB-systems.
That section states that $h f h \inv$ is homotopic to $f_{A B}$
and uses this to lift $h f h \inv$ to a map on $N \ti \bbR.$
In fact, the two functions are not homotopic in general.
For instance, for the linear partially hyperbolic maps
on the 3-torus
$\bbT^3 = \bbR^3 / \bbZ^3$ given by the matrices
\[
    \begin{pmatrix}
        5&2&0 \\
        2&1&0 \\
        0&0&1 \end{pmatrix}
    \qandq
    \begin{pmatrix}
        5&2&0 \\
        2&1&0 \\
        1&0&1
    \end{pmatrix}  \]
both have vertical center foliations and the identity
map is a leaf conjugacy between the two systems.
The two systems are not homotopic to each other
and attempting to lift the two systems to AI-system
on $\bbT^2 \ti \bbR$ as in section 8 will not work.
To fix this, we amend the definition of an AB-system to add the homotopy
as an assumption.
That is, a partially hyperbolic diffeomorphism $f$ is an \emph{AB-system} if
\begin{enumerate}
    \item it preserves the orientation of the center bundle $\Ec,$
    \item
    there is a leaf conjugacy $h$ between $f$
    and an AB-prototype $\fAB,$ and
    \item
    $h f h \inv$ is homotopic to $\fAB.$
\end{enumerate}
This additional assumption can always be achieved
by lifting $f$ and $f_{AB}$ to finite covers:

\begin{prop} \label{prop:liftAB}
    If a partially hyperbolic diffeomorphism $f$
    satisfies conditions (1) and (2) above,
    then a lift of $f$ to a finite cover satisfies
    all of (1), (2), and (3).
\end{prop}
The proof of this is given in the final section of this erratum.

For those readers interested only in the case where
the nilmanifold $N$ is a torus $\bbT^d,$
we have structured the proofs below so that
most of the details specific to the non-toral
case may be skipped over.

\smallskip{}

\noindent \textbf{Erratum acknowledgements.}\ \ 
The author wishes to thank
Danijela Damjanovic, Amie Wilkinson, and Disheng Xu
for bringing these issues to his attention
and for helpful input.
He also thanks
Jonathan Bowden,
Davide Ravotti,
Heiko Deitrich,
and
Santiago Barrera Acevedo
for helpful conversations.

Karel Dekimpe was also very helpful and suggested an alternative
method to establish \cref{prop:indexcoset}.
Instead of proving the proposition directly,
one can instead show the following fact,
from which the proposition follows as a corollary:
\begin{quote}
    Let $G$ be a finitely generated nilpotent torsion free nilpotent group and $\varphi\in{\rm \Aut}(G)$ be fixed point free. Assume that $H$ is a $\varphi$ invariant subgroup of $G$ such that $G/H$ is torsion free. Then it follows that the induced automorphism on $G/H$ is also fixed point free.
\end{quote}
\section*{Proof of \cref{prop:indexcoset}} 

This section gives a proof of \cref{prop:indexcoset}.
We first prove this in the abelian case
and then use induction on the nilpotency class
to handle the non-abelian case.

\begin{lemma} \label{lemma:zcoset}
    Let $G$ be isomorphic to $\bbZd$
    and suppose
    $\phi \in \Aut(G)$ is such that $\phi(g) \ne g$ for all non-trivial
    $g \in G$.
    If $H$ is a normal, $\phi$-invariant subgroup,
    then $\phi$ fixes at most finitely many cosets of $H$.
\end{lemma}
\begin{proof}
    Assume $G = \bbZd$ and define
    a linear map $A : \bbQd \to \bbQd$
    such that $A z = \phi(z)$ for all $z \in \bbZd.$
    If $A$ had an eigenvalue of 1,
    the corresponding eigenspace would intersect $\bbZd$
    in a non-trivial fixed point $\phi(z) = z \in \bbZd.$
    Hence, 1 is not an eigenvalue of $A$.

    Let $V \subof \bbQd$ be the set of all $\bbQ$-linear
    combinations of elements of $H.$
    We may assume $H$ has infinite index in $\bbZd,$
    and so
    $V$ is a proper $A$-invariant subspace of $\bbQd.$
    It induces a linear map $\bar A$ on the quotient space $\bbQd / V.$
    If $z \in \bbZd$ is such that $\phi(z + H) = z + H,$
    then $\bar A(z + V) = z + V$ and so $\bar A$ and therefore $A$ has
    an eigenvalue of 1.
\end{proof}

\begin{lemma} \label{lemma:fixedfactor}
    Let $\phi : G \to G$ be a group automorphism
    and let $X$ be a normal $\phi$-invariant subgroup.
    If $\phi|_X$ has at most finitely many fixed points
    and $\phi$ fixes at most finitely many cosets of $X$,
    then $\phi$ itself has finitely many fixed points.
\end{lemma}
\begin{proof}
    If $\phi(g) = g$ and $\phi(g') = g'$
    are fixed points in the same coset $g X = g' X,$
    then $\phi(g' g \inv) = g' g \inv$
    is a fixed point in $X.$
    Hence, each of the finitely many fixed cosets has
    finitely many fixed points.
\end{proof}
\begin{cor} \label{cor:cosetfactor}
    Suppose $\phi$ is an automorphism
    of a group $G$
    with center $Z,$
    and $H$ is a $\phi$-invariant normal subgroup of $G.$
    If the induced maps on $Z / (H \cap Z)$
    and $G / H Z$ have finitely many fixed points,
    then the induced map on $G / H$
    has finitely many fixed points.
\end{cor}
\begin{proof}
    Apply the previous lemma to the quotient
    \[
        0 \to Z / (H \cap Z) \to G / H \to G / H Z \to 0
        \qedhere
    \] \end{proof}
%
%
%
%
    
\begin{lemma} \label{lemma:centerfix}
    Suppose $G$ is a
    finitely generated torsion free nilpotent group
    and let $\phi : G \to G$ be an automorphism.
    Let $Z$ denote the center of $G.$
    If there is a non-trivial fixed coset $\phi(g Z) = g Z,$
    then $\phi$ has a non-trivial fixed point.
\end{lemma}
\begin{proof}
    By the properties of such groups \cite{dek1996book},
    $Z$ is isomorphic to $\bbZd$ and
    $G / Z$ is torsion free.
    Suppose $\phi(g Z) = g Z$ is a non-trivial fixed coset.
    Let $Y$ be the subgroup generated by $g$ and $Z.$
    Then $Y$ is isomorphic to $\bbZ^{d+1}$
    and within $Y,$ there are infinitely many
    fixed cosets: $\phi(g^k Z) = g^k Z$
    for $k \in \bbZ.$
    \Cref{lemma:zcoset} implies that $\phi|_{Y}$
    has a non-trivial fixed point.
\end{proof}
\begin{proof}
    [Proof of \cref{prop:indexcoset}]
    We prove this by induction on the length
    of the upper central series of $G.$
    The abelian base case is given by \cref{lemma:zcoset}.
    Assume now that $G$ is non-abelian with center $Z$
    and that \cref{prop:indexcoset} is already known to hold
    for the quotient map $\Phi : G / Z \to G / Z.$

    Since $\phi|_Z$ has no non-trivial fixed points,
    \cref{lemma:zcoset} implies that $\phi|_Z$
    fixes at most finitely many cosets of $H \cap Z.$
    \Cref{lemma:centerfix} implies that
    $\Phi$ has no non-trivial fixed points.
    By the inductive hypothesis,
    $\Phi$ fixes at most finitely many
    cosets of $H Z / Z.$
    Then \cref{cor:cosetfactor} implies that $\phi$ (on all of $G$)
    fixes at most finitely many cosets of $H.$
\end{proof}
\section*{Revised proof of lemma (6.5)} 

The incorrect proposition (6.4)
is used in the proof of (6.5)
only to establish $\lam \ne 1.$
Recall in that proof that there is 
$F \in \Aut(G)$ with no non-trivial fixed points
and a non-zero homomorphism $\tau : G \to \bbR$
such that $\tau F = \lam \tau.$
Define $H \subof G$ to be the kernel of $\tau.$
Note that the cosets of $H$ are exactly the level sets of $\tau.$
If $\lam = 1,$ then every level set of $\tau$ is fixed by $F.$
Since $\tau$ is non-zero, there are infinitely many such level sets
and \cref{prop:indexcoset} above gives a contradiction.

\section*{Circle bundles over nilmanifolds} 

Before revising the proof of (4.3),
we first prove the following.

\begin{prop} \label{prop:virttriv}
    Suppose $M$ is a circle bundle
    with oriented fibers over a nilmanifold $N.$
    If $M$ has a compact horizontal submanifold $\Sig,$
    then $M$ is a trivial bundle.
\end{prop}
\begin{remark}
    We consider everything in the $C^0$ setting here.
    The circle bundle is defined by a $C^0$
    map $p : M \to N$
    and a compact horizontal submanifold
    $\Sig$ is a codimension one $C^0$ submanifold
    such that 
    $p|_\Sig : \Sig \to N$
    is a covering map of finite degree.
    To show that $M$ is trivial,
    it enough to find another horizontal submanifold $\Sig_1$
    such that $p|_{\Sig_1} : \Sig_1 \to N$
    is a homeomorphism.
    To simplify the proof,
    we assume that
    the circle fibers are tangent to a $C^0$ vector field
    as is the case for the center leaves of
    a partially hyperbolic skew product.
\end{remark}

\begin{proof}
    Assume $\Sig$ intersects each fiber in exactly $k$ points.
    We may define a metric on each fiber
    such that the length of every fiber is exactly one
    and that its points of intersection with $\Sig$ are equally spaced;
    that is,
    the distance between one point of intersection
    and the next
    is exactly $\tfrac{1}{k}.$
    We may choose these metrics to vary continuously along $M.$

    Let $\pi : \tM \to M$ be the universal covering map.
    We may assume that $\tM = \tN \ti \bbR$
    where $\tN$ is the nilpotent Lie group
    covering $N$
    and such that the fibers of $M$
    lift to lines of the form $v \ti \bbR$
    with $v \in \tN.$
    We further assume that the metric on a fiber of $M$
    lifts to a metric on $v \ti \bbR$
    which is equal to the standard Euclidean metric given by $\bbR.$
    In particular,
    $\pi \inv(\Sig)$ intersects each fiber $v \ti \bbR$
    in a set of points of the form
    \[
        \{ \ (v, \sig(v) + t) \ : \ t \in \tfrac{1}{k} \bbZ \ \}
    \]
    for some $\sig(v)$ depending on $v.$
    We may assume $\sig : \tN \to \bbR$ is continuous.
    To see this,
    choose a connected component $\tSig$ of $\pi \inv(\Sig)$
    and define $\sig(v)$ to be the unique intersection
    of $v \ti \bbR$ with $\tSig.$

    Write $G = \pi_1(M),$ 
    and $H = \pi_1(N).$
    The bundle projection $p : M \to N$
    induces a surjective homomorphism $p_* : G \to N.$
    We now use $\tSig$ to define a homomorphism $\tau : G \to \frac{1}{k} \bbZ.$
    Without loss of generality,
    assume $\sig(e) = 0$ where $e$ is the identity element of $\tN.$
    For a deck transformation $g \in G,$
    let $\tau(g)$ be such that $(e, \tau(g))$
    is the unique intersection of $g(\tSig)$ with $e \ti \bbR.$
    Similar to lemma (7.6) in the original paper,
    one may show that $\tau : G \to \tfrac{1}{k} \bbZ$ is a homomorphism.
    We claim the following.
    \begin{quote}
        \textbf{Claim.}\ There is a (not necessarily unique)
        homomorphism \\ $\psi : H \to \kbbZ$
        such that
        $\psi p_*(g) - \tau(g) \in \bbZ$
        for all $g \in G.$
    \end{quote}
    We leave the proof of this to the end
    and first show that this gives the desired result.
    By the properties of nilmanifolds \cite{malcev},
    $\psi$ determines a Lie group homomorphism
    $\psi : \tN \to \bbR$
    where if we regard $H$ as a discrete subgroup of $\tN$
    then this is an extension of $\psi$ from $H$ to all of $\tN.$
    Define a submanifold $\tSig_1$
    as the graph of $\sig - \psi;$
    that is,
    $(v,t) \in \tSig_1$ if and only if $t = \sig(v) - \psi(v).$
    By the above claim,
    for all deck transformations $g \in G,$
    the intersection of $g(\tSig_1)$ with $e \ti \bbR$
    lies inside $e \ti \bbZ.$
    Hence $\tSig_1$
    quotients down to a compact horizontal submanifold $\Sig_1 \subof M$
    which intersects each fiber exactly once
    and therefore shows that the circle bundle is trivial

    \medskip{}

    It remains to prove the claim.
    We first consider the abelian case
    where $H$ is isomorphic to $\bbZ^d.$
    Let $\{ h_1, \ldots h_d \}$ be a generating set for $H$
    and choose elements $g_i \in G$ such that $p_*(g_i) = h_i.$
    Let $z \in G$ be the deck transformation
    \begin{math}
        (v, t) \mapsto (v, t+1)
    \end{math}
    corresponding to going once around a fiber of the circle fibering.
    Note that $\tau(z) = 1.$
    As explained in the original proof of (4.3),
    $\langle z \rangle$ is the kernel of $p_*,$
    and so $\{ z, g_1, \ldots g_d \}$ is a generating set for $G.$
    Define $\psi : H \to \kbbZ$ by $\psi(h_i) = \tau(g_i).$
    Then $\psi p_*(z) - \tau(z) = -1$
    and $\psi p_*(g_i) - \tau(g_i) = 0.$
    As $\psi p_* - \tau$ takes integer values on a generating set for $G,$
    it must take integer values on all of $G.$
    
    \medskip{}

    We now extend this argument to the non-abelian case.
    Note that both $M$ and $N$ are nilmanifolds.
    Consider the root set $G_1$ of the commutator subgroup of $G.$
    That is $g \in G_1$ if and only if
    $g^k \in [G,G]$ for some $k \ge 1.$
    Such sets are discussed in detail in Chapter 1 of \cite{dek1996book}
    (where the notation there is $\sqrt[G]{\gam_2(G)}$ instead of $G_1$).
    In particular, $G_1$ is a normal subgroup
    and any homomorphism from $G$ to a torsion-free abelian group $R$
    is identically zero on $G_1$ and so factors through
    \begin{math}
        G \to G / G_1 \to R
    \end{math}
    We can therefore define a homomorphism
    $\tau_1 : G / G_1 \to \frac{1}{k} \bbZ$ 
    as the quotient of $\tau.$

    Similarly
    write $H_1$ for the root set of $[H,H]$.
    Then $H / H_1$ is a torsion-free abelian group homomorphic
    to $\bbZd$ for some $d$ \cite{dek1996book},
    and $p_* : G \to H$ descends to a map $p_1 : G / G_1 \to H / H_1.$
    Adapting the argument above, we may define
    a map $\psi_1 : H / H_1 \to \kbbZ$ such that
    $\psi_1 p_1 - \tau_1$ takes integer values on all of $G / G_1.$
    Then $\psi_1$ determines a map $\psi : H \to \kbbZ$ as desired.
\end{proof}

\section*{Revised proof of theorem (4.3)} 

This section revises the proof of theorem (4.3)
to use \cref{prop:indexcoset} above
in place of the incorrect proposition (6.4)
of the original paper.
By virtue of \cref{prop:virttriv} above,
we need only show that the partially hyperbolic
system has a compact us-leaf.

\medskip{}

The proof of (4.3) is unchanged up to
the definition of $\hat \tau : G \to \bbR / \bbZ$
and the first use of (6.4).
Using instead \cref{prop:indexcoset} above,
the most we can say is
that $\hat \tau$ has a finite image.
In other words, there is an integer $k \ge 1$
such that $\tau(G) = \tfrac{1}{k} \bbZ.$

The existence of $\tau$ is given by
(6.1) and (6.2).
From the proofs of those results,
we can see that
then there is a measure $\mu$ on $\tilde S$
invariant under the action of $G$
and such that $\tau(g) = \mu[x, g(x))$
for any $x \in \Lam$ and $g \in G.$
Here, $\Lam$ is the intersection of the non-open accessibility classes
$\Gam$ with $\tilde S$.
Choose some point $x_0 \in \Lam$
and for each $t \in \tfrac{1}{k} \bbZ,$
define a set $X_t \subof \Lam$
by 
\[
    X_t = \{ x \in \Lam \ : \ \mu[x_0, x) = t \}.
\]
The sets $X_t$
are disjoint and the action of $g \in G$
on $\Lam$ takes $X_t$ to $X_{t + \tau(g)}.$
Define $y_t = \sup X_t$ where we are identifying
$\tilde S$ with $\bbR$ in order to define the supremum.
Then
\begin{math}
    \{ y_t : t \in \tfrac{1}{k} \bbZ \}
\end{math}
is a discrete subset of $\Lam$ which is
invariant under the action of $G.$
This implies that for any point $y_t,$
its accessibility class $AC(y_t) \subof \tM$
quotients down to a compact us-leaf on $M.$

\section*{Proof of \cref{prop:liftAB}} 

We now prove \cref{prop:liftAB}.
Assume $f : M \to M$ is a partially hyperbolic diffeomorphism
which preserves the orientation of $\Ec$
and $h : M \to M_B$ is a leaf conjugacy to $\fAB : M_B \to M_B.$
We want to show that after lifting $f$ and $\fAB$ to maps $\hat f$
and $f_{\hat A \hat B}$ on finite covers $\hat M$ of $M$
and $M_{\hat B}$ of $M_B$
that there is a leaf conjugacy $\hat h : \hat M \to M_{\hat B}$
such that $\hat h \hat f \hat h \inv$ and $f_{\hat A \hat B}$ are homotopic.
As all of the manifolds involved are
Eilenberg-MacLane spaces of type $K(\pi, 1),$
the existence of such a homotopy
is purely a question involving the actions of the functions
on the fundamental groups of the manifolds.
(See, for instance, Propositions 1.33 and 1B.9 of \cite{hat2002algebraic}.)
In particular, we do not need to use the smoothness of $f$ in any way.
Therefore,
we may replace $f$ by $h f h \inv$
and assume without loss of generality
that $M = M_B$ and that $h$ is the identity map.

\smallskip{}

In this section, we
write $G$ for the simply connected nilpotent Lie group and
$\Gam$ for the cocompact lattice such that $N = \Gam \backslash G$
is the nilmanifold.
Then quotienting $G$ by $[G, G]$ yields an abelian Lie group
isomorphic to $\bbRd$ for some $d.$
This defines a projection from $G$ to $\bbRd,$
and for $x \in G,$ we write $\bar x \in \bbRd$ for its image
under the projection.
This projection may further be chosen such that
$\Gam$ is mapped to $\bbZd.$
(If the nilmanifold is a torus $N = \bbZd \backslash \bbRd$,
then the projection $G \to \bbRd$ is the identity map
and all of the overlines in what follows may be safely
ignored.)

Let A, $B : G \to G$ be the commuting Lie group automorphisms
defining the AB-prototype.
These induce linear automorphisms $\bar A$ and $\bar B$ on $\bbRd$
with the property that $\overline{A(x)} = \bar A(\bar x)$
and $\overline{B(x)} = \bar B(\bar x)$.

The universal cover of $M_B$ is $G \ti \bbR.$
Define $\bt(x,t) = (B(x), t - 1).$
For $\gam \in \Gam,$ define $\tau_\gam(x,t) = (\gam \cdot x, t).$
Note that $\bt \tau_\gam = \tau_{B(\gam)} \bt$
and that every deck transformation may be written in the form
$\tau_\gam \bt^n$ for $\gam \in \Gam$ and $n \in \bbZ.$

Lift $f$ to a diffeomorphism $\tf : G \ti \bbR \to G \ti \bbR$
such that $\tf(0 \ti \bbR) = 0 \ti \bbR$
where 0 is the identity element of $G.$
Such a lift exists because of the leaf conjugacy.
This lift then
determines an automorphism $f_*$ of the fundamental group $\pi_1(M_B)$
defined by the property
$f_*(\tau) \circ \tf = \tf \circ \tau$
for all deck transformations $\tau.$
Since $0 \ti \bbR$ projects to an $f$-invariant circle in $M_B,$
one can show that $f_*(\bt) = \bt.$
By the leaf conjugacy,
$\tf(x \ti \bbR) = A(x) \ti \bbR$
for all $x \in G,$
and so for each $\gam \in \Gam,$
there is an integer $L(\gam)$ such that
$f_*(\tau_\gam) = \tau_{A(\gam)} \bt^{L(\gam)}.$
Using that $f_*$ is a group homomorphism,
one can show that
$L(\gam_1 \cdot \gam_2) = L(\gam_1) + L(\gam_2)$
and
$A(\gam_1 \cdot \gam_2) = A(\gam_1) B^{L(\gam_1)}$ A($\gam_2$)
for all $\gam_1, \gam_2 \in \Gam.$
This implies that $L : \Gam \to \bbZ$ is a group homomorphism
and that there is $k \ge 0$ such that $L(\Gam) = k \bbZ$ and
$B^k$ is the identity map on $G.$
If $k = 0,$ then $f$ induces the same action on $\pi_1(M_B)$
as the AB-prototype $\fAB$
and
this would imply the desired result.
Therefore, we assume in what follows that $k \ge 1.$

By the properties of nilmanifolds \cite{malcev},
$L$ extends to a Lie group homomorphism $L : G \to \bbR.$
Since $\bbR$ is abelian,
$L|_{[G,G]} \equiv 0$ and 
there is a linear map $\bar L : \bbRd \to \bbR$
such that $\bar L(\bar x) = L(x)$ for all $x \in G.$
Let $I$ denote the identity map on $\bbRd.$
As $\bar A$ is hyperbolic,
$\bar A - I$ is invertible.
Define $\bar S : \bbRd \to \bbR$
by $\bar S = \bar L (\bar A - I) \inv$
and $S : G \to \bbR$
by $S(x) = \bar S(\bar x).$
By Cramer's rule, $\bar S(m \bbZd) \subof k \bbZ$
where $m = \det(\bar A - I).$
Using $f_*(\bt \tau_\gam) = f_*(\tau_{B(\gam)} \bt),$ one can show
$L B(\gam) = L(\gam)$ for all $\gam \in \Gam.$
Hence,
$L B = L$ as functions on $G$ and one may use this to show
$\bar L \bar B = \bar L, \bar S \bar B = \bar S,$ and $S B = S.$

\smallskip{}

Define $\Gam_0 \subof \Gam$ by $\gam \in \Gam_0$ if and only if
$\bar \gam \in m \bbZd.$
Since $\bar A(m \bbZd) = m \bbZd$ and $\bar B(m \bbZd) = m \bbZd$,
it follows that $A(\Gam_0) = \Gam_0$ and $B(\Gam_0) = \Gam_0.$
Hence, $A$ and $B$ define commuting automorphisms $\hat A$ and $\hat B$ of a nilmanifold
$\hat N = \Gam_0 \backslash G$ that finitely covers $N.$
Using this we define a new AB-prototype $f_{\hat A \hat B}$
on a new suspension manifold $M_{\hat B}$ which finitely covers the original.
Further, $\tf$ quotients to a function $\hat f : M_{\hat B} \to M_{\hat B}$
which is a lift of the original $f.$

Define $\tilh : G \ti \bbR \to G \ti \bbR$ by
$\tilh(x, t) = (x, t + S(x)).$
If $\gam \in \Gam_0,$ then
$S(\gam) \in k \bbZ$ and since $B^k$ is the identity,
it follows that
\begin{math}
    \bt^{S(\gam)}(x, t) = (x, t - S(\gam))
\end{math}
which may be used to show that $\tilh \tau_\gam = \tau_\gam \bt^{-S(\gam)} \tilh.$
This implies that $\tilh$ quotients to a diffeomorphism
$\hat h$ on $M_{\hat B}$ and that induced action on $\pi_1(M_{\hat B})$
satisfies
$\hat h_*(\bt) = \bt$ and $\hat h_*(\tau_\gam) = \tau_\gam \bt^{-S(\gam)}$
for all $\gam \in \Gam_0.$
Further note that $\hat h$ is a leaf conjugacy between $\hat f$ and $f_{\hat A \hat B}.$
Since
\[
    \hat h_* \hat f_* \hat h_* \inv (\tau_\gam)
    = 
    \hat h_*     f_* \hat h_* \inv (\tau_\gam)
    =
    \tau_{A(\gam)} \bt^{-S A(\gam)} \bt^{L(\gam)} \bt^{S(\gam)}
    =
    \tau_{A(\gam)},
\]
it follows that  $\hat h \hat f \hat h \inv$ and $f_{\hat A \hat B}$
have the same action on $\pi_1(M_{\hat B})$
and so are homotopic.

\bigskip{}

\bigskip

\acknowledgement
The author would like to thank Alexander Fish,
Rafael Potrie, Federico Rodriguez Hertz,
Jana Rodriguez Hertz, Ra\'ul Ures, and Amie Wilkinson for helpful comments.
This research was partially funded by
the Australian Research Council. 


\bibliographystyle{plain}
\bibliographystyle{alpha}
\bibliography{dynamics}

\end{document}